\documentclass[a4paper,12pt,leqno]{amsart}
\usepackage{hyperref}
\usepackage{enumerate}
\usepackage{amsmath}
\usepackage{amsthm}
\usepackage{amssymb}
\theoremstyle{plain}
\newtheorem{thm}{\indent\bf Theorem}[section]
\newtheorem{lem}[thm]{\indent\bf Lemma}
\newtheorem{prop}[thm]{\indent\bf Proposition}
\newtheorem{cor}[thm]{\indent\bf Corollary}
\theoremstyle{definition}
\newtheorem{rem}{\indent\it Remark}[section]
\newtheorem{exa}{\indent\it Example}[section]
\numberwithin{equation}{section}
\numberwithin{figure}{section}

\def \z {\mathfrak{z}}
\def \re {\mathrm{Re\,}}
\def \im {\mathrm{Im\,}}
\def \diag {\mathrm{diag\,}}

\begin{document}
\title{Elliptic asymptotics for 
the complete third Painlev\'e transcendents} 
\author{Shun Shimomura} 
\date{}
\markboth{Shun Shimomura}{Third Pailev\'e transcendents}

\begin{abstract}
For a general solution of the third Painlev\'e equation of complete type 
we show the Boutroux ansatz near the point at infinity. 
It admits an asymptotic
representation in terms of the Jacobi sn-function in cheese-like
strips along generic directions. The expression is derived by using
isomonodromy deformation of a linear system governed by the 
third Painlev\'e equation of this type. In our calculation of the WKB analysis,
the treated Stokes curve ranges on both upper and lower sheets of the two
sheeted Riemann surface.
\end{abstract}
\maketitle
\allowdisplaybreaks
\section{Introduction}\label{sc1}
As in the study of the spaces of initial values
for Painlev\'e equations by Sakai \cite{Sa}, the third
Painlev\'e equations are classified into three types $P_{\mathrm{III}}(D_6)$,
$P_{\mathrm{III}}(D_7)$ and $P_{\mathrm{III}}(D_8)$. 
For $P_{\mathrm{III}}(D_7)$ and $P_{\mathrm{III}}(D_8)$ 
Ohyama et al.~\cite{Oh} examined basic properties including $\tau$-functions,
irreducibility, the spaces of initial values. 
Equation $P_{\mathrm{III}}(D_7)$ is the {\it degenerate} third Painlev\'e
equation written in the form 
\begin{equation*}
u_{\tau\tau}=\frac{u_{\tau}^2}{u} -\frac{u_{\tau}}{\tau} +\frac 1{\tau}
(-8\epsilon u^2 +2ab) +\frac{b^2}{u}
\end{equation*}
with $\epsilon=\pm 1,$ $a\in\mathbb{C},$ $b\in \mathbb{R}\setminus\{0\}$.
For this equation, using isomonodromy deformation Kitaev and Vartanian 
\cite{KV1}, \cite{KV2} obtained asymptotic solutions  
as $\tau \to \pm\infty$, $\pm i \infty$ and $\tau \to \pm 0$, $\pm i 0$ 
with connection formulas among them. Furthermore, a special
meromorphic solution is studied by \cite{Kitaev-4}, \cite{KV3}, and
trans-series solutions with related monodromy data are discussed 
by \cite{Vartanian}. 
\par
The third Painlev\'e equation $P_{\mathrm{III}}(D_6)$ of the form
\begin{equation}\label{1.1}
\frac{d^2y}{dx^2}=\frac 1y\Bigl(\frac{dy}{dx}\Bigr)^{\! 2}
 -\frac 1x \frac{dy}{dx}
+ \frac 4x (\theta_0 y^2 +1 -\theta_{\infty}) +4y^3 -\frac 4y
\end{equation}
with $\theta_0, $ $\theta_1 \in \mathbb{C}$ is equivalent to the degenerate
fifth Painlev\'e equation \cite{Gromak1}, \cite{Gromak2}, and is called the 
{\it complete} third Painlev\'e equation $P_{\mathrm{III}}$ by 
Kitaev \cite{K-Sb}, in which the asymptotics of solutions on the positive
real axis are studied. Basic properties of $P_{\mathrm{III}}(D_6)$, say 
irreducibility, $\tau$-functions, special solutions are
studied in \cite{Murata}, \cite{Umemura}, \cite{Okamoto}. 
Equation $P_{\mathrm{III}}$ of the Sine-Gordon type is obtained from 
\eqref{1.1} with $\theta_0=1-\theta_{\infty} =0$ by $y=e^{iu/2}, 4ix=\xi,$ 
and its real-valued solutions with connection formulas are found in \cite
[Chaps.~8, 10]{IN}, \cite[Chaps.~14, 15]{FIKN}.  
\par
The linear approximation of Painlev\'e equations enables us to know only 
restricted solutions, say, exponentially decaying ones, or ones along special 
directions.
To capture the global picture of a general solution as a meromorphic
function it is necessary to consider nonlinear approximation such as by an
elliptic function, which is expected to be available also in generic 
directions. Indeed, for a general solution of $P_{\mathrm{I}}$,  
Boutroux \cite{Boutroux}, \cite{Boutroux-2} first studied
asymptotic behaviours by comparing with the Weierstrass $\wp$-function, 
the modulus of which is determined 
by transcendental equations being now called the Boutroux equations. 
Rigorous treatments of these asymptotics for
$P_{\mathrm{I}}$ and $P_{\mathrm{II}}$ are made by multiscale expansions
\cite{Joshi} or by isomonodromy
techniques \cite{Kapaev-1}, \cite{Kitaev-2}, \cite{Kitaev-3}, 
\cite{Novokshenov-1}, \cite{Novokshenov-2}. Thus for Painlev\'e transcendents 
the asymptotic expression in terms of an elliptic function is called 
the Boutroux ansatz. 
Elliptic asymptotics for $P_{\mathrm{III}}$ of the Sine-Gordon type are
expressed by the Jacobi $\mathrm{sn}$-function
\cite{Novokshenov-3}, \cite{Novokshenov-4}, \cite[Chap.~16]{FIKN}, and those 
for $P_{\mathrm{III}}(D_7)$ by the Weierstrass $\wp$-function \cite{Shun}.  
\par
In this paper we show the Boutroux ansatz for the complete 
$P_{\mathrm{III}}$ \eqref{1.1} with arbitrary parameters 
$\theta_0,$ $\theta_{\infty}\in \mathbb{C},$ that is, 
present an asymptotic representation of a general 
solution in terms of the Jacobi sn-function along 
generic directions near the point at infinity. 
The main results are described in Section \ref{sc2}.
The complete third Painlev\'e equation \eqref{1.1} governs isomonodromy 
deformation of the linear system
\begin{align}\label{1.2}
\frac{dU}{d\lambda}& = \mathcal{A}(x,\lambda) U,   \phantom{----------}
\\
\notag
\mathcal{A}  (x,\lambda) =& \frac{ix}2 \sigma_3 
+ \begin{pmatrix}  -\frac 12 \theta_{\infty}  & -y\z v \\ 
\frac 12  \theta_0 x(\z v)^{-1} -v^{-1}(\frac 12 \theta_{\infty} (2-
x{\z}^{-1} ) +y(\z -x))  & \frac 12\theta_{\infty}    
\end{pmatrix} {\lambda^{-1}}
\\
\notag
& -i \begin{pmatrix} \z- \frac 12 x & - \z v \\ v^{-1}(\z-x) &  -\z +\frac 12 x 
\end{pmatrix} 
{\lambda^{-2}}
\end{align}
\cite[pp.~195--198]{FIKN}, \cite{JM}, \cite{K-Sb},
that is, the monodromy data of system \eqref{1.2} remain invariant under a 
small change of $x$ if and only if $(y,\z,v)$ fulfils the system of equations
\begin{align*}
\notag
&x \frac{dy}{dx}= 4 y^2\z -2x y^2 +(2\theta_{\infty}-1)y +2x,
\\
\notag
&x \frac{d\z}{dx}= -4 y\z^2 +(4x y +1-2\theta_{\infty})\z +(\theta_0+\theta
_{\infty})x,
\\
\notag
& x\frac d{dx}\log v = -(\theta_0+\theta_{\infty})\frac x{\z} -2xy +\theta
_{\infty},
\end{align*}
which is equivalent to \eqref{1.1}.
As shown in Section \ref{sc3} system \eqref{1.2} is a result of transformation
of system \eqref{3.1}. 
The isomonodromy deformation of system \eqref{3.1} is governed by equation
\eqref{1.1}, and solutions of \eqref{1.1} correspond to the monodromy data 
on the monodromy manifold for \eqref{3.1}, which is 
defined by Stokes matrices $S^{\infty}_j$, $S^0_j$ and 
connection matrices $G=(g_{ij})$, $\hat{G}=(\hat{g}_{ij})\in SL_2(\mathbb{C})$ 
for solutions around $\mu=0$ and $\mu=\infty.$
System \eqref{1.2} admits the same monodromy 
manifold as of \eqref{3.1}, which is described by the same matrices $S^{\infty}
_j,$ $S^0_j$ and $G, \hat{G}$ for suitably chosen matrix solutions 
(cf. Proposition \ref{prop3.7}). 
Thus we may treat isomonodromy system \eqref{3.1} instead of \eqref{1.2}.  
Applying WKB analysis under condition \eqref{5.1}, we solve 
the direct monodromy problem for linear system
\eqref{3.1} in Section \ref{sc5}, and obtain key relations 
(Corollary \ref{cor5.2}) containing the monodromy data and certain 
integrals. In its procedure we calculate analytic continuations along a Stokes
curve ranging on both upper and lower sheets of the two-sheeted Riemann surface. 
Basic necessary materials for this calculation are summarised in 
Section \ref{sc4}. Asymptotic properties of these integrals are examined
in Section \ref{sc6} by the use of the $\vartheta$-function, and 
in Section \ref{sc7}, from these 
formulas, asymptotic forms in main theorems are found by solving the inverse
monodromy problem for the prescribed monodromy data. These forms are given as
a necessary condition to be a required solution. 
The justification as a solution of \eqref{1.1} is made along the line
of Kitaev \cite{Kitaev-1}, \cite{Kitaev-3}. The final section is devoted to
the Boutroux equations, which determine a modulus contained in the elliptic
representation of solutions.
\par
Throughout this paper we use the following symbols:
\par
(1) $\sigma_1,$ $\sigma_2,$ $\sigma_3$ are the Pauli matrices
$$
\sigma_1=\begin{pmatrix} 0 & 1 \\ 1 & 0 \end{pmatrix}, \quad
\sigma_2=\begin{pmatrix} 0 & -i \\ i & 0 \end{pmatrix}, \quad
\sigma_3=\begin{pmatrix} 1 & 0 \\ 0 & -1 \end{pmatrix}; 
$$
\par
(2) for complex-valued functions $f$ and $g$, we write $f \ll g$ or $g  \gg f$
if $f = O(|g|)$, and write $f\asymp g$ if $g \ll f \ll g.$
\section{Main results}\label{sc2}
To state our main results we give some explanations on necessary facts.
\subsection{Monodromy data}\label{ssc2.1}
Isomonodromy system \eqref{1.2} admits the matrix solutions 
$$
U^{\infty}_k(\lambda) =(I+O(\lambda^{-1})) \exp(\tfrac 12 ix\lambda \sigma_3)
\lambda^{-\frac 12 \theta_{\infty}\sigma_3}
$$
as $\lambda \to \infty$ through the sector  
$|\arg \lambda +\arg x - k\pi | <\pi,$ and
$$
U^0_k(\lambda)= \Delta_0 (I +O(\lambda))
\exp(-\tfrac 12 ix\lambda^{-1} \sigma_3) \lambda^{\frac 12 \theta_0 \sigma_3}
$$
as $\lambda \to 0$ through the sector $|\arg \lambda -\arg x- k\pi| <\pi$,
where $k \in \mathbb{Z},$ and
$$
\Delta_0= \begin{pmatrix} v^{1/2} f_1  &  v^{1/2} \z f_2 \\
                          v^{-1/2} f_1 &  v^{-1/2} (\z- x)f_2 \end{pmatrix},
$$
$f_1$, $f_2$ being such that $ x f_1f_2 \equiv -1.$
Let 
$(G,\hat{G}, S^{\infty}_0, S^{\infty}_1, S^0_0, S^0_1)$ be invariant
monodromy data defined by
$U^{\infty}_{k+1}(\lambda) = U^{\infty}_k(\lambda)S^{\infty}_k,$  
$U^{0}_{k+1}(\lambda) = U^{0}_k(\lambda)S^{0}_k $ with $k=0,1$,
and by $U^{\infty}_0(\lambda)=U^0_0(\lambda) G$,
$U^{\infty}_1(\lambda)=U^0_1(\lambda) \hat{G}$ with $G=(g_{ij}),$ 
$\hat{G}=(\hat{g}_{ij}) \in SL_2(\mathbb{C}).$ These Stokes matrices are 
written as 
$$
S^{\infty}_0 =\begin{pmatrix} 1 & s^{\infty}_0 \\ 0 & 1 \end{pmatrix}, \quad
S^{\infty}_1 =\begin{pmatrix} 1 & 0 \\ s^{\infty}_1 & 1 \end{pmatrix}, \quad
S^{0}_0 =\begin{pmatrix} 1 & s^0_0 \\ 0 & 1 \end{pmatrix}, \quad 
S^{0}_1 =\begin{pmatrix} 1 & 0 \\ s^0_1  & 1 \end{pmatrix}. 
$$
Then the monodromy manifold $\mathcal{M} \subset \mathbb{C}^{12}$ is given by 
\begin{equation*}
 S^0_0 \hat{G} =G S^{\infty}_0, \quad G^{-1}S^0_0S^0_1 e^{-\pi i\theta_0
\sigma_3} G =S^{\infty}_0S^{\infty}_1 e^{\pi i\theta_{\infty} \sigma_3}
\end{equation*}
(see Section \ref{ssc3.1} and \cite{K-Sb}).
A generic point on the monodromy manifold $\mathcal{M}$ is determined
by $(G,\hat{G})$ (cf. Propositions \ref{3.4} and \ref{3.5}). 
As described in Section \ref{ssc3.1}, 
a change of the matrix solution basis induces an action on the monodromy
data on $\mathcal{M}$, and each solution of \eqref{1.1} corresponds to an orbit,
or equivalence class, yielded by dividing $\mathcal{M}$ by this action. 
Then an orbit passing through $(G, \hat{G})$
parametrises a general solution, which will be simply called
a {\it solution labelled by $(G,\hat{G}).$}
\subsection{Elliptic curve and Boutroux equations}\label{ssc2.2}
As shown in Lemma \ref{lem8.0}, as long as $A\in \mathbb{C}\setminus \{c<-2\}$,
the polynomial $\lambda^4-A\lambda^2+1$
has roots $\pm \lambda_1= \pm\lambda_1(A),$ $\pm\lambda_2=\pm\lambda_2(A)$ 
such that $\lambda_1\lambda_2 =1,$ $\lambda_1,$ 
$\lambda_2 \in \{\re \lambda \ge 0\}$ and
$-\lambda_1,$ $-\lambda_2 \in \{\re \lambda \le 0\}$.
Let $\Pi_+$ and $\Pi_-$ be the copies of $P^1(\mathbb{C}) \setminus ([-\lambda_2,
-\lambda_1] \cup [\lambda_1,\lambda_2])$ and set $\Pi_A = \Pi_+ \cup \Pi_-$ 
glued  along the cuts
$[-\lambda_2, -\lambda_1]$ and $[\lambda_1,\lambda_2]$.  
The Riemann surface $\Pi_A$ is the elliptic curve given by 
$$
w(A,\lambda)^2 =\lambda^4 -A\lambda^2+1,
$$ 
where the branch of 
\begin{align*}
w(A,\lambda)&=\sqrt{\lambda^4 -A\lambda^2 +1}:
= \sqrt{(1+\lambda^{-1}_{1}\lambda)(1-\lambda^{-1}_{1}\lambda)
(1+\lambda_{2}^{-1}\lambda)(1-\lambda_{2}^{-1}\lambda)  }
\\
& =\sqrt {1+\lambda_1^{-1}\lambda}\, \sqrt {1-\lambda_1^{-1}\lambda}\, 
\sqrt{ 1+\lambda_2^{-1}\lambda}\, \sqrt {1-\lambda_2^{-1}\lambda} 
\end{align*}
is chosen in such a way that $ \sqrt{1\pm \lambda_j^{-1}\lambda} \to 1$ as 
$\lambda \to 0$ on the upper plane $\Pi_+.$ Then $\lambda^{-2} w(A,\lambda)
\to -1$ as $\lambda \to\infty$, and $w(A,\lambda) \to 1$ as $\lambda \to 0$
on the upper plane $\Pi_+.$ The elliptic curve does not degenerate as long as
$A\not= \pm 2,$ and $\Pi_A$ may be defined continuously.
\par
As will be shown in Section \ref{sc8}, for each $\phi \in \mathbb{R}$, there 
exists $A_{\phi} \in \mathbb{C}\setminus \{c<-2\}$ 
with $\Pi_{A_{\phi}}$ such that, for 
every cycle $\mathbf{c}$ on $\Pi_{A_{\phi}}$ 
$$
\im e^{i\phi} \int_{\mathbf{c}}\frac{w(A_{\phi},\lambda)}{\lambda^2} d\lambda=0,
$$
and that $A_{\phi}$ has the properties (Proposition \ref{prop8.15}):
\par
(1) for every $\phi$, $A_{\phi}$ is uniquely determined;
\par
(2) $A_{\phi}$ is continuous in $\phi \in \mathbb{R},$ and is smooth in
$\phi \in \mathbb{R}\setminus \{m\pi/2 \,|\, m\in \mathbb{Z} \}$;
\par
(3) $A_{\phi \pm \pi/2}= - A_{\phi},$ $A_{\phi+\pi }
=A_{\phi},$ $A_{-\phi}=\overline{A_{\phi}};$
\par
(4) $\Pi_{A_{\phi}}$ degenerates if and only if $\phi=m\pi/2$ with $m\in
\mathbb{Z},$ and $A_0= 2,$ $A_{\pm \pi/2}=-2.$
\par
By Proposition \ref{prop8.16}, the roots 
$\pm \lambda_1=\pm \lambda_1(A_{\phi})$, 
$\pm \lambda_2=\pm \lambda_2(A_{\phi})$ of $w(A_{\phi},\lambda)^2$
may be numbered in such a way that, for $\phi$ close to $0$, 
\par
$\re \lambda_1(A_{\phi})< \re \lambda_2(A_{\phi})$, \,\,\,
$\im\lambda_1(A_{\phi})<0<\im \lambda_2(A_{\phi})$ \,\,\, if $\phi <0,$ and
\par
$\re \lambda_1(A_{\phi})< \re \lambda_2(A_{\phi})$, \,\,\,
$\im\lambda_1(A_{\phi})>0>\im \lambda_2(A_{\phi})$ \,\,\, if $\phi >0$ 
\par\noindent
(cf.~Figure \ref{cycles1}), and that the numbering is retained 
for $0<|\phi|<\pi/2$.
{\small
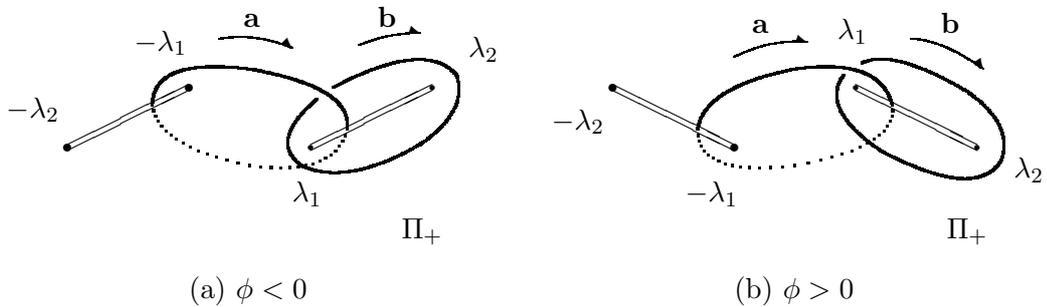
\begin{figure}[htb]
\begin{center}
\unitlength=0.8mm
\begin{picture}(80,50)(-10,-5)
 \put(0,25){\makebox{$-\lambda_2$}}
  \put(21,36){\makebox{$-\lambda_1$}}
 \put(76,35){\makebox{$\lambda_2$}}
   \put(47,11){\makebox{$\lambda_1$}}
\thinlines
\put(10,19.5){\line(2,1){20}}
\put(10,20.5){\line(2,1){20}}
\put(50,19.5){\line(2,1){20}}
\put(50,20.5){\line(2,1){20}}
 \qbezier(35,38) (41,39) (47,36.5)
   \qbezier(58,37) (63,39) (68,39)
\put(47,36.5){\vector(1,-1){0}}
  \put(68,39){\vector(3,-1){0}}
\put(39,40){\makebox{$\mathbf{a}$}}
  \put(61,40){\makebox{$\mathbf{b}$}}
 \put(65,5){\makebox{$\Pi_{+}$}}
\thicklines
\put(10,20){\circle*{1.5}}
\put(30,30){\circle*{1.5}}
\put(50,20){\circle*{1}}
\put(70,30){\circle*{1}}

\qbezier[15](24,27) (25,20) (40,17)
\qbezier[15](40,17) (54,15) (56,23)
\qbezier(40,33.0) (57,30) (56,23)
\qbezier(24,27) (24,35) (40,33.0)

  \qbezier(48.5,17) (56,13.5) (68,20.5)
  \qbezier(68,20.5) (76,26) (74,32)
  \qbezier(48.5,17) (43,21.0) (51,28.0)
  \qbezier (53.8,30.0)(69,38)(74,32)

 \put(30,-5){\makebox{(a) $\phi <0$}}

\end{picture}\quad\quad
\begin{picture}(80,50)(-10,-5)
\put(0,23){\makebox{$-\lambda_2$}}
  \put(22,11){\makebox{$-\lambda_1$}}
 \put(76,15){\makebox{$\lambda_2$}}
   \put(47,38){\makebox{$\lambda_1$}}
\thinlines
\put(10,30.5){\line(2,-1){20}}
\put(10,29.5){\line(2,-1){20}}
\put(50,30.5){\line(2,-1){20}}
\put(50,29.5){\line(2,-1){20}}
\qbezier(30,35) (36,38) (42,37.5)
  \qbezier(59,38) (64,38) (71,33)
\put(42,37.5){\vector(3,-1){0}}
  \put(71,33){\vector(2,-3){0}}
\put(33,39){\makebox{$\mathbf{a}$}}
  \put(64,39){\makebox{$\mathbf{b}$}}
\put(65,5){\makebox{$\Pi_{+}$}}
\thicklines
\put(10,30){\circle*{1.5}}
\put(30,20){\circle*{1.5}}
\put(50,30){\circle*{1}}
\put(70,20){\circle*{1}}

\qbezier(24,23) (25,30) (40,33)
\qbezier(40,33) (54,35) (56,27)
\qbezier[15](40,17) (57,20) (56,27)
\qbezier[15](24,23) (24,15) (40,17)

  \qbezier(51,34) (56,36.5) (68,29.5)
  \qbezier(68,29.5) (76,24) (74,18)
  \qbezier(48,32) (44,26.5) (54,19.5)
  \qbezier (54,19.5)(70,11)(74,18)

 \put(30,-5){\makebox{(b) $\phi >0$}}

\end{picture}
\phantom{---}
\end{center}
\caption{Cycles $\mathbf{a},$ $\mathbf{b}$ for small $\phi$} 
\label{cycles1}
\end{figure}
}

For small $|\phi|$ we define basic cycles $\mathbf{a}$ and $\mathbf{b}$ 
on $\Pi_{A_{\phi}}=\Pi_+\cup \Pi_-$ drawn as in Figure \ref{cycles1};
and, for $0<|\phi|<\pi/2$, cycles $\mathbf{a}$ and $\mathbf{b}$ are
modified continuously on $\Pi_{A_{\phi}}.$
The Boutroux equations are given by
\begin{equation}\label{2.1}
\im e^{i\phi} \int_{\mathbf{a}} \frac{w(A_{\phi},\lambda)}{\lambda^2} 
d\lambda =0, \quad
\im e^{i\phi} \int_{\mathbf{b}} \frac{w(A_{\phi},\lambda)}{\lambda^2} 
d\lambda =0
\end{equation}
admitting a unique solution $A_{\phi}$.
For $|\phi|<\pi/2$ the periods of $\Pi_{A_{\phi}}$ along $\mathbf{a}$ and
$\mathbf{b}$ are 
$$
\Omega^{\phi}_{\mathbf{a}}=\Omega_{\mathbf{a}}
=\int_{\mathbf{a}} \frac{d\lambda}{w(A_{\phi},\lambda)}, \quad
 \Omega^{\phi}_{\mathbf{b}}=\Omega_{\mathbf{b}}
=\int_{\mathbf{b}} \frac{d\lambda}{w(A_{\phi},\lambda)}, \quad
$$
which satisfy $\im \Omega_{\mathbf{b}}/\Omega_{\mathbf{a}} >0.$
\subsection{Main theorems}\label{ssc2.3}
Let $y(x)=y(G,\hat{G},x)$ be a solution of \eqref{1.1} labelled by the monodromy
data $(G,\hat{G}) \in SL_2(\mathbb{C})^2$ with $G=(g_{ij})$, 
$\hat{G}=(\hat{g}_{ij})$. The roots
$\lambda_1=\lambda_1(A_{\phi}),$ $\lambda_2=\lambda_2(A_{\phi})$ of $w(A_{\phi},
\lambda)^2$ for $0<|\phi|<\pi/2$ are given by
$$
\lambda_1=\textstyle\sqrt{\smash[b]{A_{\phi}/2 -\sqrt{\smash[b]{A_{\phi}^2/4 -1}}} },\quad
\lambda_2=\textstyle\sqrt{\smash[b]{ A_{\phi}/2 
+\sqrt{\smash[b]{ A_{\phi}^2/4 -1}} }},
$$
where the branches are chosen in such a way that $\re \textstyle\sqrt{
\smash[b]{A_{\phi}^2/4-1}} \ge 0$ as $A_{\phi} \to 2$ $(\phi \to 0)$, and 
that $\lambda_{1,2}(A_{\phi}) \to 1$ as $A_{\phi}\to 2$ (cf. Proposition 
\ref{prop8.16} and Figure \ref{trajectory}).
Then we have the following, in which $\mathrm{sn}(u; k)$ denotes the $\mathrm
{sn}$-function solving $(z_u)^2=(1-z^2)(1-k^2z^2).$ 
\begin{thm}\label{thm2.1}
Suppose that $0<\phi<\pi/2$ and that $g_{11}g_{12}g_{22}\hat{g}_{11}\hat{g}_{21}
\not=0.$ Then 
$$
y(x)^{-1}= i\lambda_1 \mathrm{sn}\bigl(2i\lambda_2(x-x_0^+) + O(x^{-\delta}); 
\lambda_1/\lambda_2 \bigr)
$$
as $2x=te^{i\phi} \to \infty$ through the cheese-like strip
$$
S(\phi,t_{\infty}, \kappa_0,\delta_0) =\{2x=te^{i\phi}\,|\, \re t>t_{\infty},\,\,
|\im t|<\kappa_0\} \setminus \bigcup_{\sigma\in \mathcal{P}(x^+_0)} 
\{ |te^{i\phi} -\sigma|<\delta_0 \}
$$
with
$$
\mathcal{P}(x^+_0)
=\{\sigma \,|\, \mathrm{sn}(i\lambda_2(\sigma-2x_0^+); \lambda_1/
\lambda_2)=\infty\}
= 2x_0^+ - i (\tfrac 12 \Omega_{\mathbf{b}} + \tfrac 12
 \Omega_{\mathbf{a}}\mathbb{Z} + \Omega_{\mathbf{b}} 
\mathbb{Z} ).
$$
Here $\delta$ is some positive number, $\kappa_0$ a given positive number,
$\delta_0$ a given small positive number, $t_{\infty}=t_{\infty}(\kappa_0,
\delta_0)$ a sufficiently large number depending on $(\kappa_0,\delta_0);$
and 
\begin{align*}
 2ix^+_0 &= 2i x^+_0(\Omega_{\mathbf{a}},\Omega_{\mathbf{b}}, G,\hat{G}) 
\\
 \equiv &\frac 1{2\pi i}\Bigl(\Omega_{\mathbf{a}} \log(g_{11}
g_{22}) +\Omega_{\mathbf{b}} \log\frac{g_{12}\hat{g}_{21}}{g_{22}\hat{g}_{11}}
\Bigr)
-\frac {\Omega_{\mathbf{a}}} 4(\theta_0-\theta_{\infty}+2)
-\frac {\Omega_{\mathbf{b}}}2
\mod \Omega_{\mathbf{a}}\mathbb{Z}+\Omega_{\mathbf{b}}\mathbb{Z}.
\end{align*}
\end{thm}
\begin{thm}\label{thm2.2}
Suppose that $-\pi/2<\phi <0$ and that $g_{11}g_{21}g_{22}\hat{g}_{12}\hat{g}_{22}\not=0.$ 
Then $y(x)$ admits an asymptotic representation of the same
form as in Theorem $\ref{thm2.1}$ with the phase shift given by  
\begin{align*}
 2ix^-_0 &= 2i x^-_0(\Omega_{\mathbf{a}},\Omega_{\mathbf{b}}, G,\hat{G}) 
\\
 \equiv &\frac {1}{2\pi i}\Bigl(\Omega_{\mathbf{a}} \log(g_{11}g_{22})
 +\Omega_{\mathbf{b}} \log\frac{g_{11}\hat{g}_{22}}{g_{21}\hat{g}_{12}}\Bigr)
-\frac{\Omega_{\mathbf{a}}}4 (\theta_0-\theta_{\infty}+2) -\frac{\Omega_{\mathbf
{b}}}2   
\mod \Omega_{\mathbf{a}}\mathbb{Z}+\Omega_{\mathbf{b}}\mathbb{Z}
\end{align*}
in $S(\phi, t_{\infty}, \kappa_0, \delta_0)$ with $\mathcal{P}(x^-_0).$
\end{thm}
\begin{rem}\label{rem2.1}
By Corollary \ref{cor3.3} and Proposition \ref{prop3.5},
\begin{align*}
&\frac{g_{22}\hat{g}_{11}}{g_{12}\hat{g}_{21}}=\frac{g_{22}}{g_{12}}
\Bigl(\frac{g_{11}}{g_{21}}-s^0_0\Bigr) =\frac{g_{11}g_{22}}{1-g_{11}g_{22}}
\Bigl(1-\frac{g_{21}}{g_{11}}s^0_0 \Bigr)
=\frac{g_{11}g_{22}}{1-g_{11}g_{22}}\cdot
\frac{\hat{g}_{11}}{g_{11}},
\\
&\frac{g_{21}\hat{g}_{12}}{g_{11}\hat{g}_{22}}=\frac{g_{12}g_{21}^2}
{g_{11}(e^{\pi i(\theta_0-\theta_{\infty})}+g_{12}g_{21})} 
\Bigl(\frac{g_{11}}{g_{21}}-s^0_0\Bigr) =\frac{1-g_{11}g_{22}}
{e^{\pi i(\theta_0-\theta_{\infty})}+1-g_{11}g_{22}} \cdot
\frac{\hat{g}_{11}}{g_{11}},
\end{align*} 
that is, the constants consisting of $g_{ij}$ and $\hat{g}_{ij}$ in $x^{\pm}_0$ 
are represented
in terms of $g_{11}g_{22}$ and $\hat{g}_{11}/g_{11}$, which are invariants
under an action on monodromy data as in Proposition  \ref{prop3.51}.
\end{rem}
For $\phi$ such that $|\phi-m\pi|<\pi/2$ $(m\in \mathbb{Z}),$ set $\Omega^{\phi}
_{\mathbf{a}, \mathbf{b}} =e^{m\pi i} \Omega^{\phi-m\pi}_{\mathbf{a},\mathbf{b}}
$ with $\Omega^{\tilde{\phi}}_{\mathbf{a},\mathbf{b}} =\Omega_{\mathbf{a},
\mathbf{b}}$ for $|\tilde{\phi}|<\pi/2,$ and $\lambda_{1,2}=\lambda_{1,2}
(A_{\phi-m\pi}).$ Then we have the following.
\begin{thm}\label{thm2.3}
Suppose that $0<\phi-m\pi<\pi/2$ $($respectively, $-\pi/2 <\phi-m\pi <0)$
with $m\in \mathbb{Z}\setminus \{0\}.$ 
Then $y(x)$ admits the expression
$$
y(x)^{-1}= i\lambda_1 \mathrm{sn}\bigl(2i\lambda_2(x-x_0^{(m)})
 + O(x^{-\delta}); \lambda_1/\lambda_2 \bigr)
$$
as $2x=te^{i\phi} \to \infty$ through $S(\phi,t_{\infty},
\kappa_0,\delta_0)$ with $\mathcal{P}(x_0^{(m)})$, 
if $g^{(m)}_{11} g^{(m)}_{12} g^{(m)}_{22}\hat{g}^{(m)}
_{11}\hat{g}^{(m)}_{21}\not=0$ $($respectively, $g^{(m)}_{11} g^{(m)}_{21} 
g^{(m)}_{22}\hat{g}^{(m)}_{12} \hat{g}^{(m)}_{22} \not=0)$. Here, for
$0<\pm (\phi-m\pi)<\pi/2$, $m\in \mathbb{Z}\setminus \{0\},$
$$
x^{(m)}_0 =x^{\pm}_0(\Omega^{\phi}_{\mathbf{a}}, \Omega^{\phi}_{\mathbf{b}},
G_m,\hat{G}_m)
$$
with
\begin{align*}
G_m =(g^{(m)}_{ij})=& 
 e^{\tfrac 12 \pi im\theta_0\sigma_3} G(S^{\infty}_0S^{\infty}_1 e^{\pi i
\theta_{\infty}\sigma_3})^m e^{-\tfrac 12 \pi im\theta_{\infty}\sigma_3},   
\\
\hat{G}_m =(\hat{g}^{(m)}_{ij})=& 
e^{\tfrac 12\pi i m\theta_0 \sigma_3} \hat{G}(S^{\infty}_0)^{-1}
(S^{\infty}_0S^{\infty}_1e^{\pi i\theta_{\infty}\sigma_3})^m 
S^{\infty}_0 e^{-\tfrac 12\pi i m\theta_{\infty} \sigma_3}.
\end{align*}
\end{thm}
\section{Isomonodromy deformation and monodromy data}\label{sc3}
\subsection{Monodromy data}\label{ssc3.1}
We examine the monodromy data for system \eqref{1.2} (see also \cite[Section 1]
{K-Sb}). 
System \eqref{1.2} admits the matrix solutions
$$
U^{\infty}_k(\lambda)=(I+O(\lambda^{-1})) \exp(\tfrac 12 ix\lambda \sigma_3)
\lambda^{-\tfrac 12 \theta_{\infty}\sigma_3}
$$
as $\lambda \to \infty$ through the sector $|\arg \lambda +\arg x - k\pi |
<\pi,$ and
$$
U^{0}_k(\lambda)=\Delta_0 (I+O(\lambda)) \exp(-\tfrac 12 ix\lambda^{-1} \sigma_3)
\lambda^{\tfrac 12 \theta_{0}\sigma_3}
$$
as $\lambda \to 0$ through the sector $|\arg \lambda -\arg x - k\pi |<\pi,$ 
where
$$
\Delta_0=\begin{pmatrix} v^{1/2} f_1  &  v^{1/2}\z f_2 \\
                        v^{-1/2} f_1  & v^{-1/2}(\z-x) f_2 \end{pmatrix},
$$
$f_1,$ $f_2$ being such that $xf_1f_2\equiv -1.$ Let $S^{\infty}_k$, $S^0_k$ be
Stokes matrices such that
$$
U^{\infty}_{k+1}(\lambda)=U^{\infty}_k(\lambda)S^{\infty}_k, \quad
U^{0}_{k+1}(\lambda)=U^{0}_k(\lambda)S^{0}_k. 
$$
Then
$$
S^{\infty}_0=\begin{pmatrix}  1 & s_0^{\infty} \\ 0 & 1   \end{pmatrix},
\quad S^{\infty}_1=\begin{pmatrix}  1 & 0 \\ s_1^{\infty} & 1   \end{pmatrix},
\quad S^{0}_0=\begin{pmatrix}  1 & s_0^{0} \\ 0 & 1   \end{pmatrix},
\quad S^{0}_1=\begin{pmatrix}  1 & 0 \\ s_1^{0} & 1   \end{pmatrix}.
$$
Let $G=(g_{ij}),$ $\hat{G}=(\hat{g}_{ij}) \in SL_2(\mathbb{C})$ be 
connection matrices defined by
$$
U^{\infty}_0(\lambda)=U^0_0(\lambda)G, \quad
U^{\infty}_1(\lambda)=U^0_1(\lambda)\hat{G}. 
$$
By the uniqueness of the asymptotic expression in each sector, for every
$m\in \mathbb{Z},$
\begin{align*}
& U^{\infty}_k(e^{-2\pi im} \lambda) e^{-\pi im\theta_{\infty}\sigma_3}
= U^{\infty}_{k+2m}(\lambda),
\\
& U^{0}_k(e^{-2\pi im} \lambda) e^{\pi im\theta_{0}\sigma_3}
= U^{0}_{k+2m}(\lambda).
\end{align*} 
Then we have the following.
\begin{prop}\label{prop3.1}
For $k\in \mathbb{Z},$ $S^{\infty}_{k+2}=e^{\pi i\theta_{\infty} \sigma_3}
S^{\infty}_k e^{-\pi i\theta_{\infty} \sigma_3},$
 $S^{0}_{k+2}=e^{-\pi i\theta_{0} \sigma_3}
S^{0}_k e^{\pi i\theta_{0} \sigma_3}.$
\end{prop}
\begin{prop}\label{prop3.2}
$S^0_0 S^0_1 e^{-\pi i \theta_0\sigma_3} G e^{-\pi i\theta_{\infty}\sigma_3} 
= GS^{\infty}_{0}S^{\infty}_1,$  $GS^{\infty}_0=S^0_0 \hat{G}.$ 
\end{prop}
\begin{cor}\label{cor3.3}
The entries of $S^{\infty}_k,$ $S^0_k$, $G$ and $\hat{G}$ fulfil
\begin{align*}
& s^0_0s^0_1 e^{-\pi i\theta_0} + 2\cos \pi \theta_0
= s^{\infty}_0s^{\infty}_1 e^{\pi i\theta_{\infty}} + 2\cos \pi \theta_{\infty},
\\
& g_{11}=\hat{g}_{11}+s^0_0\hat{g}_{21}, \quad
 g_{12}+s^{\infty}_0{g}_{11} = \hat{g}_{12}+s^0_0\hat{g}_{22}, \quad
 g_{21}= \hat{g}_{21}, \quad
 g_{22}+s^{\infty}_0{g}_{21} = \hat{g}_{22}. 
\end{align*}
\end{cor}
Comparison of the entries of 
$G^{-1} S^0_0 S^0_1 e^{-\pi i \theta_0\sigma_3} G 
= S^{\infty}_{0}S^{\infty}_1e^{\pi i\theta_{\infty}\sigma_3} $ leads us to the
following.
\begin{prop}\label{prop3.4}
The entries $s^{\infty}_0,$ $s^{\infty}_1$, $s^0_0$, $s^0_1$ and $g_{ij}$
fulfil
\begin{align*}
& e^{-\pi i\theta_{\infty}} \Bigl(s^{\infty}_0+\frac{g_{22}}{g_{21}}\Bigr)=
\frac{g_{12}}{g_{21}} e^{-\pi i\theta_{0}} \Bigl(s^0_1+ \frac
{g_{22}}{g_{12}} e^{2\pi i\theta_{0}} \Bigr), 
\\
& e^{\pi i\theta_{\infty}} \Bigl(s^{\infty}_1-\frac
{g_{11}}{g_{12}} e^{-2\pi i\theta_{\infty}} \Bigr)
= \frac{g_{21}}{g_{12}} e^{\pi i\theta_0} 
\Bigl(s^{0}_0-\frac{g_{11}}{g_{21}}\Bigr),
\\
& \Bigl(s^{\infty}_0+\frac{g_{22}}{g_{21}}\Bigr)\Bigl(s^{\infty}_1-\frac
{g_{11}}{g_{12}} e^{-2\pi i\theta_{\infty}} \Bigr)
= \Bigl(s^{0}_0-\frac{g_{11}}{g_{21}}\Bigr)\Bigl(s^{0}_1+\frac
{g_{22}}{g_{12}} e^{2\pi i\theta_{0}} \Bigr) = -1-\frac{e^{\pi i(\theta_0
-\theta_{\infty})}}{g_{12} g_{21}}.
\end{align*}
\end{prop}  
\begin{prop}\label{prop3.5}
The entries $g_{ij}$ and $\hat{g}_{ij}$ fulfil
$$
\frac{\hat{g}_{11}}{\hat{g}_{21}} =\frac{g_{11}}{g_{21}}-s^0_0, \quad
\frac{\hat{g}_{12}}{\hat{g}_{22}} =\frac{g_{12}g_{21}} {e^{\pi i(\theta_0
-\theta_{\infty})} +g_{12}g_{21}}\Bigl(
\frac{g_{11}}{g_{21}}-s^0_0\Bigr). 
$$
\end{prop}
\begin{proof}
By Corollary \ref{cor3.3}, 
$$
\frac{\hat{g}_{12}}{\hat{g}_{22}}=\frac{g_{12}+s^{\infty}_0g_{11}}
{g_{22}+s^{\infty}_0g_{21}}-s^0_0 = \frac{g_{11}}{g_{21}} -\frac{1/g_{21}}
{g_{22}+s^{\infty}_0 g_{21} }-s^0_0,
$$
and by the first and third relations in Proposition \ref{prop3.4},
$$
e^{-\pi i\theta_{\infty}}(g_{22}+s^{\infty}_0 g_{21}) =g_{12}e^{-\pi i\theta_0}
\Bigl(s^0_1 +\frac{g_{22}}{g_{12}}e^{2\pi i\theta_0}\Bigr)
 =g_{12}e^{-\pi i\theta_0}
\Bigl(s^0_0 -\frac{g_{11}}{g_{21}} \Bigr)^{-1}\Bigl(-1-\frac{e^{\pi i(\theta_0
-\theta_{\infty})} }{g_{12}g_{21}} \Bigr).
$$
From these equations the second relation of the proposition follows. 
\end{proof}
The monodromy manifold $\mathcal{M}$ is given by the equations 
$S^0_0 S^0_1 e^{-\pi i \theta_0\sigma_3} G e^{-\pi i\theta_{\infty}\sigma_3} 
= GS^{\infty}_{0}S^{\infty}_1$ and $GS^{\infty}_0=S^0_0 \hat{G}$ in 
$\mathbb{C}^{12}$. These equations contain six independent relations and
$g_{11}g_{22}-g_{12}g_{21}=\hat{g}_{11}\hat{g}_{22}-\hat{g}_{12}\hat{g}_{21}
=1$. Hence $\mathrm{dim}_{\mathbb{C}} \mathcal{M}=4$, and a generic point
on $\mathcal{M}$ is denoted by, say $(G,s^0_0)$, or $(G,\hat{G})$ by 
Propositions \ref{prop3.4} and \ref{prop3.5}.
\par
Instead of $(U^{\infty}_j(\lambda),\,U^0_j(\lambda))$, $U^0_j(\lambda)=
 U^0_j(\lambda; v,f_1,f_2)$, we may take 
$$
(\hat{U}^{\infty}_j(\lambda),\, U^0_j(\lambda; cv, f_1,f_2)), \,\,\,
\text{or} \,\,\,
(U^{\infty}_j(\lambda),\, U^0_j(\lambda; v, \tilde{c}f_1,\tilde{c}^{-1}f_2)) 
\,\,\,
\text{for any $ c, \tilde{c} \in \mathbb{C}\setminus \{0\}$}, 
$$
where 
$\hat{U}^{\infty}_j(\lambda):=c^{\sigma_3/2} U^{\infty}_j(\lambda)
 c^{-\sigma_3/2} =(I+O(\lambda^{-1}))U^{\infty}_j(\lambda)$, 
$U^0_j(\lambda; cv, f_1,f_2)=c^{\sigma_3/2} U^0_j(\lambda), $ and
$ U^0_j(\lambda; v, \tilde{c}f_1,\tilde{c}^{-1}f_2)=U^0_j(\lambda)\tilde{c}
^{\sigma_3}$. 
Then the connection formula $U^{\infty}_0(\lambda)
=U^0_0(\lambda)G$ becomes $\hat{U}^{\infty}_0(\lambda)
=U^0_0(\lambda; cv, f_1, f_2)G c^{-\sigma_3/2}$, or ${U}^{\infty}_0(\lambda)
=U^0_0(\lambda; v, \tilde{c}f_1, \tilde{c}^{-1}f_2)\tilde{c}^{-\sigma_3} G$,
respectively, and the same relation holds for $\hat{G}$. This fact induces 
the action $[*]:$ $(G,\hat{G}) \mapsto
 (\tilde{c}^{-\sigma_3} G c^{-\sigma_3/2},
 \tilde{c}^{-\sigma_3} \hat{G} c^{-\sigma_3/2})$, and
$(S^{\infty}_j, S^0_j) \mapsto (c^{\sigma_3/2}S^{\infty}_j c^{-\sigma_3/2}, 
\tilde{c}^{-\sigma_3} S^0_j \tilde{c}^{\sigma_3}) $ with any $c, \tilde{c}
\in \mathbb{C}.$ As shown in \cite[Section 3.4]{P-S}, 
dividing $\mathcal{M}$ by this action $[*]$ yields orbits
constituting an affine cubic surface $V(\mathcal{M}) \subset \mathbb{C}^3$ 
parametrised 
by $(\theta_0,\theta_{\infty})$, and each solution of \eqref{1.1} corresponds 
to a point on $V(\mathcal{M})$.  
It is easy to see the following.
\begin{prop}\label{prop3.51}
The quantities
$$
g_{11}g_{22}, \quad 
\frac{\hat{g}_{11}}{g_{11}}=1-s^0_0\frac{g_{21}}{g_{11}}, \quad 
\frac{\hat{g}_{22}}{g_{22}}=1+s^{\infty}_0\frac{g_{21}}{g_{22}}, \quad 
\frac{\hat{g}_{12}}{g_{12}}=1-s^{0}_0\frac{\hat{g}_{22}}{g_{12}} 
+s^{\infty}_0\frac{{g}_{11}}{g_{12}}, \quad 
\frac{\hat{g}_{21}}{g_{21}}=1 
$$
are invariants under the action $[*].$
\end{prop}
Note that $e^{\pi i\theta_{\infty}\sigma_3} S_k^{\infty}S^{\infty}_{k+1}
e^{-\pi i\theta_{\infty}\sigma_3}=S^{\infty}_{k+2}S^{\infty}_{k+3}.$
By induction on $m,$
$S^{\infty}_{2m-2}S^{\infty}_{2m-1}=e^{(m-1)\pi i\theta_{\infty}\sigma_3}
S^{\infty}_0S^{\infty}_1 e^{-(m-1)\pi i\theta_{\infty}\sigma_3}.$
Thus we have the following.
\begin{prop}\label{prop3.6}
Let $m \in \mathbb{Z}\setminus\{0\}$. Then
\begin{align*}
&S^{\infty}_0S^{\infty}_1 \cdots S^{\infty}_{2m-2}S^{\infty}_{2m-1}
=(S^{\infty}_0S^{\infty}_1 e^{\pi i\theta_{\infty}\sigma_3})^m e^{-m\pi i
\theta_{\infty}\sigma_3} \quad \text{for $m\ge 1,$ and}
\\
&S^{\infty}_{2m}S^{\infty}_{2m+1} \cdots S^{\infty}_{-2}S^{\infty}_{-1}
=e^{m\pi i\theta_{\infty}\sigma_3}(S^{\infty}_0S^{\infty}_1 
e^{\pi i\theta_{\infty}\sigma_3})^{-m} \quad \text{for $m \le -1.$}  
\end{align*}
\end{prop}
{\bf Comparison with the monodromy data of Kitaev.}
Kitaev \cite{K-Sb} considered the isomonodromy deformation of the linear
system
\begin{equation}\label{3.K1}
\frac{d\Phi}{d\tilde{\lambda}}
= \tau \biggl( -i\sigma_3 -\frac{ai}{2\tau\tilde{\lambda}}\sigma_3
-\frac 1{\tilde{\lambda}} \begin{pmatrix} 0 & C \\ D & 0 \end{pmatrix}
+\frac{i}{2\tilde{\lambda}^2} \begin{pmatrix} 
\sqrt{\theta^2-AB} & A \\ B &-\sqrt{\theta^2 -AB} \end{pmatrix} \biggr) \Phi
\end{equation} 
with $\theta^2= 1$ to treat the complete $P_{\mathrm{III}}$. 
For $\tau >0 $ system \eqref{3.K1} has the matrix solutions
\begin{align*}
& X_k(\tilde{\lambda})=(I+O(\tilde{\lambda}^{-1}))
\exp (-i\tau\tilde{\lambda} \sigma_3+ T_{\infty}
\log (\tilde{\lambda}/{\tau}) ) \quad \text{as $\tilde{\lambda} \to \infty$},
\\
& Y_k(\tilde{\lambda})=(\Phi_0 +O(\tilde{\lambda})) 
\exp (-\tfrac 12 {i\tau}{\tilde{\lambda}^{-1}}
 \sigma_3+T_0\log (\tilde{\lambda}/{\tau})) \quad \text{as $\tilde{\lambda}
 \to 0$}
\end{align*}
through the sector $|\arg \tilde{\lambda} +\arg \tau-(k-1)\pi|<\pi$ for $k=1,2$,
where $T_{\infty}=-\tfrac 12 ai \sigma_3,$ $T_0 =-\tfrac 12 ci \sigma_3$,
and $\det \Phi_0 \not=0.$
Then the monodromy data $\mathrm{G}_1,$ $\mathrm{G}_2$, $\mathrm{S}_j$
$(1\le j \le 4)$ are given by
\begin{align*}
&X_2=X_1 \mathrm{S}_1, \quad X_1(\tilde{\lambda} e^{-2\pi i})
e^{2\pi iT_{\infty}}=
X_2 \mathrm{S}_2, \quad  Y_2=Y_1 \mathrm{S}_3, \quad Y_1(\tilde{\lambda}
 e^{-2\pi i})
e^{2\pi iT_0}=Y_2 \mathrm{S}_4,
\\
& Y_2=X_2\mathrm{G}_2, \quad  Y_1=X_1\mathrm{G}_1.
\end{align*}
Let us set $x=\sqrt{2} i\tau,$ $\lambda=\sqrt{2} i \tilde{\lambda},$ 
$y=\sqrt{2} i w,$
$\mathfrak{z}=\sqrt{2} i z$ in our system \eqref{1.2}. Then we have
\begin{align}\label{3.K2}
\frac{dV}{d\tilde{\lambda}}=\biggl(&
 -i\tau \sigma_3 +\frac 1{\tilde{\lambda}}
 \begin{pmatrix}
-\frac 12\theta_{\infty}  &  2wzv  \\
\frac 12 \theta_0\tau(zv)^{-1} -v^{-1}(\tfrac 12 \theta_{\infty}(2-\tau z^{-1})
-2w(z-\tau) ) & \tfrac 12 \theta_{\infty} \end{pmatrix} 
\\
\notag
& - \frac i {\tilde{\lambda}^{2}}  
\begin{pmatrix}   z-\tfrac 12 \tau & -zv  \\  v^{-1}(z-\tau) & -z+\tfrac 12
\tau  \end{pmatrix}
\biggr)V,
\end{align}
which admits the matrix solutions
$$
 V^{\infty}_k(\tilde{\lambda})=(I+O(\tilde{\lambda}^{-1}))
\exp (-i\tau\tilde{\lambda} \sigma_3) \tilde{\lambda}^{-\frac 12
\theta_{\infty} \sigma_3}
$$
as $\tilde{\lambda} \to \infty$ through the sector 
$|\arg (\tilde{\lambda} \tau)-(k-1)\pi|<\pi,$ and
$$
 V^0_k(\tilde{\lambda})=\Delta_0(I +O(\tilde{\lambda})) 
\exp (-\tfrac 12 {i\tau}{\tilde{\lambda}^{-1}}\sigma_3)
\tilde{\lambda}^{\frac 12 \theta_0\sigma_3}
$$
as $\tilde{\lambda} \to 0$ through the sector 
$|\arg(\tilde{\lambda}/\tau)-k\pi |<\pi$ for $k\in \mathbb{Z}.$ Note that
$V^{\infty}_k(\tilde{\lambda})=U^{\infty}_k(\lambda)(\sqrt{2}i)
^{\frac 12 \theta_{\infty}\sigma_3},$
$V^{0}_{k-1}(\tilde{\lambda})=U^{0}_{k-1}(\lambda)(\sqrt{2}i)
^{-\frac 12 \theta_{0}\sigma_3}$.
If $V^{\infty}_k(\tilde{\lambda})$ and $V^0_{k-1}(\tilde{\lambda})$ 
solving \eqref{3.K2} are identified with, respectively, 
$X_k(\tilde{\lambda})\tau^{T_{\infty}}$ and $Y_k(\tilde{\lambda})\tau^{T_0}$
solving \eqref{3.K1} with suitable $a,$ $c$, then the monodromy data 
$(S^{\infty}_j, S^0_j, G,\hat{G})$ and $(\mathrm{S}_j,
\mathrm{G}_1, \mathrm{G}_2)$ fulfil
\begin{align*} 
&(\sqrt{2}i)^{-\frac 12 \theta_{\infty}\sigma_3} S^{\infty}_1
(\sqrt{2}i)^{\frac 12 \theta_{\infty}\sigma_3} = \tau^{-T_{\infty}} 
\mathrm{S}_1 \tau^{T_{\infty}}, \quad
(\sqrt{2}i)^{-\frac 12 \theta_{\infty}\sigma_3} S^{\infty}_2
(\sqrt{2}i)^{\frac 12 \theta_{\infty}\sigma_3} = \tau^{-T_{\infty}}
\mathrm{S}_2 \tau^{T_{\infty}}, \quad
\\
&(\sqrt{2}i)^{\frac 12 \theta_{0}\sigma_3} S^{0}_0
(\sqrt{2}i)^{-\frac 12 \theta_{0}\sigma_3} = \tau^{-T_0} \mathrm{S}_3
\tau^{T_0}, \quad
(\sqrt{2}i)^{\frac 12 \theta_{0}\sigma_3} S^{0}_1
(\sqrt{2}i)^{-\frac 12 \theta_{0}\sigma_3} = \tau^{-T_0} \mathrm{S}_4
\tau^{T_0}, 
\\
&(\sqrt{2}i)^{\frac 12\theta_0\sigma_3}GS_0^{\infty}
(\sqrt{2}i)^{\frac 12 \theta_{\infty}\sigma_3}=
\tau^{-T_0}\mathrm{G}_1^{-1} \tau^{T_{\infty}}, 
\\
&(\sqrt{2}i)^{\frac 12\theta_0\sigma_3}\hat{G}S_1^{\infty}
 (\sqrt{2}i)^{\frac 12 \theta_{\infty}\sigma_3}
=\tau^{-T_0}\mathrm{G}_2^{-1} \tau^{T_{\infty}}. 
\end{align*}
\subsection{Isomonodromy deformation}\label{ssc3.2}
By $U= v^{\sigma_3/2} Y,$ $ 2x=\xi,$
system \eqref{1.2} is taken into
\begin{equation*}
\frac{dY}{d\lambda}= B(\xi, \lambda)Y  \phantom{-----------}
\end{equation*}
with
\begin{align*}
B(\xi, \lambda)=&\frac{i\xi}4 \sigma_3 
+ \begin{pmatrix} -\tfrac 12 \theta_{\infty}  &  -y \z \\
   \tfrac 14 \theta_0 \xi \z^{-1} -\tfrac 12 \theta_{\infty} (2-\tfrac 12\xi
\z^{-1})-y(\z-\tfrac 12\xi)  & \tfrac 12 \theta_{\infty}   
\end{pmatrix} \lambda^{-1}
\\
& - i \begin{pmatrix} \z- \tfrac 14 \xi & -\z  \\  \z- \tfrac 12 \xi & 
-\z +\tfrac 14 \xi   \end{pmatrix} \lambda^{-2}.
\end{align*}
Set $\xi =t e^{i\phi}$ and
write the right-hand side in the form $B(\xi, \lambda)=
(t/4)\mathcal{B}(t,\lambda).$ In what follows instead of \eqref{1.2}
we are concerned with the linear system
\begin{align}\label{3.1}
\frac{dY}{d\lambda} & = \frac {t}4 \mathcal{B}(t,\lambda)Y, \quad
\mathcal{B}(t,\lambda) =b_1 \sigma_1 +b_2 \sigma_2 +b_3 \sigma_3,
\\
\notag
& b_1=(\Gamma(t,\z,y) -{2y\z}t^{-1})\lambda^{-1}
+i e^{i\phi} \lambda^{-2},
\\
\notag
& b_2=-i(\Gamma(t,\z,y) +{2y\z}t^{-1} )\lambda^{-1}
- (4\z t^{-1}  -e^{i\phi} )\lambda^{-2},
\\
\notag
& b_3= ie^{i\phi}-2\theta_{\infty}t^{-1} \lambda^{-1} - {i}
(4\z t^{-1} -e^{i\phi})\lambda^{-2}
\end{align}
with
\begin{align}
\notag
& \Gamma(t,\z,y)=
- y(2\z t^{-1} -e^{i\phi}) +\frac 12 {e^{i\phi}\theta_0}\z^{-1}
 -\frac 12 {\theta_{\infty}}\z^{-1}(4\z t^{-1} -e^{i\phi}),
\\
\label{3.2}
& 4\z= ty^*y^{-2} +e^{i\phi}t(1-y^{-2})  -(2\theta_{\infty} -1)y^{-1},
\end{align}
where $y$ and $y^*$ are arbitrary complex parameters. 
\par
System \eqref{3.1} admits the matrix solutions
$$
Y^{\infty}_k(\lambda)=v^{-\sigma_3/2} U^{\infty}_k(\lambda) 
= v^{-\sigma_3/2} (I+O(\lambda^{-1})) \exp(\tfrac 14 ie^{i\phi}t\lambda \sigma_3)
\lambda^{-\tfrac 12\theta_{\infty}\sigma_3}
$$
as $\lambda \to \infty$ through the sector $|\arg \lambda+\phi-k\pi|<\pi,$
and
\begin{equation}\label{3.3}
Y^{0}_k(\lambda)=v^{-\sigma_3/2} U^{0}_k(\lambda) 
= \Delta_0^*(I+O(\lambda)) \exp(-\tfrac 14 ie^{i\phi}t\lambda^{-1} \sigma_3)
\lambda^{\tfrac 12\theta_{0}\sigma_3}
\end{equation}
as $\lambda \to 0$ through the sector $|\arg \lambda -\phi -k\pi|<\pi,$
where
$$
\Delta_0^* =\begin{pmatrix} f_1 & f_2 \z \\  f_1 & f_2(\z-e^{i\phi}t/2)
\end{pmatrix},
\quad e^{i\phi}t f_1f_2\equiv -2.
$$
Then 
$$
Y^{\infty}_0(\lambda)=Y^0_0(\lambda)G, \quad
Y^{\infty}_1(\lambda)=Y^0_1(\lambda)\hat{G}, \quad
Y^{\infty}_{k+1}(\lambda)=Y^{\infty}_k(\lambda)S^{\infty}_k, \quad
Y^{0}_{k+1}(\lambda)=Y^{0}_k(\lambda)S^{0}_k. 
$$
\begin{prop}\label{prop3.7}
For $(Y^{\infty}_k(\lambda), Y^0_k(\lambda))$, system \eqref{3.1} has the 
same monodromy data $(G,\hat{G}, S^{\infty}_0, S^{\infty}_1, S^0_0, S^0_1)$
and the same monodromy manifold as of \eqref{1.2}.
\end{prop}
Then isomonodromy deformation of \eqref{3.1} may also be described as follows.
\begin{prop}\label{prop3.8}
The monodromy data of \eqref{3.1} for $(Y^{\infty}_k(\lambda),Y^0_k(\lambda))$
remain invariant under a small change of $te^{i\phi}$ if and only if $y^*=(d/dt)y$
holds in \eqref{3.2} and
$$
t \frac{d\z}{dt} =- 4y\z^2 + (2e^{i\phi}t y+ 1-2\theta_{\infty})\z +\frac 12
(\theta_0+\theta_{\infty})e^{i\phi}t,
$$
which is equivalent to \eqref{1.1}.
\end{prop}
Let $Y^{\infty, *}_k(\lambda)$ be the matrix solution of \eqref{3.1} such that
\begin{equation}\label{3.4}
Y^{\infty, *}_k(\lambda)=(I+O(\lambda^{-1})) \exp(\tfrac 14 ie^{i\phi}t\lambda\sigma_3)
\lambda^{-\tfrac 12\theta_{\infty}\sigma_3}
\end{equation}
as $\lambda\to\infty$ in the sector $|\arg\lambda +\phi -k\pi |<\pi,$ and
set
\begin{equation}\label{3.5}
Y^{\infty, *}_0(\lambda)=Y^0_0(\lambda)G^*, \quad
Y^{\infty, *}_1(\lambda)=Y^0_1(\lambda)\hat{G}^*.
\end{equation}
Then $Y^{\infty, *}_k(\lambda)=Y^{\infty}_k(\lambda)v^{\sigma_3/2},$ and
$Y^{\infty, *}_0(\lambda)=Y^0_0(\lambda)Gv^{\sigma_3/2},$
$Y^{\infty, *}_1(\lambda)=Y^0_1(\lambda)\hat{G}v^{\sigma_3/2}.$
\begin{prop}\label{prop3.9}
For $(Y^{\infty, *}_k(\lambda), Y^0_k(\lambda))$, system \eqref{3.1} has the
monodromy data
\begin{align*}
&(G^*,\hat{G}^*, S^{\infty, *}_0, S^{\infty, *}_1, S^{0, *}_0, S^{0, *}_1)
\\
=&(Gv^{\sigma_3/2},\hat{G}v^{\sigma_3/2}, 
v^{-\sigma_3/2}S^{\infty}_0v^{\sigma_3/2}, 
v^{-\sigma_3/2}S^{\infty}_1v^{\sigma_3/2}, S^{0}_0, S^{0}_1).
\end{align*}
\end{prop}
\begin{cor}\label{cor3.10}
Let $G^*=(g^*_{ij})$ and
$\hat{G}^*=(\hat{g}^*_{ij}).$ Then
\begin{align*}
& g^*_{11}g^*_{22}=g_{11}g_{22}, \quad
{g^*_{11}}/{g^*_{21}}={g_{11}}/{g_{21}}, \quad
{g^*_{12}}/{g^*_{22}}={g_{12}}/{g_{22}}, 
\\
& \hat{g}^*_{11}\hat{g}^*_{22}=\hat{g}_{11}\hat{g}_{22}, \quad
{\hat{g}^*_{11}}/{\hat{g}^*_{21}}={\hat{g}_{11}}/{\hat{g}_{21}}, \quad
{\hat{g}^*_{12}}/{\hat{g}^*_{22}}={\hat{g}_{12}}/{\hat{g}_{22}}. 
\end{align*}
\end{cor}
\section{WKB analysis}\label{sc4}
\subsection{Turning points and Stokes graphs}\label{ssc4.1}
Let us examine the characteristic roots $\pm \mu= \pm \mu(t, \lambda)$
of $\mathcal{B}(t,\lambda)$, the turning points, i.e. the roots of $\mu,$
and the Stokes graph, which are used in calculating monodromy data for
system \eqref{3.1}. The characteristic roots are given by
\begin{align}\label{4.1}
\mu^2=& b_1^2 +b_2^2 + b_3^2
\\
\notag
=& -e^{2i\phi}-4ie^{i\phi}\theta_{\infty} t^{-1} \lambda^{-1} 
+ e^{2i\phi} a_{\phi}\lambda^{-2} + 4ie^{i\phi}\theta_0 t^{-1}
\lambda^{-3} - e^{2i\phi}\lambda^{-4} 
\end{align}
with
\begin{equation}\label{4.2}
a_{\phi}=a_{\phi}(t) = e^{-2i\phi} ({y^*}y^{-1} +t^{-1} )^2
-y^2 - y^{-2} - 4e^{-i\phi}t^{-1} 
(\theta_0 y+\theta_{\infty} y^{-1} ).
\end{equation}
\par
First suppose that $\phi=0.$ The Boutroux equations \eqref{2.1} admit a unique 
solution $A_{0}=2$. Equation \eqref{1.1} has a two-parameter family of solutions
such that $y(x)=i +O(x^{-1/2})$ as $x \to +\infty$ along the positive real
axis. Taking these facts into account and viewing the supposition \eqref{5.1},
we set $a_0(t)= 2 +o(1)$, and then
$\mu(\infty,\lambda)^2|_{\phi=0} = -\lambda^{-4} (\lambda^2- 1)^2= -\lambda^{-4}
w(A_0,\lambda)^2.$
This means that, around $\phi=0$, $\mu(t,\lambda)$ admits four turning points
$\lambda_1,$ $\lambda_{-1},$ $\lambda_2,$ $\lambda_{-2}$ 
such that, for $\phi=0$, $\lambda_1$ and $\lambda_2$ (respectively, 
$\lambda_{-1}$ and $\lambda_{-2}$) coalesce at $1$ (respectively, at $-1$)
as $t \to \infty$. 
Then $\re \lambda_1 \le \re \lambda_2$ for $\phi$ close to $0$, and that 
$\lambda_{-1} \to -\lambda_1,$ $\lambda_{-2} \to -\lambda_2$ as $t \to\infty.$
Here we use the same letters $\lambda_{1},$ $\lambda_2$  
to denote the turning points and the limit turning points as $t\to\infty$. 
By Proposition \ref{prop8.15} 
the Boutroux equations admit unique solutions $A_{\pm \pi/2}=-2$, 
and $\mu(\infty,\lambda)^2=-e^{2i\phi}
(1-a_{\phi}\lambda^{-2}+\lambda^{-4})$ does not
degenerate for $0<|\phi|<\pi/2.$
\par
Suppose that $0<|\phi|<\pi/2.$
The algebraic function $\mu(t,\lambda)$ is considered on the two-sheeted
Riemann surface $\mathcal{R}_{\phi}$ glued along the cuts $[\lambda_{-2},
\lambda_{-1}]$ and $[\lambda_1,\lambda_2].$ Let
\begin{align*}
\mu(t,\lambda) &= ie^{i\phi} \lambda^{-2}
\sqrt{1- a_{\phi}\lambda^2 +\lambda^4 -4i e^{-i\phi}\theta_0 \lambda t^{-1}
+4i e^{-i\phi}\theta_{\infty} \lambda^3 t^{-1} }
\\
 &= ie^{i\phi} \lambda^{-2}
\sqrt{(1-\lambda_1^{-1}\lambda)(1-\lambda_2^{-1}\lambda)
(1-\lambda_{-1}^{-1}\lambda)(1-\lambda_{-2}^{-1}\lambda)}  
 \\
&= ie^{i\phi} \lambda^{-2}
\sqrt{1-\lambda_1^{-1}\lambda} \, \sqrt{1-\lambda_2^{-1}\lambda}\,
\sqrt{1-\lambda_{-1}^{-1}\lambda} \, \sqrt{1-\lambda_{-2}^{-1}\lambda},  
\end{align*}
in which the branches of the square roots are fixed in such a way that
$\sqrt{1-\lambda_{\pm j}^{-1}\lambda} \to 1$ as $\lambda \to 0$ on 
the upper sheet
of $\mathcal{R}_{\phi}.$ Then $\lambda^2 \mu(t,\lambda) \to ie^{i\phi}$ as
$\lambda \to 0,$ and $\mu(t,\lambda)\to -ie^{i\phi}$ as $\lambda \to \infty$
on the upper sheet of $\mathcal{R}_{\phi}.$
{\small
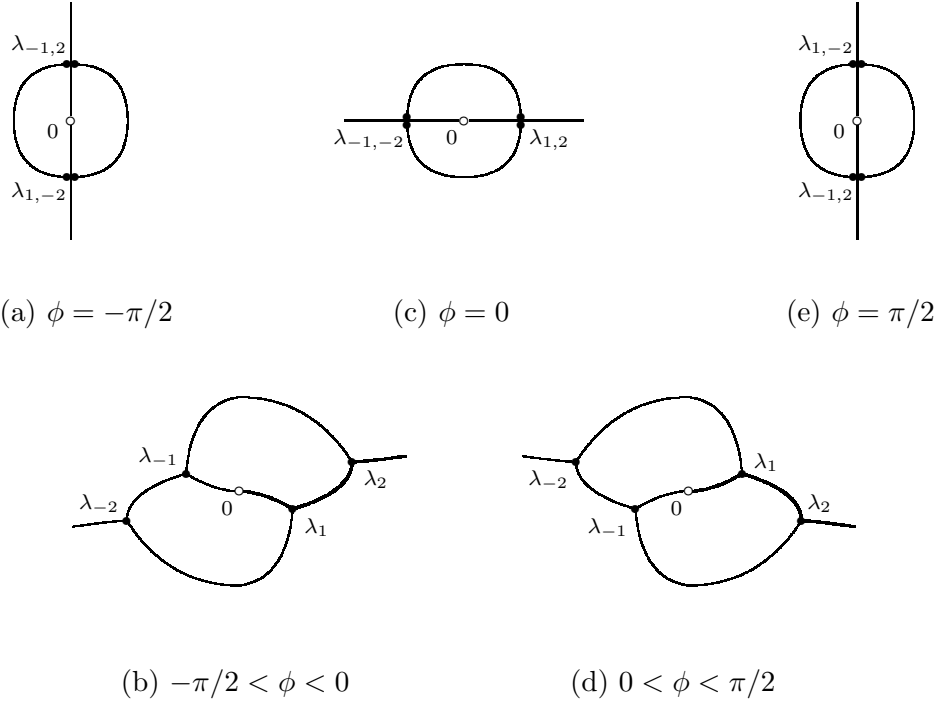
\begin{figure}[htb]
\begin{center}
\unitlength=0.78mm
\begin{picture}(40,50)(-20,-30)
\put(0,0){\circle{1.5}}
\put(0.7,9.6){\circle*{1.5}}
\put(-0.7,9.6){\circle*{1.5}}
\put(0.7,-9.6){\circle*{1.5}}
\put(-0.7,-9.6){\circle*{1.5}}

{\tiny
\put(-10,12.0){\makebox{$\lambda_{-1,2}$}}
\put(-10,-13){\makebox{$\lambda_{1,-2}$}}
\put(-4,-3){\makebox{$0$}}
}

\qbezier (0,-9.6) (0,-14) (0,-20)
\qbezier (0,9.6) (0,14) (0,20)

\qbezier (0,-1) (0,-8) (0,-15)
\qbezier (0,1) (0,8) (0,15)

\qbezier (9.6,0) (9.6,9.6)  (0,9.6)
\qbezier (-9.6,0) (-9.6,9.6)  (0,9.6)

\qbezier (9.6,0) (9.6,-9.6)  (0,-9.6)
\qbezier (-9.6,0) (-9.6,-9.6)  (0,-9.6)

\put(-12,-34){\makebox{(a) $\phi=-\pi/2$}}
\end{picture}
\qquad\qquad\quad
\begin{picture}(40,50)(-20,-30)
\put(0,0){\circle{1.5}}
\put(9.6,0.7){\circle*{1.5}}
\put(9.6,-0.7){\circle*{1.5}}
\put(-9.6,0.7){\circle*{1.5}}
\put(-9.6,-0.7){\circle*{1.5}}

{\tiny
\put(11.0,-4){\makebox{$\lambda_{1,2}$}}
\put(-22,-4){\makebox{$\lambda_{-1,-2}$}}
\put(-3,-4){\makebox{$0$}}
}

\qbezier (-9.6,0) (-14,0) (-20, 0)
\qbezier (9.6,0) (14,0) (20, 0)

\qbezier (-1,0) (-8,0) (-15, 0)
\qbezier (1,0) (8,0) (15, 0)

\qbezier (0,9.6) (9.6,9.6)  (9.6,0)
\qbezier (0,-9.6) (9.6,-9.6)  (9.6,0)

\qbezier (0,9.6) (-9.6,9.6)  (-9.6,0)
\qbezier (0,-9.6) (-9.6,-9.6)  (-9.6,0)

\put(-12,-34){\makebox{(c) $\phi=0$}}
\end{picture}
\qquad\qquad\quad
\begin{picture}(40,50)(-20,-30)

\put(0,0){\circle{1.5}}
\put(0.7,9.6){\circle*{1.5}}
\put(-0.7,9.6){\circle*{1.5}}
\put(0.7,-9.6){\circle*{1.5}}
\put(-0.7,-9.6){\circle*{1.5}}

{\tiny
\put(-10,12.0){\makebox{$\lambda_{1,-2}$}}
\put(-10,-13){\makebox{$\lambda_{-1,2}$}}
\put(-4,-3){\makebox{$0$}}
}

\qbezier (0,-9.6) (0,-14) (0,-20)
\qbezier (0,9.6) (0,14) (0,20)

\qbezier (0,-1) (0,-8) (0,-15)
\qbezier (0,1) (0,8) (0,15)

\qbezier (9.6,0) (9.6,9.6)  (0,9.6)
\qbezier (-9.6,0) (-9.6,9.6)  (0,9.6)

\qbezier (9.6,0) (9.6,-9.6)  (0,-9.6)
\qbezier (-9.6,0) (-9.6,-9.6)  (0,-9.6)

\put(-12,-34){\makebox{(e) $\phi=\pi/2$}}
\end{picture}
\vskip0.2cm

\begin{picture}(70,60)(-35,-30)
\put(0,0){\circle{1.5}}
\put(-9,3){\circle*{1.5}}
\put(9,-3){\circle*{1.5}}
\put(-19,-5){\circle*{1.5}}
\put(19,5){\circle*{1.5}}

\qbezier (-9,3) (-5,0.5) (-1, 0)
\qbezier (-19,-5) (-22,-5.1) (-28,-6)

\qbezier  (9,-3) (8, -15) (0,-16)
\qbezier  (0,-16) (-12, -16) (-19,-5)

\qbezier  (-9,3) (-8, 15) (0,16)
\qbezier  (0,16) (12, 16) (19,5)

\qbezier (-9,3) (-19,0) (-19,-5)

\thicklines
\qbezier (9,-3) (5,-0.5) (0.8, 0)
\qbezier (9,-3) (19,0) (19,5)
\qbezier (19,5) (22,5.1) (28,6)

{\tiny
\put(11,-7){\makebox{$\lambda_1$}}
\put(21,1){\makebox{$\lambda_2$}}
\put(-3,-4){\makebox{$0$}}
\put(-17,5){\makebox{$\lambda_{-1}$}}
\put(-27,-3){\makebox{$\lambda_{-2}$}}
}
\put(-20,-34){\makebox{(b) $-\pi/2 < \phi < 0$}}

\end{picture}
\,\,\,
\begin{picture}(70,60)(-35,-30)
\put(0,0){\circle{1.5}}
\put(-9,-3){\circle*{1.5}}
\put(9,3){\circle*{1.5}}
\put(-19,5){\circle*{1.5}}
\put(19,-5){\circle*{1.5}}

\qbezier (-9,-3) (-5,-0.5) (-1, 0)
\qbezier (-19,5) (-22,5.1) (-28,6)

\qbezier  (-9,-3) (-8, -15) (0,-16)
\qbezier  (0,-16) (12, -16) (19,-5)

\qbezier  (9,3) (8, 15) (0,16)
\qbezier  (0,16) (-12, 16) (-19,5)

\qbezier (-9,-3) (-19,0) (-19,5)

\thicklines
\qbezier (9,3) (5,0.5) (0.8, 0)
\qbezier (9,3) (19,0) (19,-5)
\qbezier (19,-5) (22,-5.1) (28,-6)

{\tiny
\put(11,4){\makebox{$\lambda_1$}}
\put(20,-3){\makebox{$\lambda_2$}}
\put(-3,-4){\makebox{$0$}}
\put(-17,-7){\makebox{$\lambda_{-1}$}}
\put(-27,1){\makebox{$\lambda_{-2}$}}
}
\put(-20,-34){\makebox{(d) $0 < \phi < \pi/2$}}
\end{picture}

\end{center}
\caption{Limit Stokes graphs for $|\phi|\le \pi/2$}
\label{stokes}
\end{figure}
}
\par
The Stokes graph consists of the Stokes curves and the vertices: each Stokes
curve is defined by $\re \int^{\lambda}_{\lambda_*} \mu(\lambda) d\lambda=0 $
with a turning point $\lambda_*$, and the vertices are turning points 
or singular points $\lambda=0, \infty.$ 
The Stokes graph lies on $\mathcal{R}_{\phi}$.
The limit Stokes graph with $t=\infty$ for the isomonodromy system 
\eqref{3.1} is considered
to reflect the Boutroux equations \eqref{2.1}.
When $\phi$ increases or decreases, the limit turning points for
$\lambda_1$ and $\lambda_2$ move according to the solution 
$A_{\phi}$ of the Boutroux equations \eqref{2.1}. By Proposition \ref{prop8.16}, 
for $\phi$ close to $0$, the double turning point at $\lambda_1=\lambda_2=1$ 
is resolved into two simple turning points $\lambda_1$ and $\lambda_2$  
such that $\im \lambda_1 > 0 >\im \lambda_2,$ $\re \lambda_1 < 1 < \re 
\lambda_2$ if $\phi>0,$ and such that $\im \lambda_1 <0 <\im \lambda_2,$
$\re \lambda_1 < 1 < \re \lambda_2$ if $\phi<0$. For $0<|\phi|<\pi/2$ 
coalescence of turning points does not occur, and
then the topological figure of the limit Stokes graph remains invariant. 
For $-\pi/2 <\phi< 0$ and $0<\phi<\pi/2$, the limit Stokes graphs 
are as in Figures \ref{stokes} (b) and (d), in which each limit turning point
with $t=\infty$ is also denoted by $\lambda_{\iota}$ or $\lambda_{-\iota}
=-\lambda_{\iota}.$ 
In our calculation, for $0<|\phi|<\pi/2,$ we use the Stokes curves from $0^+$ to
$\infty^-$ passing through $\lambda_1$ and $\lambda_2$ and passing through
$-\lambda_1$ to $-\lambda_2,$ where $0^+$ and $\infty^-$ denote $0$ and
$\infty$ on the upper and lower sheets, respectively.  
This ranges on both upper and lower sheets of $\mathcal{R}_{\phi}.$
For a technical reason,
the cut $[\lambda_1,\lambda_2]$ on the upper sheet is made in such a way 
that the Stokes curve 
$(\lambda_1,\lambda_2)^{\sim}$ is located along the lower shore 
(respectively, the upper shore) of the cut if $0<\phi<\pi/2$ 
(respectively, $-\pi/2<\phi <0$), and the cut $[\lambda_{-2}, \lambda_{-1}]$
in such a way that 
the Stokes curve 
$(-\lambda_1,-\lambda_2)^{\sim}$ is located along the upper shore 
(respectively, the lower shore) of the cut if $0<\phi<\pi/2$ 
(respectively, $-\pi/2<\phi <0$)  
(cf. Figures \ref{curve+}, \ref{curve+*}, \ref{curve-}).
\par
An unbounded
domain $D \subset \mathcal{R}_{\phi}$ is called a canonical domain if, for each
$\lambda \in D$, there exist contours $C_{\pm}(\lambda) \subset D$ terminating
in $\lambda$ such that
$$
\re \int^{\lambda}_{\lambda^-_0} \mu(\lambda) d\lambda \to -\infty \quad
\biggl(\text{respectively,} \,\,\,
\re \int^{\lambda}_{\lambda^+_0} \mu(\lambda) d\lambda \to +\infty \,\, \biggr)
$$
as $\lambda^-_0 \to \infty$ along $C_-(\lambda)$ (respectively, as $\lambda^+_0
\to \infty$ along $C_+(\lambda)$) (see \cite{F}, \cite[p. 242]{FIKN}).
The interior of a canonical domain contains exactly one Stokes curve, and its
boundary consists of Stokes curves.
\subsection{WKB solution}\label{ssc4.2}
The following WKB solution will be used in our calculus.
\begin{prop}\label{prop4.1}
In the canonical domain $(\subset \mathcal{R}_{\phi})$ 
whose interior contains a Stokes curve issueing from
the turning point $\lambda_{\pm 1}$ or $\lambda_{\pm 2}$, 
system \eqref{3.1} with 
$\mathcal{B}(\lambda,t)=b_1\sigma_1+b_2\sigma_2+b_3\sigma_3$ admits 
an asymptotic solution expressed as
$$
\Psi_{\mathrm{WKB}}(\lambda) = T(I+O(t^{-\delta})) \exp\Bigl(
\int^{\lambda}_{\tilde{\lambda}_*} \Lambda(\tau) d\tau \Bigr), \quad
T = \begin{pmatrix} 1 & \frac{b_3-\mu}{b_1+ ib_2} \\
  \frac{\mu- b_3}{b_1- ib_2}  &  1    \end{pmatrix}
$$
outside suitable neighbourhoods of zeros of $b_1 \pm i b_2$ as long as
$|\lambda -\lambda_{\iota}| \gg t^{-2/3 +(2/3)\delta}$ $(\iota=\pm 1, \pm 2)$. 
Here
$\delta$ is an arbitrary number such that $0<\delta <1$, $\tilde{\lambda}_*$ 
is a base point near $\lambda_{\pm 1}$ or $\lambda_{\pm 2}$, and  
$$
\Lambda(\lambda)=\frac t4 \mu(t,\lambda) \sigma_3 -\diag T^{-1}T_{\lambda}.
$$
\end{prop}
\begin{rem}\label{rem4.1}
In the proposition above
\begin{align*}
\diag T^{-1}T_{\lambda} &= \frac 1{2\mu(\mu+b_3)} (i(b_1b_2' -b_1'b_2)\sigma_3
+ (b_3 \mu' - b_3'\mu) I)
\\
&= \frac 14 \Bigl(1-\frac{b_3}{\mu} \Bigr) \frac{\partial}{\partial\lambda}
\log\frac{b_1+i b_2}{b_1 - ib_2} \sigma_3 +\frac 12 \frac{\partial}{\partial
\lambda} \log \frac{\mu}{\mu+ b_3} I,
\end{align*}
where $b_1'=(\partial/\partial\lambda)b_1$.
\end{rem}
For the Schr\"odinger equation $(-\hbar^2(d/d\lambda)^2 +p(\lambda))y=0$
with Planck's constant $\hbar$, the WKB solution has the form
$(1+O(\hbar))p(\lambda)^{-1/4} \exp(\pm \hbar^{-1} \int^{\lambda}p(\lambda)^{1/2}
d\lambda)$. On the other hand, by $Y=T\tilde{Y}$ system \eqref{3.1} is taken to
$\tilde{Y}_{\lambda}=(\tfrac 14 t \mu(t,\lambda)\sigma_3 -T^{-1}T_{\lambda})
\tilde{Y}$
formally admitting a matrix solution with the leading term
$\exp(\tfrac 14 t \int^{\lambda}\mu(t,\lambda)d\lambda \sigma_3)$ as $t\to
+\infty$ $(\arg t\to 0).$ Hence our solution in Proposition \ref{prop4.1}
is the WKB solution with the perturbation variable $4t^{-1}$ 
in place of Planck's constant $\hbar.$ We make a further transform of the form
$\tilde{Y}=(I+T_1)(I+X_1)Z$ with $T_1$ such that $[\tfrac 14 t\mu\sigma_3, T_1]
=T^{-1}T_{\lambda}-\diag T^{-1}T_{\lambda}$ as in the proof of \cite
[Proposition 3.8]{Shimom}. Suitable choice of $X_1$ reduces system \eqref{3.1}
to $Z_{\lambda}=(\tfrac 14 t\mu(t,\lambda)\sigma_3-\diag T^{-1}T_{\lambda})Z$, 
from which our WKB solution follows. Then we additionally use the fact: 
$\mu =\mp i e^{i\phi} +O(\lambda^{-2})$ near $\lambda=\infty^{\pm},$ 
and $= \pm i e^{i\phi}\lambda^{-2}+O(1)$ near $\lambda=0^{\pm}$
on $\mathcal{R}_{\phi}$, where $\infty^{+}$, $0^+$ (respectively, $\infty^-$,
$0^-$) denote $\infty,$ $0$ on the upper sheet 
(respectively, lower sheet) of $\mathcal{R}
_{\phi}$ (for details see the proof of \cite[Proposition 3.8]{Shimom} or
\cite[Theorem 7.2]{FIKN}, see also \cite{F}).
\subsection{Local solution around a turning point}\label{ssc4.3}
Near turning points the WKB solution above fails in expressing its asymptotic 
behaviour. In the neighbourhood of $\lambda_{\iota}$ system \eqref{3.1} 
is reduced to
\begin{equation}\label{4.3}
\frac{dW}{d\zeta}=\begin{pmatrix} 0 & 1 \\ \zeta & 0 \end{pmatrix} W,
\end{equation}
which has the solutions ${}^T(\mathrm{Ai}(\zeta), 
\mathrm{Ai}_{\zeta}(\zeta))$, ${}^T(\mathrm{Bi}(\zeta), 
\mathrm{Bi}_{\zeta}(\zeta))$ with the Airy function $\mathrm{Ai}(\zeta)$ and 
$\mathrm{Bi}(\zeta) = e^{-\pi i/6}\mathrm{Ai}(e^{-2\pi i/3}\zeta)$
\cite{AS}, \cite{HTF}.
Then we have the following solution near each simple turning point
\cite[Theorem 7.3]{FIKN}, \cite[Proposition 3.9]{Shimom}.
\begin{prop}\label{prop4.2}
For each simple turning point $\lambda_{\iota}$ $(\iota=\pm 1,\pm 2)$ write
$c_k=b_k(\lambda_{\iota}),$ $c'_k=(b_k)_{\lambda}(\lambda_{\iota})$ $(k=1,2,3)$,
and suppose that $c_k,$ $c_k'$ are bounded and $c_1 \pm i c_2\not=0.$ Let
$\hat{t}=2(2\kappa_c)^{-1/3}(c_1-ic_2)(t/4)^{1/3}$ with
$\kappa_c=c_1c_1' +c_2c'_2 +c_3c'_3$. Then system
\eqref{3.1} admits a matrix solution of the form 
\begin{equation*}
 \Phi_{\iota}(\lambda) =T_{\iota} (I+O(t^{-\delta'})) \begin{pmatrix}
1 & 0 \\ 0 & \hat{t}^{-1}    \end{pmatrix} W(\zeta) ,
\quad
 T_{\iota}= \begin{pmatrix}   1 &  -\frac{c_3}{c_1+ ic_2}  \\
- \frac{c_3}{c_1 -ic_2}  &  1    \end{pmatrix},
\end{equation*}
in which
$\lambda-\lambda_{\iota} =(2\kappa_c)^{-1/3}(t/4)^{-2/3} (\zeta+\zeta_0)$ with
$|\zeta_0| \ll t^{-1/3},$ 
as long as $|\zeta| \ll t^{(2/3-\delta')/3},$ that is, $|\lambda-\lambda_{\iota}
| \ll t^{-2/3+(2/3-\delta')/3}.$
Here $\delta'$ is an arbitrary number such that $0<\delta'< 2/3$, and
$W(\zeta)$ solves system \eqref{4.3}, which admits canonical solutions 
$W_{\nu}(\zeta)$ $(\nu \in \mathbb{Z})$ such that
$$
W_{\nu}(\zeta)=\zeta^{-(1/4)\sigma_3}(\sigma_3+\sigma_1)(I+O(\zeta^{-3/2}))
\exp((2/3)\zeta^{3/2}\sigma_3)
$$
as $\zeta \to \infty$ through the sector $|\arg \zeta-(2\nu-1)\pi/3|<2\pi/3,$
and that $W_{\nu+1}(\zeta)=W_{\nu}(\zeta)S_{\nu}$ with
$$
S_1 =\begin{pmatrix} 1 & -i \\ 0 & 1 \end{pmatrix}, \quad 
S_2 =\begin{pmatrix} 1 & 0 \\ -i & 1 \end{pmatrix}, \quad 
S_{\nu+1}=\sigma_1 S_{\nu} \sigma_1.
$$
\end{prop}
\begin{rem}\label{rem4.2}
Putting $\lambda-\lambda_{\iota}=(2\kappa_c)^{-1/3}(e^{2\pi i/3})^{2j}
(t/4)^{-2/3}(\zeta+\zeta_0),$ $j\in \{0,\pm 1\},$ we have an expression of
$\Phi_{\iota}(\lambda)$ with $\hat{t}=2(2\kappa_c)^{-1/3}(e^{2\pi i/3})^{2j}
(c_1-ic_2)(t/4)^{1/3}.$
\end{rem}
\section{Calculation of the connection matrices}\label{sc5}
To get necessary information on the connection matrices $G=(g_{ij})$ and 
$\hat{G}=(\hat{g}_{ij})$ (cf. Sections \ref{ssc2.1} and \ref{ssc3.1})  
we calculate $G^*=(g^*_{ij})$ and $\hat{G}^*=(\hat{g}^*_{ij})$ such that
$$
  Y^{\infty, *}_0(\lambda)=Y^0_0(\lambda)G^*, \quad
  Y^{\infty, *}_1(\lambda)=Y^0_1(\lambda)\hat{G}^* 
$$
(cf. \eqref{3.5}, Proposition \ref{prop3.9}) as a solution of the direct
monodromy problem by applying WKB analysis to system \eqref{3.1}. 
Suppose that $a_{\phi}(t)$ is given by \eqref{4.2} with a pair of
arbitrary functions $(y,y^*)=(y(t),y^*(t))$ not necessarily solving 
\eqref{1.1} with $2x=\xi=te^{i\phi}$, and that 
\begin{equation}\label{5.1}
a_{\phi}(t)=A_{\phi}+ \frac{B_{\phi}(t)}{t}, \quad B_{\phi}(t) \ll 1
\end{equation}
for $t \in S_{\phi}(t'_{\infty},\kappa_1,\delta_1)$ with given  
$\kappa_1>0$, given small $\delta_1>0$ and sufficiently large $t'_{\infty}>0$. 
Here $A_{\phi}$ is a solution of the Boutroux equations \eqref{2.1}, and
$$
S_{\phi}(t'_{\infty}, \kappa_1, \delta_1)=\{t\,|\,\, \re t > t'_{\infty},
\,\, |\im t| < \kappa_1, \,\, |y(t)|+ |y^*(t)|+|y(t)|^{-1} 
< \delta_1^{-1} \}.
$$
\par
Let $0<\phi<\pi/2.$ We calculate the analytic continuation of the 
matrix solution $Y^{\infty,*}_0(\lambda)$ near $\lambda=\infty$ 
along the Stokes curve consisting of 
$$
\mathbf{c}_{\infty}=(\infty^-, \lambda_2)^{\sim}, \quad
\mathbf{c}_{1}=(\lambda_2,\lambda_1)^{\sim}, \quad
\mathbf{c}_{0}=(\lambda_1,0^+)^{\sim}
$$
starting from $\infty^-$ and terminating in $0^+$ as in Figure \ref{curve+} to
get the connection matrix $G^*$, where $\infty^-$ and $0^+$ denote $\infty$
and $0$ on the lower and upper
sheets of $\mathcal{R}_{\phi}$, respectively. Recall
that the Stokes curve is considered on the two-sheeted Riemann surface
$\mathcal{R}_{\phi}$ of $\mu(t,\lambda)$, and that the curve $\mathbf{c}_1$ 
is located along the lower shore of the cut 
$[\lambda_1,\lambda_2]$. The curve $\mathbf{c}_{\infty}$ is on the lower 
sheet of $\mathcal{R}_{\phi}$, and $\mathbf{c}_0$ and $\mathbf{c}_1$ are 
on the upper sheet of $\mathcal{R}_{\phi}$.
Under supposition \eqref{5.1} these curves $\mathbf{c}_0,$ $\mathbf{c}_1,$
$\mathbf{c}_{\infty}$ lie within the distance $\ll t^{-1}$ from the limit
Stokes graph in Figure \ref{stokes}.
\par
In the WKB solution, write $\Lambda(\lambda)$ in the component-wise form
$ \Lambda(\lambda) =\Lambda_3(\lambda) + \Lambda_I(\lambda)$ with
$$
\Lambda_3(\lambda) =\frac t4 \mu(t,\lambda) \sigma_3 -\diag T^{-1}T_{\lambda}
|_{\sigma_3} \sigma_3,
\quad \Lambda_I(\lambda)= -\diag T^{-1}T_{\lambda}|_{I} I,
$$
in which $\diag T^{-1}T_{\lambda}|_{\sigma_3}\sigma_3  \in \mathbb{C}\sigma_3$,
$\diag T^{-1}T_{\lambda}|_{I}I \in \mathbb{C} I$. 
Denote by $\mu_-(t,\lambda)$ the branch of $\mu(t,\lambda)$ on the lower
sheet of $\mathcal{R}_{\phi}$ and by $\Lambda_-(\lambda),$ $\Lambda_{3-}
(\lambda),$ $\Lambda_{I-}(\lambda)$ and $T_-$ those related to 
$\mu_-(t,\lambda).$
In Propositions \ref{prop4.1} and
\ref{prop4.2}, if $\delta=\delta'=2/9-\varepsilon$ with any $\varepsilon $
such that $0<\varepsilon<2/9$,
then both propositions are applicable in the annulus
$$
\mathcal{A}^{\iota}_{\varepsilon}: \quad t^{-2/3+(2/3)(2/9-\varepsilon)}
\ll |\lambda -\lambda_{\iota}| \ll t^{-2/3 +(2/3)(2/9+\varepsilon/2)}
$$
$(\iota=\pm 1,\pm 2)$. In what follows we set $\delta=2/9 -\varepsilon,$ and
write $c_k=b_k(\lambda_1)$, $d_k=b_k(\lambda_2)$ $(k=1,2,3)$.
{\small
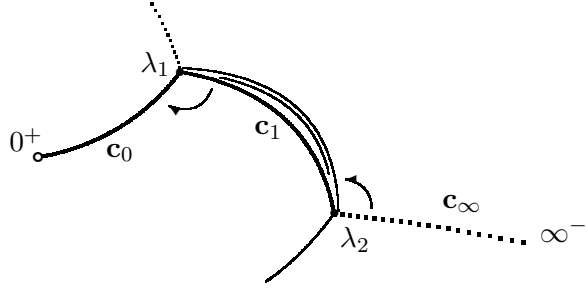
\begin{figure}[htb]
\begin{center}
\unitlength=0.75mm
\begin{picture}(100,50)(0,-25)
\put(-5,1){\makebox{$0^+$}}
\put(18,15){\makebox{$\lambda_1$}}
\put(53,-16){\makebox{$\lambda_2$}}
\put(88,-15){\makebox{$\infty^-$}}
\put(12,0){\makebox{$\mathbf{c}_0$}}
\put(38,4){\makebox{$\mathbf{c}_1$}}
\put(71,-9){\makebox{$\mathbf{c}_{\infty}$}}

\thinlines
\qbezier(32,14) (48,10.5) (51,-3)
\qbezier(25.5,15.7) (52,14) (52.7,-9.5)

  \qbezier(58.5,-9) (58.5,-5) (54,-4)
\put(54,-4){\vector(-1,0){0}}

  \qbezier(30.5,12) (28.5,7) (23,9)
\put(23,9){\vector(-3,2){0}}

\qbezier(40,-22) (46,-18) (52,-10)
\qbezier(40,-22.06) (46,-18.06) (52,-10.06)

\thicklines
\put(0,0){\circle{1.5}}
\put(25,15){\circle*{1.5}}
\put(52,-10){\circle*{1.5}}
\qbezier(0.8,0) (16,3) (25,15)
\qbezier(0.8,0.08) (16,3.08) (25,15.08)
\qbezier(25,15) (48,11) (52,-10)
\qbezier(24.92,15) (47.92,11) (51.92,-10)

\qbezier[12](25,15) (24,20) (20,27)
\qbezier[20](85,-15) (65,-11) (52,-10)
\qbezier[20](85.2,-15.3) (65.2,-11.3) (52.2,-10.3)
\qbezier[20](85.0,-15.3) (65.0,-11.3) (52.0,-10.3)
\qbezier[20](85.2,-15) (65.2,-11) (52.2,-10)

\end{picture}
\end{center}
\caption{Curves $\mathbf{c}_0,$ $\mathbf{c}_1,$ $\mathbf{c}_{\infty}$ 
for $0<\phi <\pi/2$}
\label{curve+}
\end{figure}
}

{\bf (1)} Let $\Psi_{\infty}(\lambda)$ along $\mathbf{c}_{\infty}
=(\infty^-, \lambda_2)^{\sim}$ be a WKB solution by Proposition \ref{prop4.1}, and
let ${Y}^{\infty, *}_0(\lambda)$ be given by \eqref{3.4}.
Set
${Y}^{\infty, *}_0(\lambda)=\Psi_{\infty}(\lambda) \Gamma_{\infty}.$ Note that
the branch of $\mu(t,\lambda)$ along $\mathbf{c}_{\infty}$ is
$\mu_-(t,\lambda)= i e^{i\phi} (1 +2i e^{-i\phi}\theta_{\infty}
 t^{-1}\lambda^{-1} +O(\lambda^{-2}))$, and
$\mu_- -b_3 \ll \lambda^{-1}$ as $\lambda \to \infty^-.$ Then
\begin{align*}
\Gamma_{\infty} =& 
 \Psi_{\infty}(\lambda)^{-1} {Y}^{\infty,*}_0(\lambda) 
\\
=& \exp\Bigl(-\int^{\lambda}_{\tilde{\lambda}_2} \Lambda_-(\tau) d\tau\Bigr)
T^{-1}_- (I+O(t^{-\delta}+|\lambda|^{-1})) 
\\
&\phantom{---}\times \exp \Bigl(\frac 14 ( i e^{i\phi}
t\lambda - 2\theta_{\infty} \log \lambda)\sigma_3 \Bigr)
\\
=& C_3(\tilde{\lambda}_2) c_I(\tilde{\lambda}_2)(I+O(t^{-\delta}))
\\
&\phantom{---}\times \exp\biggl( -\lim_{\substack {\lambda\to\infty^- \\ 
\lambda\in \mathbf{c}_{\infty}} }  
\Bigl(\int^{\lambda}_{\lambda_2} \Lambda_{3-}(\tau)
d\tau -\frac 14 (ie^{i\phi}t\lambda -2\theta_{\infty}\log\lambda)\sigma_3 \Bigr)
\biggr),
\end{align*}
in which $C_3(\tilde{\lambda}_2)=\exp (\int^{\tilde{\lambda}_2}_{\lambda_2}
\Lambda_{3-}(\tau)d\tau )$, $c_I(\tilde{\lambda}_2) =\exp (-\int^{\infty^-}
_{\tilde{\lambda}_2} \Lambda_{I-}(\tau) d\tau ),$ and $\tilde{\lambda}_2 \in
\mathbf{c}_{\infty},$ $ \tilde{\lambda}_2-\lambda_2 \asymp t^{-1}.$
\par
{\bf (2)} For $\Psi_{\infty}(\lambda)$ and for $\Phi^+_2(\lambda)$ given by
Proposition \ref{prop4.2} in the annulus $\mathcal{A}^2_{\varepsilon}$
around $\lambda_2$, set $\Psi_{\infty}(\lambda)=\Phi^+_2(\lambda)\Gamma_{2+}$
along $\mathbf{c}_{\infty}.$
Suppose that the curve $(2\kappa_d)^{1/3}(\lambda-\tilde{\lambda}_2)=(t/4)^{-2/3}
(\zeta+O(t^{-1/3})),$ $\kappa_d=d_1d_1'+d_2d_2'+d_3d_3'$ with 
$\lambda \in \mathbf{c}_{\infty}$ enters the sector
$|\arg \zeta -\pi/3 |<2\pi/3$ (the other cases are similarly treated by
Remark \ref{rem4.2}). Write $K^{-1}=2(2\kappa_d)^{-1/3}(d_1-i d_2).$ Then, by
Propositions \ref{prop4.1} and \ref{prop4.2},
\begin{align*}
\Gamma_{2+} =& \Phi^+_2(\lambda)^{-1} \Psi_{\infty}(\lambda)
\\
=& W(\zeta)^{-1} \begin{pmatrix} 1 & 0 \\  0 & (t/4)^{-1/3}K \end{pmatrix}
^{\!\!-1}
(I+O(t^{-\delta}))  \begin{pmatrix}  1 &  -\frac{d_3}{d_1+i d_2} \\
-\frac{d_3}{d_1-i d_2}  & 1  \end{pmatrix}^{\!\!-1}
\\
&\phantom{----} \times \begin{pmatrix}  1  & \frac{b_3 - \mu_-}{b_1 +i b_2} \\
\frac{\mu_- -b_3}{b_1- i b_2}  &  1  \end{pmatrix}  (I+O(t^{-\delta}))
\exp \Bigl(\int^{\lambda}_{\tilde{\lambda}_2} \Lambda_-(\tau)d\tau \Bigr)
\\
=& W(\zeta)^{-1} \begin{pmatrix}  1  &  \frac{d_3}{d_1+i d_2} \\
\frac{(t/4)^{1/3}\mu_-} {2K(d_1-i d_2)} & \frac{(t/4)^{1/3}\mu_-}{2K d_3}
\end{pmatrix}  (I+O(t^{-\delta})) \exp\Bigl(\int^{\lambda}_{\tilde{\lambda}_2}
\Lambda_-(\tau) d\tau \Bigr)
\end{align*}
for $\lambda \in \mathcal{A}^2_{\varepsilon} \cap \mathbf{c}_{\infty},$ where
$(\mu_- -b_3)/(b_1 \pm ib_2)=(\mu_- -d_3)/(d_1 \pm i d_2)+O(\eta)$, $\eta=
\lambda-\tilde{\lambda}_2$. Since
$
\mu_- =-(2\kappa_d)^{1/2}\eta^{1/2}(1+O(\eta)) 
 =-2K(d_1-i d_2)(t/4)^{-1/3}\zeta^{1/2}(1+O(\eta)),
$
we have
\begin{align*}
\Gamma_{2+}= & \exp\Bigl(-\frac 23 \zeta^{3/2}\sigma_3\Bigr)
(\sigma_3+\sigma_1)^{-1}
\zeta^{(1/4)\sigma_3}\begin{pmatrix} 1 & 0 \\ 0 & \zeta^{1/2} \end{pmatrix}
\\
& \phantom{---}\times
\begin{pmatrix} 1 & -\frac{d_1-i d_2}{d_3} \\ -1 & -\frac{d_1-i d_2}{d_3}
\end{pmatrix}(I+O(t^{-\delta}))\exp\Bigl(\int^{\lambda}_{\tilde{\lambda}_2}
\Lambda_-(\tau)d\tau \Bigr)
\\
=& \exp\Bigl(\int^{\lambda}_{\tilde{\lambda}_2} \sigma_1 \Lambda_-(\tau) 
\sigma_1 d\tau -\frac 23 \zeta^{3/2}\sigma_3\Bigr)\zeta^{1/4}(I+O(t^{-\delta}))
\begin{pmatrix}  0 & -\frac{d_1-i d_2}{d_3} \\ 1 & 0 \end{pmatrix}.
\end{align*}
By $\sigma_1 \Lambda_{3-}(\lambda) \sigma_1 =((2\kappa_d)^{1/2}(t/4)\eta^{1/2}(1+O(\eta))
+O(\eta^{-1/2}))\sigma_3$ and $\Lambda_{I-}(\lambda)=(-{\eta}^{-1}/4+
O(\eta^{-1/2}) )I$ (cf. Remark \ref{rem4.1}) for $\eta=\lambda
-\tilde{\lambda}_2,$ $\lambda\in \mathcal{A}^2_{\varepsilon} \cap \mathbf{c}
_{\infty},$ it follows that
$$
\Gamma_{2+} = (\tilde{\zeta}_2)^{1/4}(I+O(t^{-\delta})) C_3(\tilde{\lambda}_2)
^{-1} \begin{pmatrix}  0 & -\frac{d_1-i d_2}{d_3} \\  1 &  0 \end{pmatrix}
$$
with suitably chosen $\tilde{\zeta}_2 \asymp \tilde{\lambda}_2 -\lambda_2.$
\par
{\bf (3)} Let $\Phi^-_2(\lambda)$ be the solution by Proposition \ref{prop4.2} 
near $\mathbf{c}_1=(\lambda_2,\lambda_1)^{\sim}$, and 
set $\Phi^+_2(\lambda)=\Phi^-_2(\lambda)\Gamma_{2*},$ where 
$\Phi^+_2(\lambda)$ is the analytic continuation
along an arc in $\mathcal{A}^2_{\varepsilon}$ in the anticlockwise direction.
Then by Proposition \ref{prop4.2},
$$
\Gamma_{2*} = \Phi^-_2(\lambda)^{-1}\Phi^+_2(\lambda)=S^{-1}_1=\begin{pmatrix}
1 & i \\ 0 & 1 \end{pmatrix}.
$$
\par
{\bf (4)} For $\Phi^-_2(\lambda)$ and the WKB solution $\Psi^-_1(\lambda)$
along $\mathbf{c}_1$, set $\Phi^-_2(\lambda)=\Psi^-_1(\lambda)\Gamma_{2-}.$
Note that $\mathbf{c}_1$ is on the upper sheet of $\mathcal{R}_{\Phi}.$
Then, supposing the curve $(2\kappa_d)^{1/3}(\lambda-\tilde{\lambda}_2')
=(t/4)^{-2/3}(\zeta+O(t^{-1/3}))$ with $\lambda \in \mathbf{c}_1$ is in the
sector $|\arg\zeta -\pi|<2\pi/3,$ we have, for $\tilde{\lambda}'_2 \in 
\mathbf{c}_1,$ $|\tilde{\lambda}'_2 -\lambda_2| \asymp t^{-1},$
\begin{align*}
\Gamma_{2-} =& \Psi_1^-(\lambda)^{-1} \Phi^-_2(\lambda)
\\
=& \exp \Bigl( -\int^{\lambda}_{\tilde{\lambda}'_2} \Lambda(\tau)d\tau \Bigr)
(I+O(t^{-\delta})) \begin{pmatrix}  1  & \frac{b_3-\mu}{b_1+ ib_2}  \\
\frac{\mu-b_3}{b_1-i b_2} & 1  \end{pmatrix}^{\!\! -1}
\\
&\phantom{--}\times  \begin{pmatrix}  1 &  -\frac{d_3}{d_1+ id_2} \\
-\frac{d_3}{d_1- i d_2} &  1 \end{pmatrix} (I+O(t^{-\delta})) 
\begin{pmatrix} 1 & 0 \\ 0 & (t/4)^{-1/3}\tilde{K} \end{pmatrix} W(\zeta)
\\
=& \exp\Bigl(\frac 23\zeta^{3/2} \sigma_3 -\int^{\lambda}_{\tilde{\lambda}'_2}
\Lambda(\tau) d\tau \Bigr) \zeta^{-1/4}(I+O(t^{-\delta})) \begin{pmatrix}
1 & 0 \\ 0 & -\frac{d_3}{d_1-i d_2} \end{pmatrix},
\end{align*}
where $\tilde{K}^{-1}=2(2\kappa_d)^{-1/3} (d_1- i d_2).$
This yields 
$$
\Gamma_{2-}=(\tilde{\zeta}'_2)^{-1/4}(I+O(t^{-\delta})) C'_3(\tilde{\lambda}'_2)
\begin{pmatrix}  
1 &  0  \\  0 & -\frac{d_3}{d_1- i d_2} \end{pmatrix}
$$
with $C'_3(\tilde{\lambda}'_2)=\exp(\int^{\tilde{\lambda}'_2}_{\lambda_2}
\Lambda_3(\tau)d\tau )$ for some $\tilde{\zeta}'_2 \asymp \tilde{\lambda}'_2
-\lambda_2.$
\par
{\bf (5)} For $\Psi^-_1(\lambda)$ and the WKB solution $\Psi^+_1(\lambda)$ along
$\mathbf{c}_1$ near $\lambda_1$, set $\Psi^-_1(\lambda)=\Psi^+_1(\lambda)
\Gamma_{12}.$ Then, for $\tilde{\lambda}_1 \in \mathbf{c}_1,$ $\tilde{\lambda}
_1 -\lambda_1 \asymp t^{-1}$,
\begin{equation*}
\Gamma_{12} = \Psi^+_1(\lambda)^{-1} \Psi^-_1(\lambda)
\\
= C'_3(\tilde{\lambda}'_2)^{-1} C''_3(\tilde{\lambda}_1) c_I(\tilde{\lambda}'
_2, \tilde{\lambda}_1) \exp\Bigl( -\int^{\lambda_2}_{\lambda_1} \Lambda_3(\tau)
d\tau \Bigr),
\end{equation*}
where $C''_3(\tilde{\lambda}_1)=\exp (\int^{\tilde{\lambda}_1}_{\lambda_1}
\Lambda_3(\tau)d\tau),$ $c_I(\tilde{\lambda}'_2,\tilde{\lambda}_1)=\exp
(-\int^{\tilde{\lambda}'_2}_{\tilde{\lambda}_1} \Lambda_I(\tau) d\tau).$
\par
{\bf (6)} For $\Psi^+_1(\lambda)$ and for $\Phi^+_1(\lambda)$ given 
by Proposition
\ref{prop4.2} in the annulus $\mathcal{A}^1_{\varepsilon}$ around $\lambda_1$,
set $\Psi^+_1(\lambda)=\Phi^+_1(\lambda)\Gamma_{1+}.$ Then, by the same
argument as in (2) above with $\mu$ in place of $\mu_-$, we have 
$$
\Gamma_{1+} =\Phi^+_1(\lambda)^{-1}\Psi^+_1(\lambda) =(\tilde{\zeta}_1)^{1/4}
(I+O(t^{-\delta})) C''_3(\tilde{\lambda}_1)^{-1} \begin{pmatrix}
 1 & 0  \\ 0 & -\frac{c_1- i c_2}{c_3} \end{pmatrix}
$$
for some $\tilde{\zeta}_1 \asymp \tilde{\lambda}_1 -\lambda_1.$ 
\par
{\bf (7)} Let $\Phi^-_1(\lambda)$ be the solution by Proposition \ref{prop4.2} 
near $\mathbf{c}_0=(\lambda_1,0^+)^{\sim}$, and set $\Phi^+_1(\lambda)
 =\Phi^-_1(\lambda)\Gamma_{1\,*},$ where $\Phi^+_1(\lambda)$ is the analytic
continuation along 
an arc in $\mathcal{A}^1_{\varepsilon}$ in the clockwise direction. 
Then by Proposition \ref{prop4.2},
$$
\Gamma_{1\,*}= \Phi^-_1(\lambda)^{-1} \Phi^+_1(\lambda) = S_0= \begin{pmatrix}
1 & 0 \\  -i & 1  \end{pmatrix}.
$$
\par
{\bf (8)} For $\Phi^-_1(\lambda)$ and the WKB solution $\Psi_0(\lambda)$ along
$\mathbf{c}_0,$ set $\Phi^-_1(\lambda)=\Psi_0(\lambda)\Gamma_{1-}.$ By the
same argument as in (4), we have
$$
\Gamma_{1-}=\Psi_0(\lambda)^{-1}\Phi^-_1(\lambda)=(\tilde{\zeta}'_1)^{-1/4}
(I+O(t^{-\delta})) \hat{C}_3(\tilde{\lambda}'_1) \begin{pmatrix}
1 & 0 \\ 0 & -\frac{c_3}{c_1-i c_2} \end{pmatrix}
$$
with $\hat{C}_3(\tilde{\lambda}'_1) =\exp(\int^{\tilde{\lambda}'_1}_{\lambda_1}
\Lambda_3(\tau) d\tau )$ for some $\tilde{\zeta}'_1 \asymp \tilde{\lambda}'_1
-\lambda_1.$
\par
{\bf (9)} For $\Psi_0(\lambda)$ and ${Y}^0_0(\lambda)$ given by \eqref{3.3}, set
$\Psi_0(\lambda)={Y}^0_0(\lambda)\Gamma_0.$ Note that $\mu(t,\lambda)=ie^{i\phi}
\lambda^{-2}+O(1)$ as $\lambda\to 0^+$. 
Then
\begin{align*}
\Gamma_0=& {Y}^0_0(\lambda)^{-1}\Psi_0(\lambda)
\\
=& \exp\Bigl(\frac 14( i e^{i\phi} t\lambda^{-1}
 -2\theta_0\log\lambda)\sigma_3 \Bigr) (\Delta_0^*)^{-1}
T_0 (I+O(t^{-\delta}+|\lambda|))
\exp\Bigl(\int^{\lambda}_{\tilde{\lambda}'_1}\Lambda(\tau) d\tau \Bigr)
\end{align*}
with $T_0=T|_{\lambda=0}.$
Choosing $f_1=c^*_0,$ $f_2=c^*_0(\z-e^{i\phi}t/2)^{-1}$ with $c^*_0=
\sqrt{2} i e^{-i\phi/2} t^{-1/2}(\z -e^{i\phi}t/2)^{1/2},$ we have
$$
\Delta_0^*= c^*_0 \begin{pmatrix} 1 & \z(\z-e^{i\phi}t/2)^{-1} \\ 1 & 1
\end{pmatrix}.
$$
The $(1,2)$- and $(2,1)$-entries of $T-I$ satisfy
\begin{align*}
& \frac{b_3-\mu}{b_1+ib_2} \sim 
 \frac{-4i\z t^{-1}\lambda^{-2}}{-2it^{-1}(2\z -e^{i\phi}t)\lambda^{-2}}
\to \frac{\z}{\z-e^{i\phi} t/2},
\\
& \frac{\mu-b_3}{b_1-ib_2} \sim
\frac{ie^{i\phi}\lambda^{-2} 
+it^{-1}(4\z -e^{i\phi}t)\lambda^{-2}}{4i \z t^{-1}\lambda^{-2}}\to 1
\end{align*}
as $\lambda \to 0^+.$ This implies $(c_0^*)^{-1}\Delta_0^*=T_0.$ Hence we have
\begin{align*}
\Gamma_0 =  & \hat{C}_3(\tilde{\lambda}'_1)^{-1} \hat{c}_I(\tilde{\lambda}'_1)
(c_0^*)^{-1}(I +O(t^{-\delta}))
\\
& \times \exp\biggl(\lim_{\substack{\lambda \to 0^+ \\
\lambda \in \mathbf{c}_0 }} \Bigl(\int^{\lambda}_{\lambda_1} \Lambda_3(\tau)
d\tau +\frac 14(i e^{i\phi}t\lambda^{-1} -2\theta_0\log\lambda) 
\sigma_3 \Bigr) \biggr)
\end{align*}
with $\hat{c}_I(\tilde{\lambda}'_1) =  \exp(\int^0_{\tilde{\lambda}
'_1}\Lambda_I(\tau)d\tau).$
\par
Collecting the matrices above, we have the connection matrix
\begin{align*}
G^*=& 
 {Y}^0_0(\lambda)^{-1} {Y}^{\infty,*}_0(\lambda)
\\
=& \Gamma_0 \Gamma_{1-}\Gamma_{1\,*} \Gamma_{1+} \Gamma_{12} \Gamma_{2-}
\Gamma_{2*} \Gamma_{2+}\Gamma_{\infty}
\\
=&\epsilon_+ i(I+O(t^{-\delta}) ) \exp(J_0\sigma_3)
\begin{pmatrix} 1 & 0 \\ 0 & -c_0^{-1} \end{pmatrix}
\begin{pmatrix} 1 & 0 \\ -i & 1   \end{pmatrix}
\begin{pmatrix} 1 & 0 \\ 0 & -c_0 \end{pmatrix}
\\
&\phantom{--} \times 
 \exp(-J_1\sigma_3)
\begin{pmatrix} 1 & 0 \\ 0 & -d_0^{-1} \end{pmatrix}
\begin{pmatrix} 1 & i \\ 0 & 1   \end{pmatrix}
\begin{pmatrix} 0 & -{d_0} \\  1 & 0  \end{pmatrix}
\exp(-J_{\infty} \sigma_3)
\\
=& \epsilon_+  (I+O(t^{-\delta}))
\\
&\times
\begin{pmatrix}  -e^{J_0 -J_1-J_{\infty}}  & -id_0  e^{J_0-J_1+J_{\infty}} \\
-i (c^{-1}_0 e^{-J_1}+d^{-1}_0 e^{J_1})e^{-J_0-J_{\infty}} &
 c^{-1}_0d_0 e^{-J_0-J_1+J_{\infty}}   \end{pmatrix}
\end{align*}
if $0<\phi<\pi/2$, where $\epsilon_+^2=1,$ 
$ c_0=(c_1-i c_2)/c_3,$ $ d_0=(d_1-i d_2)/d_3$, and
\begin{align}\label{5.2}
& J_0\sigma_3 = \lim_{\substack{\lambda\to 0^+ \\ \lambda\in \mathbf{c}_0}}
\Bigl(\int^{\lambda}_{\lambda_1} \Lambda_3(\tau) d\tau +\frac 14
(i e^{i\phi} t\lambda^{-1} -2\theta_0 \log \lambda ) \sigma_3\Bigr), 
\\
\label{5.3}
& J_1\sigma_3 =\int^{\lambda_2}
_{\lambda_1} \Lambda_3(\tau)d\tau \,\,\,\, \text{(along $\mathbf{c}_1$),}
\\
\label{5.4}
& J_{\infty}\sigma_3 = \lim_{\substack{\lambda\to \infty^- \\ 
\lambda\in \mathbf{c}_{\infty}}}
\Bigl(\int^{\lambda}_{\lambda_2} \Lambda_{3-}(\tau) d\tau -\frac{1}4
(ie^{i\phi}t\lambda -2\theta_{\infty}\log \lambda)\sigma_3 \Bigr).
\end{align}
Note that the curve $\mathbf{c}_1$ joining $\lambda_1$ to $\lambda_2$
is located along the lower shore of the cut on the upper sheet of 
$\mathcal{R}_{\phi}.$
\par
The connection matrix $\hat{G}^*$ is obtained by calculating the matrix
solution $Y^{\infty,*}_1(\lambda)$ along the Stokes curve consisting of 
$\hat{\mathbf{c}}_{\infty}=(\infty^-, -\lambda_2)^{\sim},$
$\hat{\mathbf{c}}_{1}=(-\lambda_2,-\lambda_1)^{\sim},$
$\hat{\mathbf{c}}_{0}=(-\lambda_1,0^+)^{\sim},$ the union of which joins 
$\infty^-$ to $0^+$ as in Figure \ref{curve+*}.
The curve $\hat{\mathbf{c}}_{\infty}$ lies on the lower sheet of $\mathcal{R}
_{\phi}$, and $\hat{\mathbf{c}}_1$ and $\hat{\mathbf{c}}_0$ on the upper sheet
of $\mathcal{R}_{\phi}.$
{\small
\begin{figure}[htb]
\begin{center}
\unitlength=0.75mm
\begin{picture}(100,50)(-100,-25)
\put(5,-1){\makebox{$0^+$}}
\put(-20,-15){\makebox{$-\lambda_1$}}
\put(-59,15){\makebox{$-\lambda_2$}}
\put(-92,15){\makebox{$\infty^-$}}
\put(-13,0){\makebox{$\hat{\mathbf{c}}_0$}}
\put(-38,-6){\makebox{$\hat{\mathbf{c}}_1$}}
\put(-71,7){\makebox{$\hat{\mathbf{c}}_{\infty}$}}

\thinlines
\qbezier(-32,-14) (-48,-10.5) (-51,3)
\qbezier(-25.5,-15.7) (-52,-14) (-52.7,9.5)

  \qbezier(-58.5,9) (-58.5,5) (-54,4)
\put(-54,4){\vector(1,0){0}}

  \qbezier(-30.5,-12) (-28.5,-7) (-23,-9)
\put(-23,-9){\vector(3,-2){0}}

\qbezier(-40,22) (-46,18) (-52,10)
\qbezier(-40,22.06) (-46,18.06) (-52,10.06)

\thicklines
\put(0,0){\circle{1.5}}
\put(-25,-15){\circle*{1.5}}
\put(-52,10){\circle*{1.5}}
\qbezier(-0.8,0) (-16,-3) (-25,-15)
\qbezier(-0.8,-0.08) (-16,-3.08) (-25,-15.08)
\qbezier(-25,-15) (-48,-11) (-52,10)
\qbezier(-24.92,-15) (-47.92,-11) (-51.92,10)

\qbezier[12](-25,-15) (-24,-20) (-20,-27)
\qbezier[20](-85,15) (-65,11) (-52,10)
\qbezier[20](-85.2,15.3) (-65.2,11.3) (-52.2,10.3)
\qbezier[20](-85.0,15.3) (-65.0,11.3) (-52.0,10.3)
\qbezier[20](-85.2,15) (-65.2,11) (-52.2,10)

\end{picture}
\end{center}
\caption{Curves $\hat{\mathbf{c}}_0$, $\hat{\mathbf{c}}_1$,
$\hat{\mathbf{c}}_{\infty}$ for $0<\phi <\pi/2$}
\label{curve+*}
\end{figure}
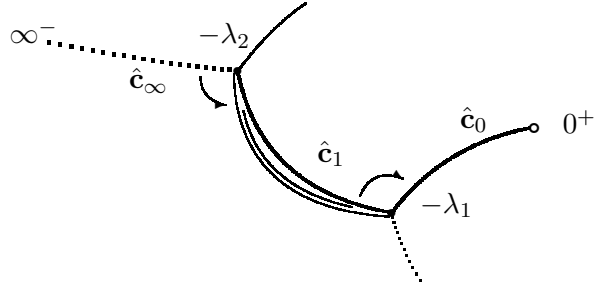
}
Then we have
\begin{align*}
 \hat{G}^*&= Y^0_1(\lambda)^{-1}Y^{\infty,*}_1(\lambda)
\\
&=\hat{\epsilon}_+ i(I+O(t^{-\delta})) \exp(\hat{J}_0\sigma_3)
\begin{pmatrix} 1 & 0 \\ 0 & -\hat{c}_0^{-1} \end{pmatrix}
\begin{pmatrix} 1 & -i \\ 0 & 1 \end{pmatrix}
\begin{pmatrix} 1 & 0 \\ 0 & -\hat{c}_0 \end{pmatrix}
\\
&\phantom{---}\times
\exp(-\hat{J}_1\sigma_3)
\begin{pmatrix} 1 & 0 \\ 0 & -\hat{d}_0^{-1} \end{pmatrix}
\begin{pmatrix} 1 & 0 \\ i & 1 \end{pmatrix}
\begin{pmatrix} 0 & -\hat{d}_0 \\ 1 & 0 \end{pmatrix}
\exp(-\hat{J}_{\infty} \sigma_3)
\\
&=\hat{\epsilon}_+ (I+O(t^{-\delta}))
\\
&\phantom{---}\times
\begin{pmatrix} \hat{c}_0\hat{d}_0^{-1} e^{\hat{J}_0+\hat{J}_1-\hat{J}_{\infty}}
& -i(\hat{d}_0 e^{-\hat{J}_1} +\hat{c}_0 e^{\hat{J}_1}) e^{\hat{J}_0+
\hat{J}_{\infty}}  \\
-i\hat{d}_0^{-1} e^{-\hat{J}_0+\hat{J}_1-\hat{J}_{\infty}} &
- e^{-\hat{J}_0+\hat{J}_1+\hat{J}_{\infty}} 
\end{pmatrix}.
\end{align*}
Here $\hat{\epsilon}_+^2=1$, $\hat{c}_0=(\hat{c}_1-i\hat{c}_2)/\hat{c}_3$, 
$\hat{d}_0=(\hat{d}_1-i\hat{d}_2)/\hat{d}_3$ with $\hat{c}_j=b_j(\lambda_{-1}),$
$\hat{d}_j=b_j(\lambda_{-2})$ $(j=1,2,3)$, and
$\hat{J}_0=J_0|_{\mathbf{c}_0 \mapsto \hat{\mathbf{c}}_0},$
$\hat{J}_1=J_1|_{\mathbf{c}_1 \mapsto \hat{\mathbf{c}}_1},$
$\hat{J}_{\infty}=J_{\infty}|_{\mathbf{c}_{\infty} \mapsto 
\hat{\mathbf{c}}_{\infty}}.$
\par
In the case $-\pi/2 <\phi <0$, calculating the analytic continuation along 
the Stokes curves consisting of $\mathbf{c}_{\infty},$ $\mathbf{c}^-_1,$
$\mathbf{c}_0$, and of $\hat{\mathbf{c}}_{\infty},$ $\hat{\mathbf{c}}^-_1,$
$\hat{\mathbf{c}}_0$ as in Figure \ref{curve-}, we have, in a similar manner,
the connection matrices $G^*$ and $\hat{G}^*$: 
\begin{align*}
{G^*}=&\epsilon_- i(I+O(t^{-\delta}) ) \exp(J_0\sigma_3)
\begin{pmatrix} 1 & 0 \\ 0 & -c^{-1}_0 \end{pmatrix}
\begin{pmatrix} 1 & i \\ 0 & 1   \end{pmatrix}
\begin{pmatrix} 1 & 0 \\ 0 & -c_0 \end{pmatrix}
\\
&\phantom{--} \times 
 \exp(-{J}^-_1\sigma_3)
\begin{pmatrix} 1 & 0 \\ 0 & -d_0^{-1} \end{pmatrix}
\begin{pmatrix} 1 & 0 \\ -i & 1   \end{pmatrix}
\begin{pmatrix} 0 & -d_0 \\  1 & 0  \end{pmatrix}
\exp(-J_{\infty} \sigma_3)
\\
=& \epsilon_-  (I+O(t^{-\delta}))
\\
&\times
\begin{pmatrix} -c_0d_0^{-1} e^{J_0 +{J}^-_1-J_{\infty}}  &
-i(c_0 e^{{J}^-_1}+d_0 e^{-{J}^-_1}) e^{J_0+J_{\infty}} \\
-i d_0^{-1} e^{-J_0+{J}^-_1-J_{\infty}}  &
 e^{-J_0+{J}^-_1+J_{\infty}}   \end{pmatrix},
\end{align*}
and
\begin{align*}
\hat{G^*}=&\hat{\epsilon}_- i(I+O(t^{-\delta}) ) \exp(\hat{J}_0\sigma_3)
\begin{pmatrix} 1 & 0 \\ 0 & -\hat{c}^{-1}_0 \end{pmatrix}
\begin{pmatrix} 1 & 0 \\ i & 1   \end{pmatrix}
\begin{pmatrix} 1 & 0 \\ 0 & -\hat{c}_0 \end{pmatrix}
\\
&\phantom{--} \times 
 \exp(-\hat{J}^-_1\sigma_3)
\begin{pmatrix} 1 & 0 \\ 0 & -\hat{d}_0^{-1} \end{pmatrix}
\begin{pmatrix} 1 & -i \\ 0 & 1   \end{pmatrix}
\begin{pmatrix} 0 & -\hat{d}_0 \\  1 & 0  \end{pmatrix}
\exp(-\hat{J}_{\infty} \sigma_3)
\\
=& \hat{\epsilon}_-  (I+O(t^{-\delta}))
\\
&\times
\begin{pmatrix}    e^{\hat{J}_0 -\hat{J}^-_1-\hat{J}_{\infty}}  &
-i \hat{d}_0 e^{\hat{J}_0-\hat{J}^-_1+\hat{J}_{\infty}} \\ 
-i(\hat{c}_0^{-1} e^{-\hat{J}^-_1}+\hat{d}_0^{-1} e^{\hat{J}^-_1}) 
e^{-\hat{J}_0-\hat{J}_{\infty}} &
 -\hat{c}_0^{-1}\hat{d}_0  e^{-\hat{J}_0-\hat{J}^-_1+\hat{J}_{\infty}} 
  \end{pmatrix}.
\end{align*}
Here $\epsilon_-^2=\hat{\epsilon}_-^2=1,$ and
\begin{equation}\label{5.5}
 J_1^-\sigma_3 =\int^{\lambda_2}
_{\lambda_1} \Lambda_3(\tau)d\tau \,\,\,\, \text{(along $\mathbf{c}^-_1$),}
\end{equation}
$\hat{J}^-_1 =J^-_1|_{\mathbf{c}_1^- \mapsto \hat{\mathbf{c}}^-_1}$,
in which ${\mathbf{c}}^-_1$ is a curve joining $\lambda_1$ to $\lambda_2$
located along the upper shore of the cut on the upper sheet of $\mathcal{R}
_{\phi}$.
{\small
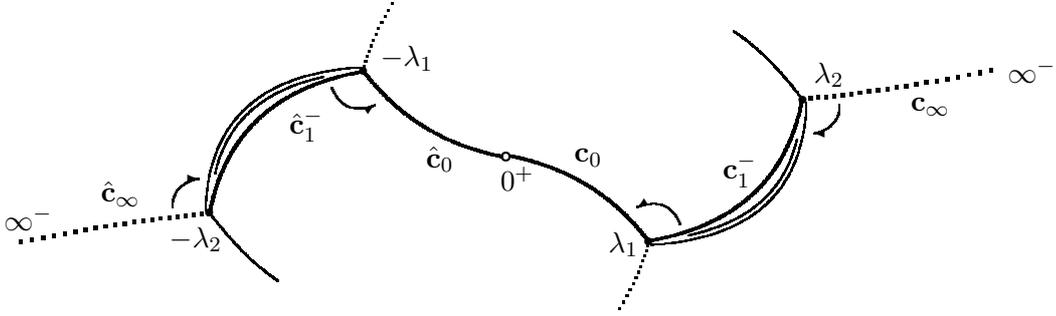
\begin{figure}[htb]
\begin{center}
\unitlength=0.75mm
\begin{picture}(100,50)(-50,-25)
\put(-1,-6){\makebox{$0^+$}}
\put(18,-17){\makebox{$\lambda_1$}}
\put(54,13){\makebox{$\lambda_2$}}
\put(88,13){\makebox{$\infty^-$}}
\put(12,-1){\makebox{$\mathbf{c}_0$}}
\put(38,-4){\makebox{${\mathbf{c}}^-_1$}}
\put(71,8){\makebox{$\mathbf{c}_{\infty}$}}

\put(-22,16){\makebox{$-\lambda_1$}}
\put(-59,-16){\makebox{$-\lambda_2$}}
\put(-88,-13){\makebox{$\infty^-$}}
\put(-14,-2){\makebox{$\hat{\mathbf{c}}_0$}}
\put(-38,4){\makebox{$\hat{\mathbf{c}}^-_1$}}
\put(-71,-8){\makebox{$\hat{\mathbf{c}}_{\infty}$}}

\thinlines

\qbezier(32,-14) (48,-10.5) (51,3)
\qbezier(25.5,-15.7) (52,-14) (52.7,9.5)

  \qbezier(58.5,9) (58.5,5) (54,4)
\put(54,4){\vector(-1,0){0}}

  \qbezier(30.5,-12) (28.5,-7) (23,-9)
\put(23,-9){\vector(-3,-2){0}}

\qbezier(40,22) (46,18) (52,10)
\qbezier(40,22.06) (46,18.06) (52,10.06)

\qbezier(-32,14) (-48,10.5) (-51,-3)
\qbezier(-25.5,15.7) (-52,14) (-52.7,-9.5)

  \qbezier(-58.5,-9) (-58.5,-5) (-54,-4)
\put(-54,-4){\vector(1,0){0}}

  \qbezier(-30.5,12) (-28.5,7) (-23,9)
\put(-23,9){\vector(3,2){0}}

\qbezier(-40,-22) (-46,-18) (-52,-10)
\qbezier(-40,-22.06) (-46,-18.06) (-52,-10.06)

\thicklines
\put(0,0){\circle{1.5}}
\put(25,-15){\circle*{1.5}}
\put(52,10){\circle*{1.5}}
\qbezier(0.8,0) (16,-3) (25,-15)
\qbezier(0.8,0.08) (16,-3.08) (25,-15.08)
\qbezier(25,-15) (48,-11) (52,10)
\qbezier(24.92,-15) (47.92,-11) (51.92,10)

\qbezier[12](25,-15) (24,-20) (20,-27)
\qbezier[20](85,15) (65,11) (52,10)
\qbezier[20](85.2,15.3) (65.2,11.3) (52.2,10.3)
\qbezier[20](85.0,15.3) (65,11.3) (52,10.3)
\qbezier[20](85.2,15) (65.2,11) (52.2,10)

\put(-25,15){\circle*{1.5}}
\put(-52,-10){\circle*{1.5}}
\qbezier(-0.8,0) (-16,3) (-25,15)
\qbezier(-0.8,-0.08) (-16,3.08) (-25,15.08)
\qbezier(-25,15) (-48,11) (-52,-10)
\qbezier(-24.92,15) (-47.92,11) (-51.92,-10)

\qbezier[12](-25,15) (-24,20) (-20,27)
\qbezier[20](-85,-15) (-65,-11) (-52,-10)
\qbezier[20](-85.2,-15.3) (-65.2,-11.3) (-52.2,-10.3)
\qbezier[20](-85.0,-15.3) (-65,-11.3) (-52,-10.3)
\qbezier[20](-85.2,-15) (-65.2,-11) (-52.2,-10)

\end{picture}
\end{center}
\caption{Stokes curve for $-\pi/2<\phi <0$}
\label{curve-}
\end{figure}
}
Thus we have the following.
\begin{prop}\label{prop5.1}
Let $c_0=(c_1-i c_2)/c_3,$ $d_0=(d_1-i d_2)/d_3$, 
$\hat{c}_0=(\hat{c}_1-i \hat{c}_2)/\hat{c}_3,$ 
$\hat{d}_0=(\hat{d}_1-i \hat{d}_2)/\hat{d}_3$ 
with $c_k=b_k(\lambda_1),$ $d_k=b_k(\lambda_2)$ $\hat{c}_k=b_k(\lambda_{-1}),$ 
$\hat{d}_k=b_k(\lambda_{-2})$ 
for $k=1,2,3.$ If $0<\phi <\pi/2$ and $g^*_{11}g^*_{12}g^*_{22}\hat{g}^*_{11}
\hat{g}^*_{21}\hat{g}^*_{22}\not=0,$ then
\begin{align*}
&g^*_{11}g^*_{22}= -c^{-1}_0d_0 (1+O(t^{-\delta}))e^{-2J_1}, \quad
g^*_{12}/g^*_{22} = -ic_0(1+O(t^{-\delta})) e^{2J_0},
\\
&\hat{g}^*_{11}\hat{g}^*_{22}= -\hat{c}_0\hat{d}_0^{-1}(1+O(t^{-\delta}))
 e^{2\hat{J}_1}, \quad
\hat{g}^*_{11}/\hat{g}^*_{21} = i\hat{c}_0 (1+O(t^{-\delta}))e^{2\hat{J}_0}.
\end{align*}
If $-\pi/2 <\phi <0$ and $g^*_{11}g^*_{21}g^*_{22}\hat{g}^*_{11}
\hat{g}^*_{12}\hat{g}^*_{22}\not=0,$ then
\begin{align*}
&g^*_{11}g^*_{22}= -c_0d_0^{-1} (1+O(t^{-\delta}))e^{2J^-_1}, \quad
g^*_{11}/g^*_{21} = -ic_0 (1+O(t^{-\delta}))e^{2J_0},
\\ 
&\hat{g}^*_{11}\hat{g}^*_{22}= -\hat{c}_0^{-1}\hat{d}_0
 (1+O(t^{-\delta}))e^{-2\hat{J}^-_1}, \quad
\hat{g}^*_{12}/\hat{g}^*_{22} = i\hat{c}_0 (1+O(t^{-\delta}))e^{2\hat{J}_0}.
\end{align*}
Here $J_0,$ $J_1$, ${J}^-_1$ are
integrals given by \eqref{5.2}, \eqref{5.3}, \eqref{5.5},
and $\hat{J}_0=J_0|_{\mathbf{c}_0 \mapsto \hat{\mathbf{c}}_0}$,  
$\hat{J}_1=J_1|_{\mathbf{c}_1 \mapsto \hat{\mathbf{c}}_1}$,  
$\hat{J}_1^-=J_1^-|_{\mathbf{c}_1^- \mapsto \hat{\mathbf{c}}_1^-}$.  
\end{prop}
From the proposition above with entries of $G^*$ and $\hat{G}^*$ combined
with Corollary \ref{cor3.10}, we derive the following key relations.
\begin{cor}\label{cor5.2}
If $0<\phi<\pi/2$ and ${g}_{11}{g}_{12}{g}_{22}\hat{g}_{11}\hat{g}_{21}\not=0,$ 
then
$$
{g}_{11}{g}_{22} = -c^{-1}_{0} d_0 (1+O(t^{-\delta})) e^{-2J_1}, \quad
\frac{{g}_{12}\hat{g}_{21}}{{g}_{22}\hat{g}_{11}} 
=- c_{0}\hat{c}_0^{-1} (1+O(t^{-\delta})) e^{2J_0 - 2\hat{J}_0}. 
$$
If $-\pi/2<\phi<0$ and ${g}_{11}{g}_{21}{g}_{22}\hat{g}_{12}\hat{g}_{22}\not=0,$
then
$$
{g}_{11}{g}_{22} = -c_{0} d_0^{-1} (1+O(t^{-\delta}))
 e^{2J_1^-}, \quad
\frac{{g}_{11}\hat{g}_{22}}{{g}_{21}\hat{g}_{12}} 
= -c_{0}\hat{c}_0^{-1} (1+O(t^{-\delta})) e^{2J_0 -2\hat{J}_0}. 
$$
\end{cor}
\section{Asymptotic properties of monodromy data}\label{sc6}
\subsection{Expressions of integrals}\label{ssc6.1}
By \eqref{5.1}, for sufficiently large $t$, consider the elliptic curve 
$\Pi_{a_{\phi}}=\Pi_+\cup \Pi_-$ for $w(a_{\phi},\lambda)=\sqrt{1-a_{\phi}
\lambda^2+\lambda^4}$, the branch of which is consistent with that of
$w(A_{\phi},\lambda)$ in Section \ref{sc2}. On $\Pi_{a_{\phi}}$ we
may define cycles $\mathbf{a}$ and $\mathbf{b}$ which are identified with
those on $\Pi_{A_{\phi}}$ as given as in Figure \ref{cycles1}. Then
the cycles $\mathbf{a}$ and $\mathbf{b}$ on $\Pi_{a_{\phi}}$ are 
independent of $t$, and are also regarded to be on $\mathcal{R}_{\phi}.$
On $\mathcal{R}_{\phi}$ around these cycles, 
\begin{equation*}
\mu(t,\lambda)=ie^{i\phi}\frac{w(a_{\phi},\lambda)}{\lambda^2} +\frac{2\theta_0
t^{-1}}{\lambda w(a_{\phi},\lambda)} -\frac{2\theta_{\infty}\lambda t^{-1}}
{w(a_{\phi},\lambda)} +O((|\lambda^4 w^{-3}|+1)t^{-2}).
\end{equation*}
\par
For integrals appearing in Corollary \ref{cor5.2} we have
\begin{equation*}
 J_1\sigma_3 =\int^{\lambda_2}_{\lambda_1,\, (\mathbf{c}_1)} \Lambda_3(\lambda)
d\lambda = -\frac 12 \int_{\mathbf{b}} \Lambda_3(\lambda)d\lambda,
\quad
 J_1^- \sigma_3 = \int^{\lambda_2}_{\lambda_1,\, (\mathbf{c}^-_1)}
\Lambda_3(\lambda)d\lambda = \frac 12 \int_{\mathbf{b}} \Lambda_3(\lambda)d\lambda
\end{equation*}
with $\Lambda_3(\lambda)=(t/4)\mu(t,\lambda)\sigma_3 -\diag T^{-1}T_{\lambda}|
_{\sigma_3}\sigma_3.$ 
Note that
\begin{align*}
& \int \frac{w}{\lambda^2} d\lambda = -\frac w{\lambda} + 2\int \frac{\lambda^2}
w d\lambda -a_{\phi} \int \frac{d\lambda}w
=\frac{w}{\lambda} -a_{\phi} \int\frac{d\lambda}w +2\int\frac{d\lambda}
{\lambda^2 w},
\\
&  \int \frac{ d\lambda}{\lambda w}= -\frac 12 \log(\lambda^{-2} -a_{\phi}/2
+\lambda^{-2} w ), 
\quad
  \int \frac{\lambda d\lambda}{ w}= \frac 12 \log(\lambda^{2} -a_{\phi}/2
+ w ). 
\end{align*}
\begin{lem}\label{lem6.1} We have
\begin{equation*}
\int_{\mathbf{b}} \mu(t,\lambda)d\lambda =ie^{i\phi} \int_{\mathbf{b}}
\frac{w(a_{\phi},\lambda)}{\lambda^2} d\lambda + 2\pi i(\theta_0-\theta_{\infty}
)t^{-1} +O(t^{-2}).
\end{equation*}
\end{lem}
Furthermore we have
\begin{align*}
&\lim_{\substack{\lambda \to 0^+ \\ \lambda\in \mathbf{c}_0}}
\biggl(\frac t4 \int^{\lambda}_{\lambda_1} \mu(t,\lambda)d\lambda
+ \frac{ie^{i\phi}t}{4\lambda} -\frac{\theta_0}2\log\lambda \biggr)
\\
=&\lim_{\substack{\lambda \to 0^+ \\ \lambda\in \mathbf{c}_0}}
\biggl(\frac {ie^{i\phi}t}4 \biggl( \int^{\lambda}_{\lambda_1}
 \frac{w}{\lambda^2}d\lambda +\frac 1{\lambda} \biggr) 
+\frac{\theta_0}2 \biggl(\int^{\lambda}_{\lambda_1} \frac{d\lambda}{\lambda w}
-\log \lambda\biggr) -\int^{\lambda}_{\lambda_1} \frac{\theta_{\infty}\lambda}
{2w} d\lambda \biggr) +O(t^{-1}), 
\end{align*}
in which
$$
  \int^{\lambda}_{\lambda_1}\frac{w}{\lambda^2}d\lambda +\frac 1{\lambda}  
 = 2\int^{\lambda}_{\lambda_1} \frac{\lambda^2}w d\lambda
 -a_{\phi}\int^{\lambda}_{\lambda_1} \frac{d\lambda}w  +O(\lambda), \quad
\int^{\lambda}_{\lambda_1}\frac{d\lambda}{\lambda w}-\log\lambda= C_{\lambda_1,
a_{\phi}} +O(\lambda^2)
$$
with $C_{\lambda_1,a_{\phi}}=\tfrac 12 \log(\lambda_1^{-2}-\tfrac 12 a_{\phi})
-\tfrac 12 \log 2$ as $\lambda\to 0^+,$
and a similar formula is obtained as $\lambda \to 0^+$ along 
$\hat{\mathbf{c}}_0.$
Hence, in $J_0-\hat{J}_0,$
\begin{align}\label{6.0}
&  \lim_{\substack{\lambda \to 0^+ \\ \lambda\in \mathbf{c}_0}}
\biggl(\frac t4 \int^{\lambda}_{\lambda_1} \mu(t,\lambda)d\lambda + \frac
{ie^{i\phi}t}{4\lambda} -\frac{\theta_0}2 \log \lambda \biggr)
\\
\notag
&\phantom{----}
 - \lim_{\substack{\lambda \to 0^+ \\ \lambda\in \hat{\mathbf{c}}_0}} 
\biggl(\frac t4 \int^{\lambda}_{-\lambda_1} \mu(t,\lambda)d\lambda + \frac
{ie^{i\phi}t}{4\lambda} -\frac{\theta_0}2 \log \lambda \biggr)
\\
\notag
= &\frac{ie^{i\phi}t}4 \biggl(2\int^{-\lambda_{1}}_{\lambda_1} \frac{\lambda^2}
w d\lambda -a_{\phi}\int^{-\lambda_{1}}_{\lambda_1} \frac{d\lambda}{w}\biggr)
 -\int^{-\lambda_{1}}_{\lambda_1}\frac{\theta_{\infty}\lambda}{2w}d\lambda 
+O(t^{-1})
\\
\notag
= & -\frac{ie^{i\phi} t}8 \int_{\mathbf{a}} \frac{w}{\lambda^2}d\lambda
+O(t^{-1}).
\end{align}
By Remark \ref{rem4.1}
\begin{equation}\label{6.1}
\diag T^{-1}  T_{\lambda}|_{\sigma_3} = \frac 14 \Bigl(1-
\frac{b_3}{\mu} \Bigr) \frac{\partial}{\partial \lambda} \log \frac{b_1+ib_2}
{b_1-ib_2}.
\end{equation}
To calculate our integrals, it is necessary to know 
$\diag T^{-1}T_{\lambda}|_{\sigma_3}$ in addition to Lemma \ref{lem6.1} and
\eqref{6.0}. 
By \eqref{3.1} and \eqref{3.2},
\begin{align*}
&b_1-ib_2 = -(\lambda -iy^{-1}) y\lambda^{-2} (y^* y^{-2} +e^{i\phi} -e^{i\phi}
y^{-2}) +O(t^{-1}), 
\\
&b_1+ib_2 = -(\lambda +iy^{-1}) y\lambda^{-2} (y^* y^{-2} -e^{i\phi} -e^{i\phi}
y^{-2}) +O(t^{-1}), 
\\
& b_3=i e^{i\phi} -i(y^* -e^{i\phi})y^{-2}\lambda^{-2} +O(t^{-1}),
\\
& \mu =ie^{i\phi}\lambda^{-2} w +O(t^{-1}).
\end{align*}
Set 
\begin{equation}\label{6.2}
\lambda_+=iy^{-1} +O(t^{-1}), \quad \lambda_-= -iy^{-1}+O(t^{-1})
\end{equation}
such that $(b_1-ib_2)(\lambda_+)=(b_1+ib_2)(\lambda_-)=0$.
Note that
$$
\frac{b_3}{\mu} = - i e^{-i\phi}\lambda^2 {b_3} \Bigl(\frac 1{w(\lambda)} +
O(t^{-1}) \Bigr),
$$
where $\lambda^2 b_3 \sim ie^{i\phi}\lambda^2$ satisfies $\lambda_{\pm}^2b_3
(\lambda_{\pm}) =
-\lambda_{\pm}^2\mu(\lambda_{\pm}) = -i e^{i\phi}w(\lambda_{\pm}) +O(t^{-1})$ 
as $\lambda_{\pm} \to \infty^+$ on the upper sheet,
since
$\mu(\lambda_{\pm})^2 =(b_1-ib_2)(b_1+i b_2)(\lambda_{\pm})
 +b_3(\lambda_{\pm})^2 =b_3(\lambda_{\pm})^2$. 
These facts combined with \eqref{6.1} yield
\begin{align*}
\diag &T^{-1}T_{\lambda}|_{\sigma_3} 
= \frac 14 \Bigl(1+ ie^{-i\phi}\lambda^2 b_3 \Bigl(\frac 1{w(\lambda)} +
O(t^{-1}) \Bigr)\Bigr) \Bigl(\frac 1{\lambda-\lambda_-} 
-\frac 1{\lambda-\lambda_+} \Bigr)
\\
=& -\frac 14 \Bigl( \frac 1{\lambda-\lambda_+} 
-\frac 1{\lambda-\lambda_-} \Bigl) - \frac 14 \Bigl(\lambda_-
-\lambda_+ +\frac{w(\lambda_+)}{\lambda-\lambda_+}
-\frac{w(\lambda_-)}{\lambda-\lambda_-}\Bigr) \frac 1{w(\lambda)}  +O(t^{-1}).
\end{align*}
Since, for $W_0(\lambda)=(\lambda-\lambda_+)^{-1}-(\lambda-\lambda_-)^{-1}$,
$$
\int^{\lambda_1}_{\lambda_{-1}} W_0(\lambda)d\lambda
 = 2\log (c_0 \hat{c}_0^{-1}), \quad 
\int^{\lambda_2}_{\lambda_1,\, (\mathbf{c}_1),\,\,\, (\mathbf{c}_1^-)} 
\!\!\!\!\!\!\!\!\!\!\!\!
W_0(\lambda)d\lambda = 2\log (c_0^{-1} {d}_0), \,\, -2\log (c_0^{-1} {d}_0),  
$$
we have, by the definitions of $J_1,$ $J_1^-$ and $J_0-\hat{J}_0$ with
\eqref{6.0},
\begin{align*}
& J_0-\hat{J}_0 = -\frac 12 \int_{\mathbf{a}}\Bigl(\frac {ie^{i\phi}t w}
 {4\lambda^2} 
+\frac 14 W(\lambda) \Bigr) d\lambda  -\frac 12 \log(c_0\hat{c}^{-1}_0)
+O(t^{-1}),
\\
& J_1 = -\frac 12 \int_{\mathbf{b}}\Bigl(\frac t4 \mu(t,\lambda)
+\frac 14 W(\lambda) \Bigr) d\lambda -\frac 12 \log(c_0{d}^{-1}_0)+O(t^{-1}),
\\
& J_1^- = \frac 12 \int_{\mathbf{b}}\Bigl(\frac t4 \mu(t,\lambda)
+\frac 14 W(\lambda) \Bigr) d\lambda -\frac 12 \log(c_0{d}^{-1}_0)+O(t^{-1})
\end{align*}
with
$$
W(\lambda)=\Bigl(\frac{w(\lambda_+)}{\lambda-\lambda_+} -
\frac{w(\lambda_-)}{\lambda-\lambda_-} - \lambda_+ +\lambda_- \Bigr)\frac
1{w(\lambda)}.
$$
Then Corollary \ref{cor5.2} and Lemma \ref{lem6.1} lead to the following.
\begin{prop}\label{prop6.2}
$(\mathrm{i})$ Suppose that $0<\phi<\pi/2.$ Then
\begin{align*}
\log(g_{11}g_{22})&=\frac{ie^{i\phi}t}4 \int_{\mathbf{b}}
\frac{w(a_{\phi},\lambda)}{\lambda^2} d\lambda +\frac 14 \int_{\mathbf{b}}
W(\lambda) d\lambda+\frac{\pi i}2(\theta_0-\theta_{\infty}) 
+\pi i +O(t^{-\delta}),
\\
\log\frac{g_{12}\hat{g}_{21}} {g_{22} \hat{g}_{11}}
&= - \frac{ie^{\phi i}t}4 \int_{\mathbf{a}} \frac{w(a_{\phi},\lambda)}{\lambda^2}
d\lambda -\frac 14 \int_{\mathbf{a}}
W(\lambda)  d\lambda +\pi i +O(t^{-\delta}).
\end{align*}
\par
$(\mathrm{ii})$ Suppose that $-\pi/2 <\phi <0.$ Then
\begin{align*}
\log(g_{11}g_{22})
&=\frac{ie^{i\phi}t}4 \int_{\mathbf{b}}
\frac{w(a_{\phi},\lambda)}{\lambda^2} d\lambda +\frac 14 \int_{\mathbf{b}}
W(\lambda) d\lambda+\frac{\pi i}2(\theta_0-\theta_{\infty}) 
+\pi i +O(t^{-\delta}),
\\
\log\frac{g_{11}\hat{g}_{22}} {g_{21} \hat{g}_{12}}
&= - \frac{ie^{\phi i}t}4 \int_{\mathbf{a}} \frac{w(a_{\phi},\lambda)}{\lambda^2}
d\lambda -\frac 14 \int_{\mathbf{a}}
W(\lambda)  d\lambda +\pi i +O(t^{-\delta}).
\end{align*}
\end{prop}
\begin{rem}\label{rem6.1}
In the proposition above
$$
\int_{\mathbf{a},\, \mathbf{b}} \frac{w(a_{\phi},\lambda)}{\lambda^2}d\lambda
= - a_{\phi} \int_{\mathbf{a}, \, \mathbf{b}}
\frac{d\lambda}{w(a_{\phi},\lambda)} +2\int_{\mathbf{a},\, \mathbf{b}}
\frac{d\lambda}{\lambda^2w(a_{\phi},\lambda)}.
$$
\end{rem}
\subsection{Expressions by the $\vartheta$-function}\label{ssc6.2}
For $w(\lambda)^2=w(a_{\phi},\lambda)^2=\lambda^4-a_{\phi}\lambda^2+1,$ 
the differential equation
$(d\lambda/du)^2=w(a_{\phi},\lambda)^2$ defines the doubly periodic 
function
$$
\lambda=\mathrm{sn}_*(u)=\lambda_1\mathrm{sn}(\lambda_2u; \lambda_1/\lambda_2), 
$$
where $\mathrm{sn}(u;k)$ is the Jacobi $\mathrm{sn}$-function, and 
$\lambda_1$ and $\lambda_2$ satisfy 
$\lambda^4-a_{\phi}\lambda^2+1=(1-\lambda_1^{-2}\lambda^2)
(1-\lambda_2^{-2}\lambda^2)$ and $0<\re \lambda_1 \le \re\lambda_2$ if $a_{\phi}$
is close to $2$. 
The periods of the elliptic curve $\Pi_{a_{\phi}}$ for $w(a_{\phi},\lambda)$
along the cycles $\mathbf{a}$ and $\mathbf{b}$ are
$$
\omega_{\mathbf{a}} =\int_{\mathbf{a}} \frac{d\lambda}{w(a_{\phi},\lambda)},\quad
\omega_{\mathbf{b}} =\int_{\mathbf{b}} \frac{d\lambda}{w(a_{\phi},\lambda)},\quad
\tau =\frac{\omega_{\mathbf{b}}}{\omega_{\mathbf{a}}}, \quad
\im \tau >0
$$
(cf. Section \ref{sc2}). The $\vartheta$-function $\vartheta(z,\tau)
=\vartheta(z)$ is defined by
$$
\vartheta(z,\tau) =\sum_{n=-\infty}^{\infty} e^{\pi i\tau n^2 +2\pi izn}
$$
and we set
$$
\nu=\frac{1+\tau}2
$$
(cf. \cite{HC}, \cite{WW}). For $\lambda,$ $\tilde{\lambda} \in \Pi_{a_{\phi}}
=\Pi_+ \cup \Pi_-,$ let
$$
F(\tilde{\lambda}, \lambda)=\frac 1{\omega_{\mathbf{a}}}\int_{\tilde{\lambda}}
^{\lambda}\frac {d\lambda}{w(\lambda)}. 
$$
For any $\lambda_0 \in \Pi_{a_{\phi}}$ denote the projections of $\lambda_0$ 
on the respective sheets by 
$$
\lambda_0^+=(\lambda_0, w(\lambda_0)) = (\lambda_0, w(\lambda_0^+)), \quad 
\lambda_0^-=(\lambda_0, -w(\lambda_0)) =(\lambda_0,-w(\lambda_0^+)) .
$$ 
If $\lambda_0 \in \Pi_+$ (respectively, $\lambda_0 \in \Pi_-$), 
then $\lambda_0^{\pm} \in \Pi_{\pm}$ (respectively,
$\lambda_0^{\pm} \in \Pi_{\mp}$). 
Do not confuse this usage of the superscripts $\pm$ with that in $\infty
^{\pm}$ $($respectively, $0^{\pm})$ denoting $\infty$ $($respectively, $0)$
on $\Pi_{\pm}.$
\begin{prop}\label{prop6.3}
For any $\lambda_0 \in \Pi_{a_{\phi}}$
\begin{align*}
\frac{d\lambda}{(\lambda-\lambda_0)w(\lambda)}=& \frac 1{w(\lambda^+_0)}
 d\log \frac{\vartheta(F(\lambda^+_0, \lambda)
+\nu, \tau)}{\vartheta(F(\lambda^-_0,\lambda)+\nu,\tau)}
 -g_0(\lambda_0) \frac{d\lambda}{w(\lambda)},
\\
g_0(\lambda_0)= & \frac{w'(\lambda^+_0)}{2w(\lambda^+_0)}
 -\frac 1{\omega_{\mathbf{a}}}
\frac 1{w(\lambda^+_0)} \Bigl(\pi i+\frac{\vartheta'}{\vartheta} 
(F(\lambda^-_0, \lambda^+_0)
+\nu, \tau) \Bigr).
\end{align*}
\end{prop}
\begin{proof}
For $\lambda_0= \mathrm{sn}_*(u_0) \in \Pi_{a_{\phi}}$ let $u^{\pm}_0$ be such 
that $\lambda^{\pm}_0=\mathrm{sn}_*(u^{\pm}_0).$ Then
\begin{align*}
\frac{d\lambda}{(\lambda-\lambda_0)w(\lambda)}
 =&\frac{du}{\mathrm{sn}_*(u)-\mathrm{sn}_*(u_0)} 
\\
=& \frac 1{w(\lambda^+_0)} \Bigl(\zeta(u-u^+_0)-\zeta(u-u^-_0)
 +\zeta(u^+_0 -u^-_0)-\frac 12 w'(\lambda^+_0) \Bigr) du
\\
=& \frac 1{w(\lambda^+_0)} d\log \frac{\sigma(u-u^+_0)}{\sigma(u-u^-_0)}
+ \frac 1{w(\lambda^+_0)} \Bigl(\zeta(u^+_0-u^-_0) -\frac 12 w'(\lambda^+_0) \Bigr)du.
\end{align*}
From
\begin{align*}
d\log \frac{\sigma(u-u^+_0)}{\sigma(u-u^-_0)} &= -\frac{2\eta_{\mathbf{a}}}
{\omega_{\mathbf{a}}}(u^+_0 -u^-_0) du + d\log \frac{\vartheta(F(\lambda^+_0,
 \lambda)+\nu,\tau)}{\vartheta(F(\lambda^-_0,\lambda)+\nu,\tau)},
\\
\zeta(u^+_0 -u^-_0) &= \frac{\sigma'}{\sigma}(u^+_0-u^-_0)=\frac{2\eta_{\mathbf
{a}}}{\omega_{\mathbf{a}}} (u^+_0-u^-_0) +\frac{\pi i}{\omega_{\mathbf{a}}}
+\frac 1{\omega_{\mathbf{a}}}\frac{\vartheta'}{\vartheta}
(F(\lambda^-_0, \lambda^+_0)+\nu,\tau)
\end{align*}
with $F(\lambda^{\pm}_0, \lambda) =\omega_{\mathbf{a}}^{-1}
\int^\lambda_{\lambda^{\pm}_0} d\lambda/w(\lambda),$
the desired formula follows.
\end{proof}
Observe that
\begin{align*}
\log \vartheta(F(\lambda^{\pm}_0, \lambda) +\nu,\tau)|_{\mathbf{a}}&=0,
\\
\log \frac{\vartheta(F(\lambda^+_0,\lambda)+\nu,\tau)}
{\vartheta(F(\lambda^-_0,\lambda)+\nu,\tau)}
\Bigr|_{\mathbf{b}} &= 
\log \frac{\vartheta(F(\lambda^+_0, \lambda_{\mathbf{b}})+\tau+\nu,\tau)
\vartheta(F(\lambda^-_0,
\lambda_{\mathbf{b}})+\nu, \tau)   }
{\vartheta(F(\lambda^-_0, \lambda_{\mathbf{b}})+\tau+\nu,\tau)
\vartheta(F(\lambda^+_0, \lambda_{\mathbf{b}})+\nu, \tau)   }
\\
&= \log \exp(-\pi i (2(F(\lambda^+_0, \lambda_{\mathbf{b}})+\nu) +\tau))
\\
& \phantom{---} + \log
\exp(\pi i (2(F(\lambda^-_0, \lambda_{\mathbf{b}})+\nu) +\tau)) 
\\
&= 2\pi i F(\lambda^-_0, \lambda^+_0)
\end{align*}
for $\lambda_{\mathbf{b}} \in \mathbf{b} \cap (\Pi_+)^{\mathrm{cl}} \cap (\Pi_-)
^{\mathrm{cl}}$, where $(\Pi_{+})^{\mathrm{cl}}$ denotes the closure of
$\Pi_{+}$, since $\vartheta(z\pm \tau,\tau)=
 e^{-\pi i (\tau\pm 2z)} \vartheta (z,\tau).$ 
Then
\begin{align*}
\int_{\mathbf{a}} \frac{d\lambda}{(\lambda-\lambda_0)w(\lambda)}
 &=-g_0(\lambda_0)\omega_{\mathbf{a}},
\\
\int_{\mathbf{b}} \frac{d\lambda}{(\lambda-\lambda_0)w(\lambda)}
 &=\frac{2\pi i}{w(\lambda^+_0)}F(\lambda^-_0,\lambda^+_0)
 +\tau \int_{\mathbf{a}} \frac{d\lambda}{(\lambda-\lambda_0)w(\lambda)}.
\end{align*}
Differentiation of both sides with respect to $\lambda_0$ at $\lambda_0=0$ yields
$$
\int_{\mathbf{b}} \frac{d\lambda}{\lambda^2 w(\lambda)} =
 \frac{4\pi i}{\omega_{\mathbf{a}}}
+\tau \int_{\mathbf{a}} \frac {d\lambda}{\lambda^2 w(\lambda)}.
$$
In what follows let us adopt the convention that the path of the integral
$\int^{\lambda_0^+}_{\lambda_0^-}w(\lambda)^{-1}d\lambda$ passes through
$\lambda_1$, i.e. the left end of the cut $[\lambda_1,\lambda_2]$. Then
\begin{align*}
\int^{\lambda_0^+}_{\lambda_0^-} \frac{d\lambda}{w(\lambda)} &=
2\int^{\lambda^+_0}_{\lambda_1}\frac{d\lambda}{w(\lambda)}=
-2\int^{-\lambda^+_0}_{-\lambda_1}\frac{d\lambda}{w(\lambda)}=
\\
&=-2\biggl(\int^{-\lambda^+_0}_{\lambda_1}+\int^{\lambda_1}_{-\lambda_1}
\biggr)\frac{d\lambda}{w(\lambda)}
=-\int^{-\lambda^+_0}_{-\lambda^-_0}\frac{d\lambda}{w(\lambda)} 
-\omega_{\mathbf{a}},
\end{align*}
which implies $-F(-\lambda^-_0,-\lambda^+_0)=F(\lambda^-_0,\lambda_0^+)+1.$
Using these formulas we have
\begin{prop} \label{prop6.4}
For $W(\lambda)$ as in Proposition $\ref{prop6.2}$ and for $\lambda_{\pm}$ by 
\eqref{6.2},
\begin{align*}
\int_{\mathbf{a}} W(\lambda) d\lambda &= -(w(\lambda_+)g_0(\lambda_+)
-w(\lambda_-)g_0(\lambda_-) +\lambda_+ -\lambda_-)\omega_{\mathbf{a}}
\\
&= -(\lambda_++\tfrac 12 w'(\lambda^+_+)-\lambda_--\tfrac 12 w'(\lambda^+_-))
 \omega_{\mathbf{a}}
\\
& \phantom{---} + \frac{\vartheta'}{\vartheta}(F(\lambda^-_+, \lambda^+_+) +\nu,\tau)
- \frac{\vartheta'}{\vartheta}(F(\lambda^-_-, \lambda^+_-) +\nu,\tau),
\\
 \Bigl(\int_{\mathbf{b}}-\tau \int_{\mathbf{a}} \Bigr) & W(\lambda) d\lambda =
2\pi i (F(\lambda^-_+, \lambda^+_+)-F(\lambda^-_-, \lambda^+_-))
=4\pi i F(\lambda^-_+, \lambda^+_+)+2\pi i+ O(t^{-1}),
\end{align*}
and
\begin{align*}
&\int_{\mathbf{a}} \frac{d\lambda}{\lambda w(\lambda)}=-g_0(0^+)\omega_{\mathbf{a}}, \quad
g_0(0^+)=-\frac 1{\omega_{\mathbf{a}}}\Bigl(\pi i + \frac{\vartheta'}{\vartheta}
(F(0^-,0^+)+\nu, \tau) \Bigr),
\\
&\Bigl(\int_{\mathbf{b}} -\tau \int_{\mathbf{a}} \Bigr)\frac{d\lambda}
{\lambda w(\lambda)}=2\pi i F(0^-,0^+),
\\
&\Bigl(\int_{\mathbf{b}} -\tau \int_{\mathbf{a}} \Bigr)\frac{d\lambda}
{\lambda^2w(\lambda)}
=\frac{4\pi i}{\omega_{\mathbf{a}}} .
\end{align*}
\end{prop}
\begin{rem}\label{rem6.2}
In the proposition above the first formula is rewritten in the form
\begin{align*}
\int_{\mathbf{a}} W(\lambda)d\lambda =& - 2\Bigl(\frac{\vartheta'}{\vartheta} 
( \tfrac 12 F(\lambda^-_+, \lambda^+_+) +\tfrac 14, \tau)
 -\frac{\vartheta'}{\vartheta}
( \tfrac 12 F(\lambda^-_-, \lambda^+_-) +\tfrac 14,\tau) \Bigr)
\\
 =& - 4\frac{\vartheta'}{\vartheta} 
( \tfrac 12 F(\lambda^-_+, \lambda^+_+) +\tfrac 14, \tau),
\end{align*}
in which the right-hand side is found by comparing the poles 
in both expressions. 
Note that $\Upsilon=-(\lambda_+ +w'(\lambda_+)/2)\omega_{\mathbf{a}} 
+(\vartheta'/\vartheta)(F(\lambda^-_+,\lambda^+_+)+\nu,\tau)$ is meromorphic
in $\lambda_+$ on $\Pi_{a_{\phi}}=\Pi_+\cup \Pi_-$. The part $\lambda_+
+w'(\lambda_+)/2$ has a pole $\sim 2\lambda_+$ at $\lambda_+= \infty^- \in\Pi_-$
and is analytic at $\lambda_+=\infty^+ \in \Pi_+.$ Since, for $k=\lambda_1/
\lambda_2,$
$$
\int^1_0 \frac{dz}{\sqrt{(1-z^2)(1-k^2z^2)}}=K
= \lambda_2\frac{\omega_{\mathbf{a}}}4,
\quad \int^{\infty}_0 \frac{dz}{\sqrt{(1-z^2)(1-k^2z^2)}}=iK'= \lambda_2
\frac{\omega_{\mathbf{b}}}2
$$
(cf. Section \ref{ssc7.1}), 
we have 
$$
\frac 12 F((\infty^-)^-,(\infty^-)^+)=\frac 12 \omega_{\mathbf{a}}^{-1}
\int^{(\infty^-)^+}_{(\infty^-)^-}\frac{d\lambda}{w}
= \omega_{\mathbf{a}}^{-1}\int^{\infty^-}_{\lambda_1}\frac{d\lambda}{w}=
\frac 14 -\frac{\tau}2\equiv -\frac 14 +\nu,
$$ 
and then we may write
$\Upsilon= -2(\vartheta'/\vartheta)(\frac 12 F(\lambda^-_+,\lambda^+_+)+
\frac 14,\tau) +C.$ Thus the equality above follows (see also \cite
[pp.~117--119]{Kitaev-3}).
\end{rem}
\subsection{Expression of $B_{\phi}(t)$}\label{ssc6.3}
Let us write the quantity $B_{\phi}(t)$ in terms of
\begin{align*}
\Omega_{\mathbf{a}} &=\int_{\mathbf{a}} \frac{d\lambda}{w(A_{\phi},\lambda)}, 
\quad
\Omega_{\mathbf{b}} =\int_{\mathbf{b}} \frac{d\lambda}{w(A_{\phi},\lambda)}, 
\\
\mathcal{J}_{\mathbf{a}} &=\int_{\mathbf{a}} \frac{w(A_{\phi},\lambda)}
{\lambda^2} d\lambda, \quad
\mathcal{J}_{\mathbf{b}} =\int_{\mathbf{b}} \frac{w(A_{\phi},\lambda)}
{\lambda^2} d\lambda
\end{align*}
with $w(A_{\phi},\lambda)=\sqrt{1-A_{\phi}\lambda^2+\lambda^4}$ and 
$\mathbf{a},$ $\mathbf{b}$ on 
$\Pi_{A_{\phi}}= \Pi_+ \cup \Pi_- =\lim_{a_{\phi}(t)
 \to A_{\phi}} \Pi_{a_{\phi}}$. 
\par
Let $0<\phi<\pi/2.$ 
By Proposition \ref{prop6.4}, the integral $\int_{\mathbf{a}} W(\lambda)
 d\lambda$ is expressed in terms of
$\vartheta_*=(\vartheta'/\vartheta)(\tfrac 12 F(\lambda^-_{+}, 
\lambda^+_{+})+\tfrac 14,\tau)$ in Remark \ref{rem6.2} or 
$w'(\lambda_{\pm}^+)$ and 
$(\vartheta'/\vartheta)( F(\lambda^-_{\pm}, \lambda^+_{\pm})+\nu,\tau)$, 
in which
$$
F(\lambda^-_{\pm}, \lambda^+_{\pm})=\frac 1{\omega_{\mathbf{a}}}
 \int^{\lambda^+_{\pm}}
_{\lambda^-_{\pm}} \frac{d\lambda}{w(a_{\phi},\lambda)}
=\frac 2{\omega_{\mathbf{a}}} \int^{\lambda^+_{\pm}}
_{\lambda_1} \frac{d\lambda}{w(a_{\phi},\lambda)}.
$$
Note that $\int_{\mathbf{a}} W(\lambda)d\lambda$ has no poles or zeros in  
$S_{\phi}(t'_{\infty},\kappa_1,\delta_1).$ 
Indeed, if, say $\vartheta_{*}=\infty$ at $t=t_*$, then $\lambda_{+}=\infty$, 
and hence $t_*$ is a zero of 
$y(t)$, which is excluded from $S_{\phi}(t'_{\infty}, \kappa_1,\delta_1).$  
Consider $\lambda_{+}=\lambda_{+}(t)$ (cf. \eqref{6.2}) moving on the
elliptic curve $\Pi_{a_{\phi}}$ crossing $\mathbf{a}$- and $\mathbf{b}$-cycles,
and then $F(\lambda^-_{+}, \lambda^+_{+})=2 p(t) + 2q(t)\tau +O(1)$
with $p(t),$ $q(t) \in \mathbb{Z}$. This implies the boundedness of
$\re (\vartheta'/\vartheta)(\tfrac 12 F(\lambda^-_{+}, \lambda^+_{+})+
\tfrac 14,\tau)$ in $S_{\phi}(t'_{\infty},\kappa_1,\delta_1)$, 
and hence the modulus of $\re \int_{\mathbf{a}}W(\lambda) d\lambda$  
is uniformly bounded in $S_{\phi}(t'_{\infty},\kappa_1,\delta_1)$.   
Note that, by \eqref{5.1}, 
\begin{align*}
\frac{1}{\lambda^2} (w(a_{\phi},\lambda)-w(A_{\phi},\lambda)) &= 
\frac 1{\lambda^2} (\sqrt{\lambda^4-a_{\phi}\lambda^2+ 1} - \sqrt{\lambda^4
-A_{\phi}\lambda^2+1})
\\
&= -\frac{t^{-1} B_{\phi}(t)} {2w(A_{\phi}, \lambda)} (1+O(t^{-1} B_{\phi}(t) )).
\end{align*}
By using this with $B_{\phi}(t)\ll 1$ and Proposition \ref{prop6.4}, 
the second formula in Proposition 
\ref{prop6.2} (i) is written in the form 
\begin{equation*}
\log\frac{g_{12} \hat{g}_{21}}{g_{22}\hat{g}_{11}}  
= -\frac{i e^{i\phi}t}{4} \int_{\mathbf{a}} \Bigl( 
\frac{w(A_{\phi},\lambda)}{\lambda^2}
-\frac{t^{-1}B_{\phi}(t)}{2w(A_{\phi},\lambda)} \Bigr) d\lambda
 -\frac 14 \int_{\mathbf{a}} W(\lambda) d\lambda +\pi i +O(t^{-\delta}),
\end{equation*}
which implies
\begin{equation*}
i e^{i\phi}\Bigl(  t\mathcal{J}_{\mathbf{a}} -\frac{\Omega_{\mathbf{a}}} 2
B_{\phi}(t) \Bigr)
 = - \int_{\mathbf{a}} W(\lambda) d\lambda +4\pi i - 4\log \frac{g_{12}
\hat{g}_{21}}{g_{22}\hat{g}_{11}} +O(t^{-\delta}).
\end{equation*}
Note that $(G,\hat{G})=(G^* v^{-\sigma_3/2}, \hat{G}^* v^{-\sigma_3/2})
=( (g_{ij}), (\hat{g}_{ij}) )$ with $g_{ij}=g_{ij}(t),$ $\hat{g}_{ij}=
\hat{g}_{ij}(t)$ is a solution of the direct monodromy problem.
Suppose that
\begin{equation}\label{6.3}
\Bigl|\log \frac{g_{12}\hat{g}_{21} } {g_{22}\hat{g}_{11}}\Bigr| \ll 1,  \quad
|\log(g_{11}g_{22}) |\ll 1 \quad \text{
in $S_{\phi}(t'_{\infty},\kappa_1,\delta_1).$}
\end{equation}
By the Boutroux equations \eqref{2.1}, 
$\im e^{i\phi} \Omega_{\mathbf{a}}B_{\phi}(t)$ is bounded as $e^{i\phi}t
\to \infty$ through $S_{\phi}(t'_{\infty},\kappa_1,\delta_1)$. 
By using the first formula of Proposition \ref{prop6.2} (i), we have
\begin{equation*}
i e^{i\phi}\Bigl(  t\mathcal{J}_{\mathbf{b}} -\frac{\Omega_{\mathbf{b}}} 2
B_{\phi}(t) \Bigr)
 = - \int_{\mathbf{b}} W(\lambda) d\lambda -2\pi i(\theta_0-\theta_{\infty})
+4\log({g_{11}}{g_{22}}) -4\pi i +O(t^{-\delta}),
\end{equation*}
in which $\int_{\mathbf{b}} W(\lambda)d\lambda$ admits an expression 
in terms of the
$\vartheta$-function with $\hat{\tau}=(-\omega_{\mathbf{a}})/
\omega_{\mathbf{b}},$ so that 
$\im e^{i\phi}\Omega_{\mathbf{b}}B_{\phi}(t) \ll 1$.  
From these facts it follows 
that $|B_{\phi}(t)|\le C_0$ in $S_{\phi}(t'_{\infty},\kappa_1,
\delta_1)$ for some $C_0>0.$ The implied constant of $B_{\phi}(t)\ll 1$ in
\eqref{5.1} may be supposed to be $2C_0$, 
which causes no changes
in the subsequent equations 
by choosing $t'_{\infty}$ larger if necessary. 
Thus, under \eqref{6.3}, the boundedness of $B_{\phi}(t)$ has been 
derived independently of \eqref{5.1}.
The case $-\pi/2<\phi<0$ is discussed in the same way.
\begin{rem}\label{rem6.3}
The argument above works also under a weaker condition, say $B_{\phi}(t) \ll
t^{(1-\delta)/2}.$ The supposition $B_{\phi}(t)\ll 1$ in \eqref
{5.1} also guarantees that each turning point is located within the distance
$O(t^{-1})$ from its limit one, which enables us to employ the limit Stokes
graph in the WKB analysis.
\end{rem}
\begin{prop}\label{prop6.5}
Suppose that $0<|\phi|<\pi/2$ and that $|\log(g_{11}g_{22})|,$ $|\log
\mathfrak{g}| \ll 1$ in $S_{\phi}(t'_{\infty},\kappa_1,\delta_1),$ where 
$$
\mathfrak{g}=\frac{g_{12}\hat{g}_{21}}{g_{22}\hat{g}_{11}} \quad \text{if
$0<\phi<\pi/2,$\quad  and}\,\,\,
\mathfrak{g}=\frac{g_{11}\hat{g}_{22}}{g_{21}\hat{g}_{12}} \quad \text{if
$-\pi/2 <\phi<0.$} 
$$
Then, 
in $S_{\phi}(t'_{\infty},\kappa_1,\delta_1),$ we have $B_{\phi}(t)\ll 1$ and 
\begin{equation*}
i e^{i\phi}\Bigl( t\mathcal{J}_{\mathbf{a}} -\frac{\Omega_{\mathbf{a}}} 2
B_{\phi}(t) \Bigr) =
 4\frac{\vartheta'}{\vartheta}(\tfrac 12 F(\lambda^-_+, \lambda_+^+)
 +\tfrac 14 , \tau)
 +4\pi i -4\log \mathfrak{g}+O(t^{-\delta}).
\end{equation*}
\end{prop}
\begin{rem}\label{rem6.51}
Conversely, \eqref{5.1}, i.e. $B_{\phi}(t) \ll 1 $ in $S_{\phi}(t'_{\infty},
\kappa_1, \delta_1)$ implies $|\log(g_{11}g_{22})|,$ $|\log \mathfrak{g}| \ll 1.$
\end{rem}
The following fact guarantees the possibility of limitation with respect to
$a_{\phi}$.
\begin{prop}\label{prop6.6}
Under the same supposition as in Proposition $\ref{prop6.5},$ we have
$$
\int^{\lambda^{+}_+}_{\lambda^-_+}  \frac{d\lambda}{w(a_{\phi},\lambda)}=  
\int^{\lambda^{+}_+}_{\lambda^-_+}  \frac{d\lambda}{w(A_{\phi},\lambda)}
 +O(t^{-1})
$$
uniformly in $\lambda^{\pm}_+$ as $t e^{i\phi} \to \infty$ through 
$S_{\phi}(t'_{\infty}, \kappa_1,\delta_1).$
\end{prop}
\begin{proof}
To show this proposition we note the lemma below, which follows from the
relations
\begin{align*}
& \int \frac{w}{\lambda^2} d\lambda= -\frac{6}{A_{\phi}} \int w d\lambda 
+ \Bigl(\frac{4}{A_{\phi}}-A_{\phi} \Bigr) \int \frac {d\lambda}w
 - \frac w{\lambda}  +\frac{2}{A_{\phi}} \lambda w,
\\
& \Omega_{\mathbf{b}} J_{\mathbf{a}} - \Omega_{\mathbf{a}} J_{\mathbf{b}}
= \frac{4A_{\phi}}{3} \pi i, \quad J_{\mathbf{a},\,\, \mathbf{b}}
=\int_{\mathbf{a},\,\, \mathbf{b}} w d\lambda
\end{align*}
with $w=w(A_{\phi},\lambda),$ the latter being obtained by the same way as 
in the proof of Legendre's relation \cite{HC}, \cite{WW}.
\begin{lem}\label{lem6.7}
$
\Omega_{\mathbf{a}}\mathcal{J}_{\mathbf{b}} -\Omega_{\mathbf{b}}\mathcal{J}
_{\mathbf{a}}= 8\pi i.
$
\end{lem}
From the boundedness of $B_{\phi}(t)$ it follows that $\omega_{\mathbf{a},\,
\mathbf{b}}=\Omega_{\mathbf{a},\, \mathbf{b}}+O(t^{-1}).$ By Propositions
\ref{prop6.2}, \ref{prop6.4} and Remark \ref{rem6.1}, in the case $0<\phi
<\pi/2,$
\begin{align*}
& \log (g_{11}g_{22})+ \tau \log\frac{g_{12}\hat{g}_{21}}{g_{22}\hat{g}_{11}}
\\
=& \Bigl(\int_{\mathbf{b}}-\tau\int_{\mathbf{a}}\Bigr) \Bigl(\frac{ie^{i\phi}t}
4 \cdot\frac{w(a_{\phi},\lambda)}{\lambda^2} +\frac 14 W(\lambda) \Bigr)d\lambda
+\frac{\pi i}2(\theta_0-\theta_{\infty}) +(1 +\tau) \pi i +O(t^{-\delta})
\\
=&-\frac{2\pi e^{i\phi}t}{\omega_{\mathbf{a}}}+{\pi i}F(\lambda^-_+, 
\lambda^+_+)  +O(1)
\\
=& -\frac{2\pi e^{i\phi}t}{\omega_{\mathbf{a}}} +\pi i\Bigl(p(t) +\frac{\omega
_{\mathbf{b}}}{\omega_{\mathbf{a}}} q(t) \Bigr) +O(1)=\Upsilon \ll 1
\end{align*}
with $p(t)$, $q(t) \in \mathbb{Z}$.  
Set $-e^{i\phi}\mathcal{J}_{\mathbf{a}}t/4 +\pi q(t) =X,$ $-e^{i\phi}\mathcal{J}
_{\mathbf{b}} t/4-\pi p(t) =Y,$ where $ X, Y\in \mathbb{R} $ by the Boutroux
equations \eqref{2.1}. Then, by $\omega_{\mathbf{a}}
\mathcal{J}_{\mathbf{b}} -\omega_{\mathbf{b}}\mathcal{J}_{\mathbf{a}}
=\Omega_{\mathbf{a}}\mathcal{J}_{\mathbf{b}} -\Omega_{\mathbf{b}}
\mathcal{J}_{\mathbf{a}} +O(t^{-1})$ and Lemma \ref{lem6.7}, 
\begin{align*}
\omega_{\mathbf{a}}\Upsilon =& -2\pi e^{i\phi}t -i(e^{i\phi}t(\Omega_{\mathbf{a}}
\mathcal{J}_{\mathbf{b}}-\Omega_{\mathbf{b}}\mathcal{J}_{\mathbf{a}})/4
+\omega_{\mathbf{b}}X-\omega_{\mathbf{a}}Y) +O(1)
\\
=& -i(\omega_{\mathbf{b}}X -\omega_{\mathbf{a}}Y) +O(1) \ll 1
\end{align*}
with $\im ( \omega_{\mathbf{b}}/\omega_{\mathbf{a}}) >0,$ which implies
$|X|,$ $|Y| \ll 1,$ and hence
$$
\pi p(t)= - e^{i\phi}\mathcal{J}_{\mathbf{b}}t/4 +O(1), \quad
\pi q(t)= e^{i\phi}\mathcal{J}_{\mathbf{a}}t/4 +O(1). 
$$
Since $w(a_{\phi},\lambda)^{-1} -w(A_{\phi},\lambda)^{-1}
=(\lambda^2/2)w(A_{\phi},\lambda)^{-3}B_{\phi}(t)t^{-1} +O(t^{-2})$, we obtain 
\begin{align*}
\biggl| & 
\int^{\lambda^+_+}_{\lambda^-_+} 
\Bigl( \frac 1{w(a_{\phi},\lambda)} -\frac 1{w(A_{\phi},
\lambda)} \Bigr) d\lambda \biggr|  \ll \biggl| 
\int^{\lambda^+_+}_{\lambda^-_+}  \frac{\lambda^2 B_{\phi}(t)
t^{-1}}{w(A_{\phi},\lambda)^3} d\lambda \biggr| +O(t^{-1})
\\
& \ll \biggl | t^{-1} \int^{\lambda^+_+}_{\lambda^-_+} 
 \frac{\lambda^2 d\lambda}{w(A_{\phi},\lambda)^3}
\biggr| +O(t^{-1})
\ll t^{-1} |p(t)j_{\mathbf{a}} +q(t) j_{\mathbf{b}} |+O(t^{-1})
\\
&\ll |\mathcal{J}_{\mathbf{b}} j_{\mathbf{a}}-\mathcal{J}_{\mathbf{a}}
j_{\mathbf{b}}| +O(t^{-1})
  = 2|(\partial/\partial A_{\phi}) (\mathcal{J}_{\mathbf{b}}\Omega_{\mathbf{a}}
 - \mathcal{J}_{\mathbf{a}}\Omega_{\mathbf{b}} ) |+O(t^{-1}) \ll t^{-1},
\end{align*}
where $j_{\mathbf{a},\mathbf{b}} = \int_{\mathbf{a},\mathbf{b}} \lambda^2
w(A_{\phi},\lambda)^{-3} d\lambda.$ This completes the proof of the proposition.
\end{proof}
\section{Proofs of the main theorems}\label{sc7}
\subsection{Proofs of Theorems \ref{thm2.1} and \ref{thm2.2}}\label{ssc7.1}
Suppose that $0<\phi <\pi/2$. Let us consider the inverse monodromy problem
for the prescribed matrices $G=(g_{ij}),$ $\hat{G}=(\hat{g}_{ij})
\in SL_2(\mathbb{C})$ with $g_{11}g_{12}g_{22}\hat{g}_{11}\hat{g}_{21}\not=0$.  
Set
\begin{align*}
 & \log({g_{11}}{g_{22}})+\tau \log\frac{g_{12}\hat{g}_{21}}{g_{22}\hat{g}_{11}}
\\
=& -\frac{2\pi e^{i\phi}t}{\omega_{\mathbf{a}}} 
 + {\pi i}F(\lambda^-_+, \lambda^+_+) + \frac{\pi i}2(\theta_0-\theta_{\infty}
 + 1) +(1+\tau) \pi i +O(t^{-\delta})
\end{align*}
(cf. Proof of Proposition \ref{prop6.6}), in which
$$
F(\lambda^-_+, \lambda^+_+) =\frac 2{\omega_{\mathbf{a}}}
\biggl(\int^{0^+}_{\lambda_1} +\int_{0^+}^{\lambda^+_+}\biggr) \frac{d\lambda}
{w(a_{\phi},\lambda)} = -\frac 12 +\frac{2}{\omega_{\mathbf{a}}} \int_{0^+}
^{\lambda^+_+} \frac{d\lambda}{w(a_{\phi},\lambda)}. 
$$
Hence, by Proposition \ref{prop6.6},
\begin{align*}
& \log({g_{11}}{g_{22}})+\tau \log\frac{g_{12}\hat{g}_{21}}{g_{22}\hat{g}_{11}}
\\
=& -\frac{2\pi e^{i\phi}t}{\omega_{\mathbf{a}}} 
 + \frac{2\pi i}{\omega_{\mathbf{a}}} \int^{\lambda^+_+}_{0^+} \frac{d\lambda}
{w(a_{\phi},\lambda)} + \frac{\pi i}2(\theta_0-\theta_{\infty})
 +(1+\tau) \pi i +O(t^{-\delta})
\\
=& -\frac{2\pi e^{i\phi}t}{\Omega_{\mathbf{a}}} 
 + \frac{2\pi i}{\Omega_{\mathbf{a}}} \int^{\lambda^+_+}_{0^+} \frac{d\lambda}
{w(A_{\phi},\lambda)} + \frac{\pi i}2(\theta_0-\theta_{\infty}+2)
 +\frac{\Omega_{\mathbf{b}}}{\Omega_{\mathbf{a}}} \pi i +O(t^{-\delta}),
\end{align*}
which yields
$$
-\int^{\lambda^+_+}_{0^+} \frac{d\lambda}{w(A_{\phi},\lambda)}
=i(e^{i\phi}t -2 x_0^+(G,\hat{G},\Omega_{\mathbf{a}},\Omega_{\mathbf{b}}))
+O(t^{-\delta})
$$
with
$$
2i x_0^+(G,\hat{G},\Omega_{\mathbf{a}},\Omega_{\mathbf{b}})=\frac 1 {2\pi i}
\Bigl(\Omega_{\mathbf{a}} \log(g_{11}g_{22}) +\Omega_{\mathbf{b}} \log
\frac{g_{12}\hat{g}_{21}}{g_{22}\hat{g}_{11}} \Bigr) -\frac{\Omega_{\mathbf{a}}}
4(\theta_0-\theta_{\infty}+2) -\frac{\Omega_{\mathbf{b}}}2.
$$
Observing that $\lambda_+= i y(x)^{-1},$ $e^{i\phi}t=\xi=2x,$ and that
\begin{align*}
\int^{\lambda^+_+}_{0^+} \frac{d\lambda}{w(A_{\phi},\lambda)}
&=\frac 1{\lambda_2} \int^{\lambda^+_+/\lambda_1}_{0^+} \frac{dz}
{\sqrt{(1-z^2)(1-(\lambda_1/\lambda_2)^2z^2)}}
\\
&=\frac 1{\lambda_2} \int^{iy^{-1}/\lambda_1}_{0^+} 
\frac{dz}{\sqrt{(1-z^2)(1-k^2z^2)}} +O(t^{-1}),
\end{align*}
in which $k=\lambda_1/\lambda_2,$ we find
$$
y^{-1}=i\lambda_1 \mathrm{sn}(i\lambda_2(e^{i\phi}t-2 x_0^+(G,\hat{G},\Omega
_{\mathbf{a}},\Omega_{\mathbf{b}})) +O(t^{-\delta}); \lambda_1/\lambda_2)
$$
with $2i x_0^+
 \mod \Omega_{\mathbf{a}} \mathbb{Z} +\Omega_{\mathbf{b}}\mathbb{Z}.$
For $-\pi/2<\phi <0,$ in a similar way we derive
$$
2i x_0^-(G,\hat{G},\Omega_{\mathbf{a}},\Omega_{\mathbf{b}})=\frac 1 {2\pi i}
\Bigl(\Omega_{\mathbf{a}} \log(g_{11}g_{22}) +\Omega_{\mathbf{b}} \log
\frac{g_{11}\hat{g}_{22}}{g_{21}\hat{g}_{12}} \Bigr) -\frac{\Omega_{\mathbf{a}}}
4(\theta_0-\theta_{\infty}+2) -\frac{\Omega_{\mathbf{b}}}2.
$$
These asymptotic forms of $y(x)$ coincide with those in Theorems \ref{thm2.1} 
and \ref{thm2.2}. The expression of $B_{\phi}(t)$ follows from
Proposition \ref{prop6.5}.
\begin{prop}\label{prop7.1}
In $S_{\phi}(t'_{\infty},\kappa_1,\delta_1)$, we have $B_{\phi}(t)\ll 1$ and
\begin{align*}
i e^{i\phi} \Bigl(t\mathcal{J}_{\mathbf{a}} -\frac{\Omega_{\mathbf{a}}}2
B_{\phi}(t) \Bigr) = & -4\frac{\vartheta'}{\vartheta}
\bigl(\Omega_{\mathbf{a}}^{-1} i(e^{i\phi}t-2x_0^{\pm}
(G,\hat{G},\Omega_{\mathbf{a}},
\Omega_{\mathbf{b}})),\Omega_{\mathbf{b}}\Omega_{\mathbf{a}}^{-1}\bigr)
\\
& + 4\pi i -4\log \mathfrak{g} +O(t^{-\delta}).
\end{align*}
\end{prop}
{\bf Justification.} In the argument above each derived asymptotic form of
$y(x)$ is a necessary condition to represent a solution corresponding to
the prescribed monodromy data. 
The justification of $y(x)$ as a solution of \eqref{1.1}
is made along the line in \cite[pp.~105--106, pp.~120--121]{Kitaev-3}.
Suppose $0<\phi <\pi/2$.  
Let $\mathcal{G}=(g_{11}g_{22},g_{12}\hat{g}_{21}(g_{22}\hat{g}_{11})^{-1})$ 
be a given point such that $g_{11}g_{12}g_{22}\hat{g}_{11}\hat{g}_{21}\not=0$ 
on the monodromy manifold for system \eqref{1.2} or \eqref{3.1}. Let
$$
(y_{\mathrm{as}})^{-1}=y_{\mathrm{as}}(\mathcal{G},t)^{-1}
=i\lambda_1 \mathrm{sn}(i\lambda_2(e^{i\phi}t-2 x_0^+(G,\hat{G},\Omega
_{\mathbf{a}},\Omega_{\mathbf{b}})); \lambda_1/\lambda_2) 
$$
and $(B_{\phi})_{\mathrm{as}}=(B_{\phi})_{\mathrm{as}}(\mathcal{G},t)$
be the leading term expressions of $y(t)^{-1}$ and $B_{\phi}(t)$
without $O(t^{-\delta})$ found in the argument above and Proposition 
\ref{prop7.1}. Viewing \eqref{4.2} and \eqref{5.1} we set
$$
y^*_{\mathrm{as}}=-y_{\mathrm{as}}t^{-1}
+e^{i\phi} \sqrt{y_{\mathrm{as}}^4 +A_{\phi}y_{\mathrm{as}}^2 +1 +(4e^{-i\phi}
(\theta_0y_{\mathrm{as}}^3+\theta_{\infty}y_{\mathrm{as}}) +(B_{\phi})
_{\mathrm{as}}y_{\mathrm{as}}^2)t^{-1} },
$$
in which the branch of the square root is chosen in such a way that 
$y^*_{\mathrm{as}}$ is compatible with $(d/dt)y_{\mathrm{as}}$. Then
$(y_{\mathrm{as}},y^*_{\mathrm{as}})=(y_{\mathrm{as}}(\mathcal{G},t),
y^*_{\mathrm{as}}(\mathcal{G},t) )$ fulfils \eqref{5.1} with $B_{\phi}(t)
=(B_{\phi})_{\mathrm{as}}(\mathcal{G},t)$ in the domain 
$$
\tilde{S}(\phi,t_{\infty}, \kappa_0, \delta_2)=\{t\,|\, \re t>t_{\infty},
|\im t|<\kappa_0 \} \setminus \bigcup_{i \sigma\in Z_0} \{|t-e^{-i\phi}\sigma|
<\delta_2 \}
$$
with $Z_0=2x_0^+i +(\tfrac 12 \Omega_{\mathbf{a}}\mathbb{Z} 
+\tfrac 12 \Omega_{\mathbf{b}}\mathbb{Z} )$.
Let $\mathcal{G}_{\mathrm{as}}(t)$ be the monodromy data for system \eqref{3.1}
containing $(y_{\mathrm{as}}, y^*_{\mathrm{as}})$. As a result of 
the direct monodromy problem by the WKB analysis we have $\|\mathcal{G}
_{\mathrm{as}}(t) -\mathcal{G} \| \le C t^{-\delta}$ for $|t|\ge t_{\infty}
(\mathcal{G})$ in a neighbourhood of $\mathcal{G},$ where $C$ and $\delta$ are
independent of $\mathcal{G}.$ 
Then the justification scheme of Kitaev \cite{Kitaev-1}
applies to our case combined with Proposition \ref{prop3.8}. Using 
the maximal modulus principle in each neighbourhood
of $i\sigma =2x_0^+ i +(\tfrac 12 \Omega_{\mathbf{a}} \mathbb{Z}
+  \Omega_{\mathbf{b}}\mathbb{Z})$, we obtain Theorem \ref{thm2.1}. Theorem
\ref{thm2.2} is similarly proved.
\subsection{Proof of Theorem \ref{thm2.3}}\label{ssc7.3}
Let \eqref{1.2} with $4\z=xy'y^{-2}+2x(1-y^{-2})-(2\theta_0-1)y^{-1}$ 
be an isomonodromy system.
Equation \eqref{1.1} and system \eqref{1.2}  
remain invariant under the substitution
$
y=e^{im\pi }\tilde{y}, \,\, x=e^{im\pi }\tilde{x},\,\,  
\lambda=e^{im\pi}\tilde{\lambda}, \,\, 
\phi=m\pi +\tilde{\phi}, 
$
with $\phi=\arg x,$ $\tilde{\phi}=\arg\tilde{x}.$
To show the theorem we use this symmetry (cf. \cite{Kitaev-2}). Let $\phi$
be such that $0<|\phi-m\pi|<\pi/2.$ Then a new system with respect to
$(\tilde{\lambda}, \tilde{y}, \tilde{x}, \tilde{\phi})$ is an isomonodromy 
system for $0<|\tilde{\phi}|<\pi/2.$ Denote by $G_m$ and $\hat{G}_m$ connection 
matrices corresponding to $G$ and $\hat{G}$ as the matrix monodromy data for 
the system governed by 
$\tilde{y}(\tilde{x})=e^{-im\pi } y(x)= e^{-im\pi } y(e^{im\pi}\tilde{x})$. 
We would like to know the relation between $(G_m,\hat{G}_m)$ and $(G,\hat{G})$. 
The matrix solutions of the new system are
$$
\tilde{U}^{\infty}_j(\tilde{\lambda}) \sim 
\exp(\tfrac 12 i\tilde{x} \tilde{\lambda} \sigma_3)
\tilde{\lambda}^{-\tfrac 12\theta_{\infty}\sigma_3}
$$
as $\tilde{\lambda} \to \infty$ through the sector $|\arg \tilde{\lambda}
 +\tilde{\phi} -j\pi | <\pi,$ and
$$
\tilde{U}^0_j(\tilde{\lambda}) \sim  \Delta_0
\exp(-\tfrac 12 i\tilde{x} \tilde{\lambda}^{-1} \sigma_3)
\tilde{\lambda}^{\tfrac 12\theta_0\sigma_3}
$$
as $\tilde{\lambda} \to 0$ through the sector $|\arg \tilde{\lambda}
 -\tilde{\phi} -j\pi | <\pi.$ 
The connection matrices are defined by
$\tilde{U}^{\infty}_0(\tilde{\lambda})=\tilde{U}^0_0(\tilde{\lambda}) G_m$ and
$\tilde{U}^{\infty}_1(\tilde{\lambda})=\tilde{U}^0_1(\tilde{\lambda})\hat{G}_m$. 
Note that $\tilde{U}^{\infty}_0(\tilde{\lambda})$
and $\tilde{U}^0_0(\tilde{\lambda})$ are also expressed as
$$
\tilde{U}^{\infty}_0(\tilde{\lambda})=\tilde{U}^{\infty}_0(e^{-im\pi}{\lambda})
 \sim 
\exp(\tfrac 12 i x\lambda \sigma_3){\lambda}^{-\tfrac 12 \theta_{\infty}
\sigma_3} e^{\tfrac 12 i\pi m\theta_{\infty} \sigma_3} 
$$
in the sector $|\arg {\lambda} +{\phi} -2m\pi | <\pi,$ and that
$$
\tilde{U}^0_0(\tilde{\lambda}) =\tilde{U}^0_0(e^{-i m\pi}{\lambda})
 \sim \Delta_0 \exp(-\tfrac 12 i x\lambda^{-1} \sigma_3)
{\lambda}^{\tfrac 12\theta_0\sigma_3}
e^{-\tfrac 12 i\pi m\theta_0 \sigma_3}
$$
in the sector $|\arg {\lambda} -{\phi}| <\pi.$ Then we have $\tilde{U}^0_0
(\tilde{\lambda})={U}^0_0(\lambda)e^{-\tfrac 12 i\pi m\theta_0 \sigma_3}$ and 
$\tilde{U}^{\infty}_0
(\tilde{\lambda})={U}^{\infty}_{2m}(\lambda)e^{\tfrac 12 i\pi m\theta_{\infty} 
\sigma_3}$. Similarly, $\tilde{U}^0_1(\tilde{\lambda})={U}^0_1(\lambda)
e^{-\tfrac 12 i\pi m\theta_0 \sigma_3}$ and $\tilde{U}^{\infty}_1
(\tilde{\lambda})={U}^{\infty}_{2m+1}(\lambda)e^{\tfrac 12 i\pi 
m\theta_{\infty} \sigma_3}$.   
Then, for $m\ge 1$, we have
\begin{align*}
&G_m= e^{\tfrac 12 i\pi m\theta_0\sigma_3} G S^{\infty}_0 S^{\infty}_1 \cdots
S^{\infty}_{2m-2} S^{\infty}_{2m-1} e^{\tfrac 12 i\pi m\theta_{\infty}\sigma_3}, 
\\
&\hat{G}_m= e^{\tfrac 12 i\pi m\theta_0\sigma_3} \hat{G} S^{\infty}_1 
S^{\infty}_2 \cdots
S^{\infty}_{2m-1} S^{\infty}_{2m} e^{\tfrac 12 i\pi m\theta_{\infty}\sigma_3}. 
\end{align*}
This combined with Proposition \ref{prop3.6} yields the expression of $G_m$
and $\hat{G}_m$ for $m\ge 1$ as in the theorem. The case $m\le 0$ is treated
similarly.
\par
Note that $F(u;A_{\phi})=\lambda_1\mathrm{sn}(\lambda_2u; \lambda_1/\lambda_2)$ 
solves $(F_u)^2=F^4-A_{\phi}F^2 +1.$ 
Then, using $F(e^{-im\pi}u;A_{\phi})=e^{-im\pi} F(u;A_{\phi}),$ 
we have, for $0<|\phi-m\pi|<\pi/2$,
\begin{align*}
i(e^{-im\pi} y(x))^{-1} &= i\tilde{y}(\tilde{x})^{-1}
= F(-2i(\tilde{x}-x_0^{(m)}
(\Omega^{\tilde{\phi}}_{\mathbf{a}},\Omega^{\tilde{\phi}}
_{\mathbf{b}},G_m,\hat{G}_m)) +O(x^{-\delta});A_{\phi})
\\
&= e^{-im\pi} F(-2i({x}- e^{im\pi}
 x_0^{(m)}(\Omega^{\tilde{\phi}}_{\mathbf{a}},\Omega^{\tilde{\phi}}_{\mathbf{b}},
G_m,\hat{G}_m))+O(x^{-\delta});A_{\phi}).
\end{align*}
Since $\Omega^{\phi}_{\mathbf{a},\, \mathbf{b}} =e^{i m\pi} \Omega^{\phi- \pi}
_{\mathbf{a},\, \mathbf{b}}$, we have
$$
i y(x)^{-1}=F(-2i({x}- x_0^{(m)}
(\Omega^{\phi}_{\mathbf{a}},\Omega^{\phi}_{\mathbf{b}},
G_m,\hat{G}_m))+O(x^{-\delta});A_{\phi}),
$$
which is a desired solution.
\section{Modulus $A_{\phi}$ and the Boutroux equations}\label{sc8}
For a given $\phi \in \mathbb{R}$ we 
seek the modulus $A=A_{\phi}\in \mathbb{C}$ such that for every cycle
$\mathbf{c} \subset \Pi_A$
$$
\im e^{i\phi} \int_{\mathbf{c}} \frac{w(A,\lambda)}{\lambda^2} d\lambda=0
$$
for $w(A,\lambda)^2=\lambda^4-A\lambda^2+1.$
We note the following.
\begin{lem}\label{lem8.0}
As long as $A\in \mathbb{C}\setminus \{c<-2\},$ the polynomial $\lambda^4-A
\lambda^2+1$ has zeros $\lambda_1=\lambda_1(A)$ and $\lambda_2=\lambda_2(A)$
continuous in $A$ and satisfying $\lambda_1\lambda_2=1$ and $\lambda_1,\lambda_2
\in \{\re \lambda \ge 0\}.$
\end{lem}
\begin{proof}
For $A$, say, close to $2$ there exist such zeros. If $\lambda_1$ or $\lambda_2
\in i\mathbb{R}$, then $A/2\pm \sqrt{A^2/4-1} =-r_0 <0$, implying
$A=-(r_0^2+1)/r_0 \le -2.$
\end{proof}
By Lemma \ref{lem8.0}, for each $a\in \mathbb{C}$, there exist zeros
$\lambda_1,\lambda_2 \in \{\re \lambda\ge 0\}.$ 
Define the elliptic curve $\Pi_A=\Pi_+\cup \Pi_-$ glued along
the cuts $[-\lambda_1,-\lambda_2]$ and $[\lambda_1,\lambda_2]$ as in Section
\ref{ssc2.2}. Let $\mathbf{a}$ and $\mathbf{b}$ be basic cycles on $\Pi_A$
such that $\tfrac 12 \mathbf{a}=(-\lambda_1,\lambda_1)^{\sim}$,
$\tfrac 12 \mathbf{b}=(\lambda_1,\lambda_2)^{\sim}$ on the upper sheet $\Pi_+$,
where $(\lambda_1,\lambda_2)^{\sim}$ is located to the left of the cut $[\lambda
_1,\lambda_2].$ The cycles thus defined may be supposed to 
coincide with $\mathbf{a}$ and $\mathbf{b}$ in Figure \ref{cycles1} 
if $A=A_{\phi}.$
For $|\phi|\le \pi/2$, let us consider the Boutroux equations
$$
(\mathrm{BE})_{\phi}: \quad\qquad \im e^{i\phi} I_{\mathbf{a}}(A)=0, \quad   
\im e^{i\phi} I_{\mathbf{b}}(A)=0, \phantom{------------}
$$
where
$$
I_{\mathbf{a}, \, \mathbf{b}} (A):=\int_{\mathbf{a},\,\mathbf{b}}
 \frac{w(A,\lambda)}{\lambda^2} d \lambda  
=\int_{\mathbf{a},\,\mathbf{b}} \frac{1}{\lambda^2} 
\sqrt{\lambda^4-A\lambda^2+1} \, d\lambda.
$$
It is easy to see that $w(A,\lambda)^2=\lambda^4-A\lambda^2+1$ has double 
roots if and only if 
\begin{equation*}
A= 2, \,\,\,\, \lambda_1,\lambda_2=1, \,\,\,\, -\lambda_{1},-\lambda_{2}=-1;
\quad
A=- 2, \,\,\,\, \lambda_1,-\lambda_{2} =\pm i, 
\,\,\,\, -\lambda_{1},\lambda_{2}= \mp i.
\end{equation*}
\begin{exa}\label{exa8.1}
We have $I_{\mathbf{a}}(2)=-8,$ $I_{\mathbf{b}}(2)=0,$ and
$I_{\mathbf{a}}(-2)=0,$ $I_{\mathbf{b}}(-2)= 8i.$ 
Indeed, say,
$$
I_{\mathbf{a}}( 2)= 2\int^{1}_{-1}\frac {1-\lambda^2}{\lambda^2}d\lambda=-8, 
$$
in which the residue of the integrand at $\lambda=0$ vanishes.
Hence $A_0=2$ (respectively, $A_{\pm \pi/2}=-2$) 
fulfils (BE)$_0$ (respectively, (BE)$_{\pm \pi/2}$). 
\end{exa}
In accordance with \cite[Section 7]{Kitaev-2} we begin with the following: 
\begin{prop}\label{prop8.1}
Suppose that $\im I_{\mathbf{a}}(A)=0.$ Then $A\in \mathbb{R}.$
\end{prop}
\begin{proof}
First we treat the case where $0\le \re \lambda_1 <\re \lambda_2,$ or
where $0\le \re\lambda_1= \re\lambda_2$ and $|\im \lambda_1| <|\im \lambda_2|.$ 
Note that $\frac 12 I_{\mathbf{a}}(A)$ is the integral along the segment
joining $-\lambda_1$ to $ \lambda_1$ on $\Pi_+$.  
Let $\Pi_*=\Pi_*^+ \cup \Pi_*^-$ be the two-sheeted Riemann surface glued
along the cuts $[-\lambda_1,\lambda_1]$ and $[-\lambda_2,\lambda_2]=[-\lambda_2,
\infty]\cup [\infty,\lambda_2],$ and let 
$w_*(A,\lambda)=\sqrt{\lambda^4-A\lambda^2+1}$ be 
the function on $\Pi_*$ such that, around $\lambda=\varepsilon i \in
\Pi_*^+$ with small $\varepsilon >0$, $w_*(A,\lambda)$ coincides with $w(A,
\lambda).$ Then
$$
I_{\mathbf{a}}(A)=I_{\mathbf{a}_*}(A):=\int_{\mathbf{a}_*} \frac{w_*(A,\lambda)}
{\lambda^2} d\lambda,
$$
where $\mathbf{a}_*$ is a cycle on $\Pi^+_*$ surrounding the cut $[-\lambda_1,
\lambda_1]$ in the clockwise direction. Set 
$$
J_{\mathbf{a}_*}(A) =\int_{\mathbf{a}_*} \frac{v(A,\lambda)}{\lambda^2} d\lambda,
\quad v(A,\lambda)=\sqrt{-\lambda^4+A\lambda^2-1},
$$
and suppose that the cycle $\mathbf{a}_*$ surrounds $\pm\overline{\lambda_1}$
as well. Then the supposition $\im I_{\mathbf{a}}(A)=0$ means
$J_{\mathbf{a}_*}(A)= iI_{\mathbf{a}_*}(A) =iI_{\mathbf{a}}(A)\in i\mathbb{R}$, 
and 
\begin{align*}
0=& J_{\mathbf{a}_*}(A)+\overline{J_{\mathbf{a}_*}(A)} 
=J_{\mathbf{a}_*}(A)+J_{\overline{\mathbf{a}_*}}(\overline{A}) 
=J_{\mathbf{a}_*}(A)-J_{\mathbf{a}_*}(\overline{A}) 
\\
=& \int_{\mathbf{a}*} \frac 1{\lambda^2} (v(A,\lambda)-v(\overline{A},\lambda))
d\lambda
= (A-\overline{A}) \int_{\mathbf{a}_*} \frac{d\lambda}{v(A,\lambda)+v(\overline
{A},\lambda)}.
\end{align*}
In this proof, to simplify the description, we write 
$v(A,\lambda)=v_A(\lambda)$, $v(\overline{A},\lambda)
=v_{\overline{A}}(\lambda)$ and $v(A,\lambda)\pm v(\overline{A},\lambda)
=(v_A\pm v_{\overline{A}})(\lambda)$. 
The polynomial $v_{\overline{A}}(\lambda)^2=-w(\overline{A},\lambda)^2$ has the 
roots $\pm \overline{\lambda_1},$ $\pm \overline{\lambda_2}$. The algebraic 
functions $(v_A \pm v_{\overline{A}})(z)$ may be considered on $\Pi_*$ with
the additional cuts
$[-\overline{\lambda_1}, \overline{\lambda_1}]$ and
$[-\overline{\lambda_2}, \overline{\lambda_2}]=
[-\overline{\lambda_2},\infty]\cup [\infty, \overline{\lambda_2}]$, that is,
the two-sheeted Riemann surface glued along the cuts $[-\lambda_1, \lambda_1]$,
$[-\overline{\lambda_1}, \overline{\lambda_1}]$, 
$[-\lambda_2,\infty]\cup[\infty,\lambda_2]$ and
$[-\overline{\lambda_2},\infty]\cup [\infty, \overline{\lambda_2}]$.
The cycle $\mathbf{a}_*$ may be supposed to enclose both cuts $[-\lambda_1,
\lambda_1]$ and $[-\overline{\lambda_1},\overline{\lambda_1}].$ 
These related cuts, say as in Figure \ref{cycle3} (1), are modified 
into $[-\lambda_1,-\overline{\lambda_1}]
\cup [\lambda_1,\overline{\lambda_1}] \cup [-\re \lambda_1,\re \lambda_1]$ 
and then the cycle $\mathbf{a}_*$ as in Figure \ref{cycle3} (1)
may be deformed into a contour consisting of 
horizontal and vertical lines as in Figure \ref{cycle3} (2). 
To show $A\in \mathbb{R}$ it is sufficient to verify that, under the supposition
$A-\overline{A}\not=0,$
$$
J=\int_{\mathbf{a}_*} \frac{d\lambda}{(v_A + v_{\overline{A}})(\lambda)} \not=0.
$$
\par
Let us compute this integral along the contour $\mathbf{a}_*$ as in Figure
\ref{cycle3} (2) with vertices $\alpha \pm i\beta,$ $-\alpha \pm i{\beta}$,
in which $\alpha=\re\lambda_1,$ $\beta=|\im \lambda_1|.$


{\small
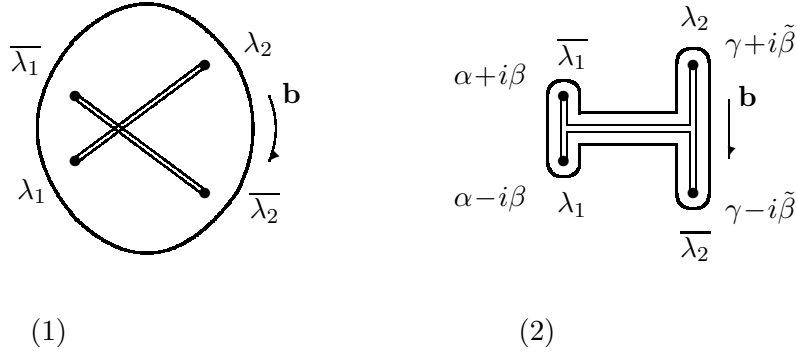
\begin{figure}[htb]
\begin{center}
\unitlength=0.85mm
\begin{picture}(60,65)(-30,-35)
\put(15,10){\circle*{1.5}}
\put(15,-10){\circle*{1.5}}
\put(-5,10){\circle*{1.5}}
\put(-5,-10){\circle*{1.5}}

 \put(9.5,22.5){\vector(3,-2){0}}

 \put(22,14){\makebox{$\overline{\lambda_{1}}$}}
 \put(22,-16){\makebox{${\lambda_{1}}$}}
 \put(-21,12){\makebox{$-{\lambda_{1}}$}}

 \put(-21,-18){\makebox{$-\overline{\lambda_{1}}$}}

 \put(-5,23.5){\makebox{$\mathbf{a}_*$}}

 \qbezier (15,-10.7) (5,-0.7) (-5,9.3)
 \qbezier (15,-9.3) (5,0.7)    (-5,10.7)
\qbezier (-5,-10.7) (5,-0.7) (15,9.3)
\qbezier (-5,-9.3) (5,0.7) (15,10.7)

 \qbezier (-4.0,20.5) (3.0,24.4) (9.5,22.5)

\thicklines
\qbezier (19.5,11.5) (6.5,26.3) (-6.5,14)
\qbezier  (20,-10.1) (6.8,-27.2) (-8.3,-12.1)

\qbezier (23.1,2.0) (23.2,-5.2) (20,-10.1)
\qbezier (23.1,2.0) (22.8,6.9) (19.5,11.5)

\qbezier (-12.1,-1.0) (-12.1,-7.0) (-8.3, -12.1)
    \qbezier (-6.5,14) (-12.5,7.5) (-12.1, -1.0)

\put(-12,-33){\makebox{(1) }}
\end{picture}
\qquad \qquad
\begin{picture}(70,65)(-35,-35)

\put(15,10){\circle*{1.5}}
\put(15,-10){\circle*{1.5}}
\put(-5,10){\circle*{1.5}}
\put(-5,-10){\circle*{1.5}}

 \put(20.5,-3.0){\vector(-1,-3){0}}

 \put(-8,16){\makebox{$-{\lambda_{1}}$}}
 \put(13,16){\makebox{$\overline{\lambda_{1}}$}}
 \put(-8,-20){\makebox{$-\overline{\lambda_{1}}$}}
 \put(13,-20){\makebox{${\lambda_{1}}$}}

 \put(-27,11){\makebox{$-\alpha \!+\!i\beta$}}
 \put(-27,-14){\makebox{$-\alpha\! -\!i\beta$}}
 \put(22,12){\makebox{$\alpha\! +\!i{\beta}$}}
\put(22,-14){\makebox{$\alpha\! -\!i{\beta}$}}

   \put(23.5,2.5){\makebox{$\mathbf{a}_*$}}

 \qbezier (15.5,10) (15.5,0) (15.5,-10)
 \qbezier (14.5,10) (14.5,5) (14.5, 0.5)

 \qbezier (14.5,-10) (14.5,-5) (14.5,-0.5)

 \qbezier (-5.5,10) (-5.5,0) (-5.5,-10)
 \qbezier (-4.5,10) (-4.5,5) (-4.5,0.5)

 \qbezier (-4.5,-10) (-4.5,-2) (-4.5,-0.5)
 \qbezier (-4.5,0.5) (0,0.5) (14.5,0.5)
 \qbezier (-4.5,-0.5) (0,-0.5) (14.5,-0.5)

   \qbezier (20.5,5.5) (20.5,0.5) (20.5,-3.0)

\thicklines

   \qbezier (17.5, 10) (17.5,0) (17.5,-10)
 \qbezier (12.5,10) (12.5,5) (12.5, 2.5)
 \qbezier (12.5,-10) (12.5,-6.2) (12.5, -2.5)

 \qbezier (-7.5,10) (-7.5,-2) (-7.5,-10)

 \qbezier (-2.5,10) (-2.5,2) (-2.5,2.5)
 \qbezier (-2.5,-10) (-2.5,-6.2) (-2.5,-2.5)
 \qbezier (-2.5,2.5) (0,2.5) (12.5,2.5)
 \qbezier  (-2.5,-2.5) (5,-2.5) (12.5,-2.5)

   \qbezier (17.5, 10) (17.5, 12.5) (15,12.5)
   \qbezier (12.5, 10) (12.5, 12.5) (15,12.5)

 \qbezier (-7.5,10) (-7.5,12.5) (-5,12.5)
 \qbezier (-2.5,10) (-2.5,12.5) (-5,12.5)

   \qbezier (17.5, -10) (17.5,-12.5) (15,-12.5)
   \qbezier (12.5, -10) (12.5, -12.5) (15,-12.5)

 \qbezier (-7.5,-10) (-7.5,-12.5) (-5,-12.5)
 \qbezier (-2.5,-10) (-2.5,-12.5) (-5,-12.5)

\put(-12,-33){\makebox{$(2)$ }}
\end{picture}

\end{center}

\caption{Modification of the cycle $\mathbf{a}_*$}
\label{cycle3}
\end{figure}
}

The integral $J$ is decomposed into three parts: $J= 2 J_0 +J_{\mathrm{rver}}
+J_{\mathrm{lver}}$ with the real line part
$$
J_0= \int^{\alpha}_{-\alpha} \frac{dt}{(v_A+v_{\overline{A}})(t)},
$$
the right vertical part
$J_{\mathrm{rver}}=J_{\mathrm{rver}}^+ + J_{\mathrm{rver}}^-,$ in which
\begin{align*}
J_{\mathrm{rver}}^+&=\int_{0}^{{\beta}}\frac{i dt}{(v_A+v_{\overline{A}})
(\alpha+it)}
+ \int^{0}_{{\beta}}\frac{i dt}{(v_A-v_{\overline{A}})(\alpha+it)},
\\
J_{\mathrm{rver}}^- &= \int_{0}^{-{\beta}}\frac{idt}{(v_A-v_{\overline{A}})
(\alpha+it)}
+ \int^{0}_{-{\beta}}\frac{idt}{(-v_A-v_{\overline{A}})(\alpha+it)},
\end{align*}
and the left vertical part
$J_{\mathrm{lver}}=J_{\mathrm{lver}}^+ + J_{\mathrm{lver}}^-,$ in which
\begin{align*}
J_{\mathrm{lver}}^+ &= \int^{-\beta}_{0}\frac{i dt}{(-v_A-v_{\overline{A}})
(-\alpha+it)}
+ \int^{0}_{-\beta}\frac{i dt}{(-v_A+v_{\overline{A}})(-\alpha+it)},
\\
J_{\mathrm{lver}}^- &= \int_{0}^{\beta}\frac{idt}{(-v_A+v_{\overline{A}})
(-\alpha+it)}
+ \int^{0}_{\beta}\frac{idt}{(v_A+v_{\overline{A}})(-\alpha+it)}.
\end{align*}
Then we have
$$
J_{\mathrm{rver}}=\frac {-2i}{A-\overline{A}} \Bigl( \int^{{\beta}}_{0}
\frac{v_{\overline{A}}(\alpha+it)}{(\alpha+it)^2}dt
+ \int^{{\beta}}_{0}
\frac{v_{{A}}(\alpha-it)}{(\alpha-it)^2}dt \Bigr) \in \mathbb{R}
$$
and
$$
J_{\mathrm{lver}}=\frac{-2i}{A-\overline{A}} \Bigl(\int^{\beta}_0
\frac{v_A(-\alpha+it)}{(-\alpha+it)^2}dt
+ \int^{\beta}_0 \frac{v_{\overline{A}}(-\alpha-it)}{(-\alpha-it)^2} dt \Bigr)
\in \mathbb{R},
$$
and hence $J_{\mathrm{rver}}+J_{\mathrm{lver}} \in \mathbb{R}.$
Furthermore, observing 
\begin{align*}
v_{A}(t)&= i \sqrt{ \sigma(t)+i \tau(t)}, \quad
v_{\overline{A}}(t)= i \sqrt{ \sigma(t)-i \tau(t)}, \quad
\\
\sigma(t)&= t^4-\re A\cdot t^2 +1,  \quad \tau(t)= -\im A\cdot t^2
\end{align*} 
and $v_A(0)=v_{\overline{A}}(0)=i$, we have
\begin{equation*}
J_0 = \frac{-i}{\sqrt 2} \int^{\alpha}_{-\alpha}\frac{dt}
{\sqrt{\sigma(t)+ \sqrt{\sigma(t)^2+\tau(t)^2} }},
\end{equation*}
which implies $J_0 \in i\mathbb{R}$ and $J_0\not=0$, provided that 
$\alpha \not=0.$
Thus in the case $\re \lambda_1 \not= 0$ it is shown that $A\in
\mathbb{R}.$ If $\re \lambda_1=0$, then from $\lambda_1\lambda_2
=1$ it follows that $A=\lambda_1^2+\lambda_2^2 \in \mathbb{R}.$
\par
In the case where $0\le \re\lambda_2<\re\lambda_1$, or where $0\le \re\lambda_2
=\re\lambda_1$ and $|\im \lambda_2|<|\im \lambda_1|$, the same argument as above
applies. The remaining case $\lambda_2=\lambda_1$ or $\lambda_2
=\overline{\lambda_1}$ occurs only when $A=\pm 2.$ Thus we have the proposition. 
\end{proof}
Let us examine $I_{\mathbf{b}}(A)$ for $A\in \mathbb{R}$. It is easy to see that,
$w(A,z)^2$ has real roots $-\lambda_2<-\lambda_1 <\lambda_1 <\lambda_2$ if 
$A^2> 4.$ Then $I_{\mathbf{b}}(A) =2\int_{[\lambda_1,\lambda_2]} \lambda^{-2}
w(A,\lambda)d\lambda \in i \mathbb{R}\setminus \{0\}.$ If $A=\pm 2,$ then
$\lambda_1
=\lambda_2=1$ or $\lambda_1=-\lambda_2=\pm i$, and 
$I_{\mathbf{b}}(2)=0$ and $I_{\mathbf{b}}(-2)= 8i$ as in Example \ref{exa8.1}.
\par
Suppose that $0\le A^2 < 4.$ The roots of $w(A,z)^2$ are $\alpha\pm i\beta,$
$-(\alpha\pm i\beta)$ with $\alpha,$ $\beta\ge 0.$ Then $\mathbf{b}$ is
a cycle enclosing the cut $[\alpha-i\beta, \alpha+i\beta]$. 
Since $\overline{I_{\mathbf{b}}(A)} =-I_{\mathbf{b}}(A),$ we have
$I_{\mathbf{b}}(A) \in i\mathbb{R}.$ 
Since $w(A,\alpha\pm i\beta)=0$, the integral
\begin{equation*}
I_{\mathbf{b}}(A) = 2i \int^{\beta}_{-\beta} \frac{w(A, \alpha+it)}{(\alpha+
it)^2} dt = 4i \int^{\beta}_0 \re \frac{w(A,\alpha+it)}{(\alpha+it)^2} dt
\end{equation*}
satisfies, for $-2 <A<2,$
$$
\frac{\partial}{\partial A} \Bigl(\frac 1i I_{\mathbf{b}}(A) \Bigr)
= 2\int^{\beta}_0 \re w(A, \alpha+it)^{-1} dt =\sqrt{2} \int^{\beta}_0
\frac{\sqrt{{g} +\sqrt{{g}^2+{h}^2}}}{\sqrt{{g}^2+{h}^2}} dt >0,
$$
where
\begin{align*}
g={g}(t)&=\re w(A,\alpha+it)^2 =t^4 -6\alpha^2 t^2 +\alpha^4 -A\alpha^2 
+At^2+1,
\\
h={h}(t)&=\im w(A,\alpha+it)^2 =-4\alpha t^3+4\alpha^3 t-2A\alpha t.
\end{align*}
This implies $I_{\mathbf{b}}(A) \in i \mathbb{R} \setminus \{0\}$ for $-2<A<2.$
\par
The fact above combined with Proposition \ref{prop8.1} and Example \ref{exa8.1}
implies the following.
\begin{prop}\label{prop8.2}
If $\phi=0$, then the Boutroux equations $(\mathrm{BE})_0$ admit a unique
solution $A_0= 2.$ 
\end{prop}
\begin{cor}\label{cor8.3}
For every $A\in \mathbb{C}$, $(I_{\mathbf{a}}(A), I_{\mathbf{b}}(A))\not=(0,0).$
\end{cor}
\begin{prop}\label{prop8.4}
Suppose that, for $A_{\phi}$ solving $(\mathrm{BE})_{\phi}$ with $0<|\phi|\le
\pi/2$, the elliptic curve $\Pi_{A_{\phi}}$ degenerates. Then $\phi=\pm \pi/2$ 
and $A_{\pm \pi/2}=-2.$
\end{prop}
\begin{proof}
When $\Pi_{A_{\phi}}$ degenerates, $A_{\phi}= \pm 2.$
Suppose that $A_{\phi}= -2.$ Then the roots of $w(A_{\phi},\lambda)^2$ are 
$\lambda_1=-\lambda_2=\pm i$, $-\lambda_{1}=\lambda_{2}=\mp i$, and 
$$
\mathbb{R}\ni e^{i\phi} \int^{i}_{-i} \frac 1{\lambda^2} \sqrt{\lambda^4
-A_{\phi}\lambda^2 +1}\, d\lambda= e^{i(\phi- \pi/2)} \int^{1}_{-1} 
\frac{1}{\zeta^2}\sqrt{\zeta^4+2e^{\pi i} \zeta^2 +1}\, d\zeta \not=0,
$$
which implies $\phi=\pm \pi/2.$ Similarly, if $A_{\phi}= 2,$ then $\phi=0.$
\end{proof}
\begin{prop}\label{prop8.5}
If $\phi=\pm \pi/2,$ then the Boutroux equations $(\mathrm{BE})_{\pm \pi/2}$ 
admit a unique solution $A_{\pm \pi/2}=- 2.$
\end{prop}
\begin{proof}
For $\phi=\pi/2$, $(\mathrm{BE})_{\pi/2}$ are equivalent to
$$
e^{\pi i/2} \int_{\mathbf{c}} \frac 1{\lambda^2} \sqrt{\lambda^4-A_{\pi/2} 
\lambda^2+1} \, d\lambda \in \mathbb{R}
$$
for every cycle $\mathbf{c}$ on $\Pi_{A_{\pi/2}}$, which is written as
$(\mathrm{BE})_0$ with $\phi=0$
$$
 \int_{\mathbf{c} e^{-\pi i/2}} \frac 1{\zeta^2} \sqrt{\zeta^4 
-e^{\pi i}A_{\pi/2} \zeta^2  +1}\, d\zeta \in\mathbb{R}  \quad (\lambda
=e^{\pi i/2}\zeta).
$$
Then by Proposition \ref{prop8.2}, $e^{\pi i}A_{\pi/2}= 2,$
i.e. $A_{\pi/2}=- 2$ is a unique solution of $(\mathrm{BE})_{\pi/2}.$
\end{proof}
\par
The quotient $h(A)=I_{\mathbf{b}}(A)/I_{\mathbf{a}}(A)$ 
\cite[Appendix I]{Novokshenov-2} is useful in examining $A_{\phi}$.
\begin{prop}\label{prop8.6}
Suppose that $A\in \mathbb{C}.$ 
\par
$(1)$ If $A$ solves $(\mathrm{BE})_{\phi}$ for some $\phi \in \mathbb{R}$ and
$I_{\mathbf{a}}(A) \not= 0,$ then $h(A) \in \mathbb{R}.$
\par
$(2)$ If $h(A)\in \mathbb{R} \setminus \{0\}$, then, for some $\phi \in \mathbb
{R}$, $A$ solves $(\mathrm{BE})_{\phi}$. 
\end{prop}
\begin{proof}
Suppose that $h(A)=\rho \in \mathbb{R}\setminus\{0\},$ and write 
$I_{\mathbf{a}}(A)=u+iv,$
$I_{\mathbf{b}}(A)=U+iV$ with $u$, $v$, $U,$ $V\in \mathbb{R}$. Then $U=\rho u,$
$V=\rho v,$ and hence $v/u=V/U= -\tan \phi$ for some $\phi \in [-\pi/2,\pi/2].$ 
This implies $\im e^{i\phi}I_{\mathbf{b}}(A)=\im e^{i\phi}I_{\mathbf{a}}(A)=0.$
\end{proof}
\begin{prop}\label{prop8.7}
The set $\{ A\in \mathbb{C} \,|\, \text{$A$ solves $(\mathrm{BE})_{\phi}$ for
some $\phi \in\mathbb{R}$} \}$ is bounded.
\end{prop}
\begin{proof}
Note that $w(A,z)$ admits the roots $\lambda_{\pm 1} \sim \pm A^{-1/2},$ 
$\lambda_{\pm 2} \sim \pm A^{1/2},$ if $A$ is large. Then
\begin{equation*}
 \int^{\lambda_2}_{\lambda_1} \frac {w(A,\lambda)}{\lambda^2} d\lambda \sim
\int^{A^{1/2}}_{A^{-1/2}} \frac 1{\lambda^2} \sqrt{\lambda^4-A\lambda^2+1}\,
d\lambda  \sim  iA^{1/2} \int^{1}_{A^{-1}} \frac 1t \sqrt{1-t^2} \, dt
 \sim i A^{1/2} \log A,
\end{equation*}
and
$$
 \int^{\lambda_{-1}}_{\lambda_1} \frac {w(A,\lambda)}{\lambda^2} d\lambda \sim
\int^{-A^{-1/2}}_{A^{-1/2}} \frac 1{\lambda^2} \sqrt{\lambda^4-A\lambda^2+1}\,
d\lambda \sim  
- A^{1/2} \int^{1}_{-1} \frac 1{t^2} \sqrt{1-t^2} \, dt \sim \pi A^{1/2}.
$$
This implies $h(A) \not\in \mathbb{R}$ if $A$ is sufficiently large, which
completes the proof.
\end{proof}
\par
The following fact is used in discussing solutions of (BE)$_{\phi}$.
\par
Let $0<|\phi|<\pi/2,$ and write
$$
I_{\mathbf{a}}(A)=u(A)+i v(A), \quad I_{\mathbf{b}}(A)=U(A)+i V(A).
$$
Note that $A$ solves (BE)$_{\phi}$ if and only if
$$
\im e^{i\phi}I_{\mathbf{a}}(A)=u(A)\sin\phi +v(A)\cos \phi=0, \quad
\im e^{i\phi}I_{\mathbf{b}}(A)=U(A)\sin\phi +V(A)\cos \phi=0,
$$
that is,
\begin{equation}\label{8.1}
u(A)\tan \phi +v(A)=0, \quad  U(A)\tan \phi +V(A)=0.
\end{equation}
Then, by the Cauchy-Riemann equations the Jacobian for \eqref{8.1} with
$A=x+iy$ is written as
\begin{align}\label{8.2}
\det J(\phi,A) &= \det \begin{pmatrix} u_x \tan\phi +v_x & u_y \tan\phi +v_y \\
U_x \tan\phi +V_x & U_y \tan\phi +V_y  \end{pmatrix}
\\
\notag
&= (1+\tan^2\phi) (v_xV_y -v_y V_x)
\\
\notag
&= -\frac 14(1+\tan^2\phi) |\Omega_{\mathbf{a}}(A)|^2 \im \frac
{\Omega_{\mathbf{b}}(A)}{\Omega_{\mathbf{a}}(A)},
\end{align}
where $\Omega_{\mathbf{a}}(A)$ and $\Omega_{\mathbf{b}}(A)$ are periods of the
elliptic curve $w(A,z).$
For $0<|\phi|<\pi/2$ the derivatives of \eqref{8.1} with $t=\tan\phi$ 
are written as
$$
J(\phi, A)\begin{pmatrix}  x'(t) \\ y'(t) \end{pmatrix} +
\begin{pmatrix} u(A) \\ U(A)  \end{pmatrix} \equiv \mathbf{o}. 
$$
Then we have
\begin{equation}\label{8.3}
\text{$(x'(t),y'(t))\not=(0,0)$ and $ (d/d\phi)A=(x'(t)+i y'(t))\cos^{-2}\phi
\not=0 $} 
\end{equation}
for $0<|\phi|<\pi/2.$
\begin{prop}\label{prop8.8}
Suppose that, for some $\phi_0$ such that $0<|\phi_0|<\pi/2$, $A_{\phi_0}$
solves $(\mathrm{BE})_{\phi_0}$. Then there exists a trajectory $T_0:$ 
$A=\chi(\phi_0,\phi)$ for $0<|\phi| <\pi/2$ with the properties$:$
\par
$(1)$ $\chi(\phi_0,\phi_0)=A_{\phi_0};$ 
\par
$(2)$ for each $\phi,$ $A=\chi(\phi_0,\phi)$ solves $(\mathrm{BE})_{\phi};$
\par
$(3)$ $\chi(\phi_0,\phi)$ is smooth for $0<|\phi|<\pi/2.$
\end{prop}
\begin{proof}
Since the Jacobian \eqref{8.2} satisfies $\det J(\phi_0, A_{\phi_0}) \in
\mathbb{R}\setminus \{0\},$ there exists a local trajectory 
$A=\chi_{\mathrm{loc}}(\phi_0,\phi)$ having the properties $(1)$, $(2)$ and 
$(3)$ above for $|\phi-\phi_0|<\delta$, where $\delta$ is sufficiently small. 
Since \eqref{8.2} does not vanish 
for $0<|\phi|<\pi/2,$ $\chi_{\mathrm{loc}}(\phi_0,
\phi)$ may be extended to the interval $0<|\phi|<\pi/2.$
\end{proof}
\begin{prop}\label{prop8.9}
The trajectory $T_0:$ $A=\chi(\phi_0,\phi)$ given above
may be extended to $|\phi|\le \pi/2$ in such a way that $\chi(\phi_0,\phi)$ is 
continuous in $\phi$ and that $\chi(\phi_0,0)=A_0= 2,$ 
$\chi(\phi_0,\pm \pi/2)=A_{\pm \pi/2}= - 2.$
\end{prop}
\begin{proof}
By Proposition \ref{prop8.7} the trajectory $T_0$ for $0<|\phi|<\pi/2$ is
bounded. To show that $\chi(\phi_0,\phi)\to A_0$ as $\phi \to 0$, suppose
to the contrary that there exists a sequence $\{\phi_{\nu}\}$ such that
$\phi_{\nu}\to 0$ and that $\chi(\phi_0,\phi_{\nu})$ does not converge
to $A_0.$ By the boundedness of $T_0$ there exists a subsequence 
$\{ \phi_{\nu(n)} \}$ such that
$\chi(\phi_0,\phi_{\nu(n)}) \to A'_0$ for some $A'_0\not=A_0.$ Then we have
$\im I_{\mathbf{a}}(A'_0)=\im I_{\mathbf{b}}(A'_0)=0,$ which contradicts
the uniqueness of a solution of $(\mathrm{BE})_0.$ Hence $\chi(\phi_0,\phi)$
is extended to $\phi=0$ and is continuous in a neighbourhood of $\phi=0.$
By Proposition \ref{prop8.5}, it is possible to extend $\chi(\phi_0,\phi)$ 
to $\phi=\pm\pi/2$ by the same argument. 
\end{proof}
\begin{lem}\label{lem8.10}
$h'(A)=4\pi i I_{\mathbf{a}}(A)^{-2}.$
\end{lem}
\begin{proof}
From $I'_{\mathbf{a},\, \mathbf{b}}(A)=-\Omega_{\mathbf{a},\, \mathbf{b}}/2$
and Lemma \ref{lem6.7}, the conclusion follows.
\end{proof}
\begin{cor}\label{cor8.11}
If $I_{\mathbf{a}}(A)\not=0, \infty$, then $h(A)$ is conformal around $A$.
\end{cor}
\par
By Example \ref{exa8.1}, $h(A)$ is conformal at $A_0= 2$ and 
$h(A_0)=0.$ Then, by Lemma \ref{lem8.10},
$$
h(A)=h'(A_0)(A-A_0)+o(A-A_0)=\frac{\pi i}{16}(A-A_0)+o(A-A_0)
$$
around $A=A_0.$ 
By Proposition \ref{prop8.6}, for a
sufficiently small $\varepsilon >0,$ the inverse image of $(-\varepsilon,0) 
\cup (0,\varepsilon)$ under $h(A)$ is a trajectory $T_{0-} \cup T_{0+}$:
$A= \chi^{\pm}_0(\phi)$ solving (BE)$_{\phi}$, and is expressed as 
\begin{equation}\label{8.4}
\chi^{\pm}_0(\phi) = A_0 +\gamma_0(\phi) i + o(\gamma_0(\phi)), 
\end{equation}
near $\phi=0$, where $\gamma_0(\phi) \in\mathbb{R}$ is continuous in $\phi$
and $\gamma_0(0)=0.$
\par
The fact above implies that there exists a local trajectory close to $A_0$ 
solving (BE)$_{\phi}$ for $\phi$ near $0.$ From this with Proposition 
\ref{prop8.8}, a trajectory
for $|\phi|\le \pi/2$ as in Proposition \ref{prop8.9} may be obtained. 
Furthermore, if two trajectories $\chi_1(\phi)$ and $\chi_2(\phi)$ solving
(BE)$_{\phi}$ satisfy $\chi_1(\phi_0)=\chi_2(\phi_0)$ for some $\phi_0$ such
that $0<|\phi_0|<\pi/2,$ then \eqref{8.2} or the conformality of $h(A)$ at
$A=A_0$ implies $\chi_1(\phi)\equiv \chi_2(\phi).$ Thus we have the following.
\begin{prop}\label{prop8.12}
There exists a trajectory $A=A_{\phi}$ for $|\phi| \le \pi/2$ with the 
properties$:$
\par
$(1)$ for each $\phi$, $A_{\phi}$ is a unique solution of $(\mathrm{BE})_{\phi};$
\par
$(2)$ $A_{\phi}$ is smooth in $\phi$ for $0<|\phi|<\pi/2$ and continuous in
$\phi$ for $|\phi|\le \pi/2.$
\end{prop} 
\par
For any cycle $\mathbf{c}$, it is easy to see that
\begin{align*}
&e^{i\phi}\int_{\mathbf{c}}\frac 1{\lambda^2}w(A_{\phi},\lambda)d\lambda
=  e^{i(\phi \mp \pi )}
\int_{e^{\mp \pi i}\mathbf{c}} \frac 1{\zeta^2} w(A_{\phi},\zeta) d\zeta,
\\
&e^{i\phi}\int_{\mathbf{c}}\frac 1{\lambda^2}w(A_{\phi},\lambda)d\lambda
=  e^{i(\phi \mp \pi/2 )}
\int_{e^{\mp \pi i/2}\mathbf{c}} \frac 1{\zeta^2} w(-A_{\phi},\zeta) d\zeta,
\end{align*}
which yields the following.
\begin{prop}\label{prop8.13}
The function $A_{\phi}$ with $|\phi|\le \pi/2$ may be extended to 
$\phi\in \mathbb{R}$ by setting
$A_{\phi\pm \pi}=A_{\phi}$, and solves the Boutroux equations $(\mathrm{BE}
)_{\phi}$.
Furthermore $A_{\phi \pm \pi/2}=-A_{\phi},$ $A_{-\phi}=\overline{A_{\phi}}.$ 
\end{prop}
Let us examine the properties of $A_{\phi}$ in more detail. Note that the 
trajectory $A=A_{\phi}=x+iy$ with $|\phi|<\pi /2$ satisfies  
$h(A_{\phi})\in\mathbb{R}$. 
Then by Lemma \ref{lem8.10} and \eqref{8.3},
$$
\frac{d}{dt} h(A_{\phi}) =(x'(t)+iy'(t))4\pi i I_{\mathbf{a}}(A_{\phi})^{-2}
\in \mathbb{R} \setminus \{0\}
$$ 
with $t=\tan\phi$ for $0<|\phi|<\pi/2.$ Setting $I_{\mathbf{a}}(A_{\phi})^{-1}
=P+i Q,$ we have
$$
\frac 1{4\pi}\im \frac d{dt} h(A_{\phi}) =x'(t)(P^2-Q^2) -2y'(t) PQ=0.
$$
If $x'(t_0)=0$ for some $t_0 =\tan(\phi_0)\not=0, \pm \infty$, then $PQ
= \re I_{\mathbf{a}}(A_{\phi})^{-1}\cdot \im I_{\mathbf{a}}(A_{\phi})^{-1} =0,$
and hence $I_{\mathbf{a}}(A_{\phi_0}) \in i\mathbb{R} \setminus \{0\}$ or
$\mathbb{R} \setminus \{0\}.$ This is impossible for $0<|\phi|<\pi/2,$ which
implies $x'(t)\not=0$ for $0<|\phi|<\pi/2.$ Since $A_{\pm \pi/2}= -2$,
we have $x'(t)<0$ for $0<\phi <\pi/2$ and $x'(t)>0$ for
$-\pi/2 <\phi <0.$ It is easy to see that $x'(0)=x'(\tan(\pm \pi/2))=0.$
If $y'(t_0)=0$ for some $t_0$ with $0<|\phi_0|<\pi/2$, then
$P^2-Q^2=0,$ i.e. $I_{\mathbf{a}}(A_{\phi_0})^{-1}=P(1 \pm i),$ 
implying $\phi_0=\pm \pi/4.$ Note that $PQ<0$, $|P|>|Q|$ for $-\pi/4 <\phi<0$ 
and that $PQ>0$, $|P|>|Q|$ for $0<\phi<\pi/4.$ It follows that $y'(t)<0$ for 
$0<|\phi|<\pi/4.$
\begin{prop}\label{prop8.14}
The trajectory $A_{\phi}=x(t)+iy(t)$ with $t=\tan\phi$ has the properties$:$
\par
$(1)$ $x'(t)>0$ for $-\pi/2<\phi<0,$ and $x'(t)<0$ for $0<\phi<\pi/2;$
\par
$(2)$ $x'(0)=x'(\tan(\pm \pi/2))=0;$
\par
$(3)$ $y'(t)<0$ for $0<|\phi|<\pi/4$ and $y'(\tan(\pm \pi/4))=0.$
\end{prop} 
Thus we have
\begin{prop}\label{prop8.15}
For $\phi \in \mathbb{R}$ there exists a trajectory $A=A_{\phi}$ with the
properties$:$
\par
$(1)$ for each $\phi$, $A_{\phi}$ is a unique solution of $(\mathrm{BE})_{\phi};$
\par
$(2)$ $A_{\phi \pm \pi /2}=- A_{\phi},$ $A_{\phi+\pi}=A_{\phi},$
$A_{-\phi}=\overline{A_{\phi}};$
\par
$(3)$ $A_0= 2,$ $A_{\pm \pi/2}=- 2;$
\par
$(4)$ $A_{\phi}$ is continuous in $\phi\in \mathbb{R}$, and smooth in $\phi \in
\mathbb{R}\setminus \{ m\pi/2 \,|\, m\in \mathbb{Z} \}.$
\end{prop}
The trajectory of $A_{\phi}$ is roughly drawn in Figure \ref{trajectory}.
{\small
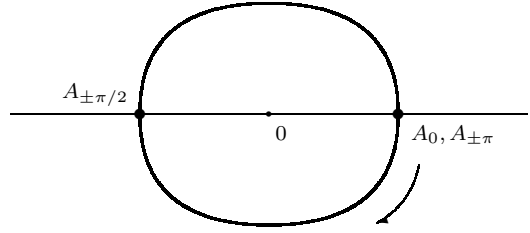
\begin{figure}[htb]
\begin{center}
\unitlength=0.85mm
\begin{picture}(80,50)(-40,-25)
\put(0,0){\circle*{1}}
\put(-20,0){\circle*{1.5}}
\put(20,0){\circle*{1.5}}

\put(-40,0){\line(1,0){80}}

\qbezier (23.2,-8) (22,-14) (16.7, -17)
\put(16.7,-17){\vector(-3,-1){0}}

\thicklines
\qbezier (20,0) (20,17.3) (0, 17.3)
\qbezier (20,0) (20,-17.3) (0, -17.3)
\qbezier (-20,0) (-20,17.3) (0, 17.3)
\qbezier (-20,0) (-20,-17.3) (0, -17.3)

{\tiny
\put(-32,2.5){\makebox{$A_{\pm \pi/2}$}}
\put(22,-4){\makebox{$A_0, A_{\pm \pi}$}}
\put(1,-4){\makebox{$0$}}
}

\end{picture}

\end{center}
\caption{Trajectory of $A_{\phi}$ for $|\phi|\le \pi$}
\label{trajectory}
\end{figure}
}
\par
By Proposition \ref{prop8.14}, when $|\phi|$ is sufficiently small, the location
of the turning points may be examined. Small variance
of $A_{\phi}$ around $\phi=0$ is given by $A_{\phi}=2+\delta_{\phi}$ with
$\delta_{\phi}$ having the properties:
(1) $\delta_{\phi}\to 0$ as $\phi\to 0;$ (2) $\re\delta_{\phi} \le 0$;
(3) $\im \delta_{\phi} \ge 0$ if $\phi\le 0$ and $\im \delta_{\phi}\le 0$ 
if $\phi \ge 0$.
Let the turning points $\lambda_1$, $\lambda_2$ with $\re \lambda_1 \le \re 
\lambda_2$ be such that $\lambda_1=\sqrt{1-\rho_1},$ $\lambda_2=\sqrt{1+\rho_2}$.
Note that $\lambda_1^2\lambda_2^2=1$ and $\lambda_1^2+\lambda_2^2=2
+\delta_{\phi}$, which yield $\rho_2-\rho_1=\delta_{\phi},$ $\rho_1\rho_2=
\delta_{\phi}.$ Since $\delta_{\phi}$ is small, $\rho_1 \sim \delta_{\phi}^{1/2}
-\tfrac 12 \delta_{\phi},$ $\rho_2 \sim \delta_{\phi}^{1/2} +\tfrac 12 \delta
_{\phi}.$ 
Thus we have the following.
\begin{prop}\label{prop8.16}
If $|\phi|$ is sufficiently small, the turning points $\lambda_1$  
and $\lambda_2$ are represented as  
$$
\lambda_1= 1 -  \varepsilon_{\phi} e^{i\theta_{\phi}}
 +O(\varepsilon_{\phi}^2), 
\qquad
\lambda_2= 1 +  \varepsilon_{\phi} e^{i\theta_{\phi}}
 +O(\varepsilon_{\phi}^2), 
$$
where $\varepsilon_{\phi}$ and $\theta_{\phi}$ fulfil 
\par
$(1)$ $\varepsilon_{\phi}>0$ and $\varepsilon_{\phi}\to 0$ as $\phi\to 0;$ 
and 
\par
$(2)$ $\theta_{\phi} \to \pi/4$ as $\phi\to 0$ with $\phi<0$, and 
$\theta_{\phi} \to -\pi/4$ as $\phi\to 0$ with $\phi>0$.  
\end{prop}  
\vskip0.2cm
{\it Acknowledgement.} The author is grateful to the referee for valuable
comments and suggestions.

\small{

}

\vskip0.3cm
\small{
\par
Department of Mathematics, 
\endgraf
Keio University, 
\endgraf
3-14-1, Hiyoshi, Kohoku-ku, Yokohama 223-8522 Japan
\endgraf
{\tt shimomur@math.keio.ac.jp}
}
\end{document}